\documentclass[12pt,leqno,twoside]{article}
\usepackage{amssymb}
\usepackage{amsmath}
\usepackage{amsthm,amscd}
\usepackage[utf8]{inputenc} 
\usepackage[T1]{fontenc}    

\usepackage{a4,indentfirst,latexsym}
\usepackage{graphics}
\usepackage{mathrsfs}
\usepackage{cite,enumitem,graphicx}
\usepackage[colorlinks=true,urlcolor=black,
citecolor=black,linkcolor=black,linktocpage,pdfpagelabels,
bookmarksnumbered,bookmarksopen]{hyperref}
\usepackage[colorinlistoftodos]{todonotes}
\usepackage{color}

\def\Xint#1{\mathchoice {\XXint\displaystyle\textstyle{#1}} {\XXint\textstyle\scriptstyle{#1}} {\XXint\scriptstyle\scriptscriptstyle{#1}} {\XXint\scriptscriptstyle\scriptscriptstyle{#1}}
	\!\int} \def\XXint#1#2#3{{\setbox0=\hbox{$#1{#2#3}{\int}$ } \vcenter{\hbox{$#2#3$ }}\kern-.6\wd0}}  \def\dashint{\Xint-}

\newtheorem{thm}{Theorem}[section]
\newtheorem{defi}{Definition}[section]
\newtheorem{prop}[thm]{Proposition}
\newtheorem{lem}[thm]{Lemma}
\newtheorem{cor}[thm]{Corollary}
\newtheorem{rem}[thm]{Remark}

\newtheorem{claim}[thm]{\it \hspace{1.4em} Claim}

\newenvironment{altproof}[1]
{\noindent
	{\bf Proof of {#1}}.}
{\nopagebreak\mbox{}\hfill $\Box$\par\addvspace{0.5cm}}

\newcommand{\oH}{\overline{
		\mathrm{H}}^1(\Sigma)}

\newcommand\lan{\langle}
\newcommand\ran{\rangle}

\newcommand{\rg}{\mathrm{g}}
\newcommand{\rb}{\mathrm{b}}

\def\N{\mathbb{N}}

\def\Q{\mathbb{Q}}
\def\R{\mathbb{R}}

\def\i{\mathrm{i}}

\def\dim{\mathrm{dim}}

\def\Dc{\mathscr{D}}

\newcommand{\bW}{{\mathbf W}}
\newcommand{\bphi}{\boldsymbol{\phi}}
\newcommand{\bh}{{\mathbf h}}
\newcommand{\bu}{{\mathbf u}}

\newcommand{\rL}{{\mathrm L}}
\newcommand{\rH}{{\mathrm H}}

\newcommand{\cB}{{\mathcal B}}
\newcommand{\cC}{{\mathcal C}}
\newcommand{\cD}{{\mathcal D}}

\newcommand{\cF}{{\mathcal F}}

\newcommand{\cH}{{\mathcal H}}

\newcommand{\cL}{{\mathcal L}}

\newcommand{\cN}{{\mathcal N}}

\newcommand{\cR}{{\mathcal R}}
\newcommand{\cS}{{\mathcal S}}
\newcommand{\cT}{{\mathcal T}}

\newcommand{\cV}{{\mathcal V}}

\newcommand{\fJ}{{\mathfrak J}}

\newcommand{\De}{\Delta}

\newcommand{\Si}{\Sigma}

\newcommand{\intsigma}{{\mathring\Si}}

\newcommand{\beq[1]}{\begin{equation}\label{eq:#1}}
	\newcommand{\eeq}{\end{equation}}

\numberwithin{equation}{section}

\DeclareMathOperator{\dist}{dist}

\begin{document}
	\title{Partial Blow-up Phenomena in the $SU(3)$ Toda System  on Riemann Surfaces}
	\author{Zhengni Hu \and Mohameden Ahmedou \and Thomas Bartsch }
	\maketitle
	
	
	
	
	
	\noindent{\bf Abstract:} 
	This work studies the partial blow-up phenomena for the $SU(3)$ Toda system on compact Riemann surfaces with smooth boundary. We consider the following coupled Liouville system with Neumann boundary conditions: 
	\begin{equation*}
		\left\{\begin{array}{ll}
			-\Delta_g u_1  = 2\rho_1\left( \frac{V_1 e^{u_1}}{\int_{\Sigma} V_1 e^{u_1} \, dv_g} - \frac 1 {|\Sigma|_g}\right) - \rho_2\left( \frac{V_2 e^{u_2}}{\int_{\Sigma} V_2 e^{u_2} \, dv_g} - \frac{1}{|\Sigma|_g}\right) & \text{in} \, \intsigma\\
			-\Delta_g u_2 = 2\rho_2\left( \frac{V_2 e^{u_2}}{\int_{\Sigma} V_2 e^{u_2} \, dv_g} - \frac{1}{|\Sigma|_g}\right) - \rho_1\left( \frac{V_1 e^{u_1}}{\int_{\Sigma} V_1 e^{u_1} \, dv_g} - \frac{1}{|\Sigma|_g}\right) & \text{in} \, \intsigma\\
			\partial_{\nu_g} u_1 = \partial_{\nu_g} u_2 = 0 & \text{on} \, \partial \Sigma
		\end{array}\right. ,
	\end{equation*}
	where $(\Sigma, g)$ is a compact Riemann surface with the interior $\intsigma$ and smooth boundary $\partial\Sigma$, 
	$\rho_i$ is a non-negative parameter and $V_i$ is a smooth positive function for $i=1,2$.

We construct a family of blow-up solutions via the Lyapunov-Schmidt reduction and variational methods, wherein one component remains uniformly bounded from above, while the other exhibits partial blow-ups at a prescribed number of points, both in the interior and on the boundary.  This construction is based on the existence of a non-degeneracy solution of a so-called shadow system. Moreover,   we establish the existence of partial blow-up solutions in three cases: (i) for any $\rho_2>0$ sufficiently small; (ii) 
	for generic $V_1, V_2$ and any $\rho_2\in (0,2\pi)$; (iii) for generic $V_1, V_2$,  the Euler characteristic $\chi(\Sigma)<1$ and any $\rho_2\in (2\pi,+\infty)\setminus 2\pi \N_+$.
	\newpage
	\tableofcontents
	\section{Introduction}
	Let $(\Sigma,g)$ be a compact  Riemann surface with the interior $\intsigma$  and smooth boundary $\partial \Sigma$. 
	This paper studies the following $SU(3)$ Toda system:
	\begin{equation}\label{eq:toda}
		\left\{	\begin{array}{ll}
			-\Delta_g u_1  = 2\rho_1\left( \frac{V_1 e^{u_1}}{\int_{\Sigma} V_1 e^{u_1} \, dv_g} - \frac 1 {|\Sigma|_g}\right) - \rho_2\left( \frac{V_2 e^{u_2}}{\int_{\Sigma} V_2 e^{u_2} \, dv_g} - \frac{1}{|\Sigma|_g}\right) & \text{in} \, \intsigma\\
			-\Delta_g u_2 = 2\rho_2\left( \frac{V_2 e^{u_2}}{\int_{\Sigma} V_2 e^{u_2} \, dv_g} - \frac{1}{|\Sigma|_g}\right) - \rho_1\left( \frac{V_1 e^{u_1}}{\int_{\Sigma} V_1 e^{u_1} \, dv_g} - \frac{1}{|\Sigma|_g}\right) & \text{in} \, \intsigma\\
			\partial_{\nu_g} u_1 = \partial_{\nu_g} u_2 = 0 & \text{on} \, \partial \Sigma
		\end{array}\right. ,
	\end{equation}
	where $\Delta_g$ is the Laplace-Beltrami operator, $dv_g$ is the area element of $(\Sigma,g)$,  $\nu_g$ is the unite outwards normal of $\partial \Sigma$,   $\rho_i$ is a non-negative parameter and $V_i:\Sigma\rightarrow \R$ is a smooth positive function for any $i=1,2$. For simplicity, we normalize the area of $\Sigma$, i.e., $|\Sigma|_g =1$. 
	
	Toda system arises in many fields both in geometry and physics. In geometry, it is related to holomorphic curves of $\Sigma$ in $\mathbb{C} \mathbb{P}^N$, flat $SU(N+1)$ connection, complete integrability and
	harmonic sequences (see~\cite{guest1997harmonic,bolton1988conformal,doliwa1997holomorphic,chern1987harmonic,bolton1997geometrical}, for instance). Particularly, for $\Sigma=S^2$, the solution space of the $SU(3)$ Toda system is identical to the space of holomorphic curves of $S^2$ into $\mathbb{C} \mathbb{P}^2$ (see~\cite{lin2012classification}).  In physics, the system is one of the limiting equations of non-abelian Chern–Simons gauge field theory (see \cite{dunne1991self,dunne1995self,nolasco1999double,nolasco2000vortex,yang1999relativistic,yang2001solitons} and references therein).
	
	Toda system has been widely studied both on a bounded domain in $\R^2$ and a closed surface. For the existence results we refer to \cite{jost_lin_wang2006,battaglia_jevnikar_ruiz2015, Malchiodi2007SomeER,Malchiodi2010NewIM,jevnikar_kallel_Malchiodi2015} for non-critical parameters  and \cite{li2005solutions, zhu2011solutions} for critical parameters.
	
	Analog to mean field equations, when parameter $\rho=(\rho_1,\rho_2)$ approaches critical values set $\Lambda=\left(2\pi \N_+\times[0,\infty)\right)\cup\left( [0,\infty)\times 2\pi\N_+\right)$ where $\N_+$ is the set of positive integers, the blow-up phenomenon may occur. We will focus on the construction of blow-up solutions in different blow-up scenarios. 
	We define the blow-up set for $u^n_i$ as 
	\[ \mathcal{S}_i:=\left\{x\in \Sigma: \exists x_n\rightarrow x \text{ such that } u^n_i(x_n)\rightarrow +\infty \right\}, \]
	$\cS=\cS_1\cup\cS_2$ and the local limit masses as 
	\[ \sigma_i(x)= \lim_{{r\rightarrow 0}} \lim_{{n\rightarrow +\infty}}\int_{U_r(x)} \frac{ \rho^n_i  V_i e^{u^n_i}}{\int_{\Sigma}  V_i e^{u^n_i} dv_g}dv_g\quad \text{ for } i=1,2. \]
	For $\Sigma$ compact Riemann surfaces with smooth boundary, one can prove that 
	if $\cS\neq\emptyset$, for any $x\in\cS$, $(\sigma_1(x),\sigma_2(x) )$ can only take the following five possible values (refer to \cite{jost_lin_wang2006} for a bounded domain in $\R^2$) 
	\[\left(\frac 1 2 \varrho(x),0\right), \left(0, \frac 1 2 \varrho(x)\right), \left(\frac 1 2 \varrho(x), \varrho(x)\right), \left(\varrho(x), \frac 1 2 \varrho(x)\right) \text{ and } (\varrho(x),\varrho(x)), \]
	where 	$\varrho(x)= 8\pi$ if $x\in\intsigma$ and $4\pi$ if $x\in \partial\Sigma$. 
	Based on the values of $(\sigma_1(x), \sigma_2(x))$ at blow-up point $x$, we have the following three scenarios: 
	\begin{itemize}
		\item[i).] Partial blow-up: $\left(\frac 1 2 \varrho(x),0\right), \left(0, \frac 1 2 \varrho(x)\right)$. In this case, one component is bounded from above and the other blows up;
		\item [ii).] Asymmetric blow-up:$ \left(\frac 1 2 \varrho(x), \varrho(x)\right), \left(\varrho(x), \frac 1 2 \varrho(x)\right)$. In this case, both components blow up at the same point $x$ but with different rates;
		\item [iii).] Full blow-up: $(\varrho(x),\varrho(x))$. In this case, both components blow up at the same point $x$ but with the same rates. 
	\end{itemize}
	
	The study of blow-up solutions for the Toda system remains a relatively underdeveloped area within the literature, with a handful of results addressing Dirichlet boundary conditions in bounded domains.
	
	For partial blow-up scenarios, D'Aprile, Pistoia, and Ruiz employ finite dimensional reduction and variational methods to successfully construct partial blow-up solutions, either within simply connected domains or when one of the parameters is sufficiently small in \cite{DAprile_Pistoia_Ruiz2015}. Independently, \cite{Lee2018degree} constructed one-point blow-up solutions exhibiting partial blow-up for $\min\left\{\rho_1,\rho_2\right\}<8\pi$ on closed Riemann surfaces by the method of degree counting. 
	Concerning asymmetric blow-up, Ao and Wang in \cite{ao2014new} introduced a family of blow-up solutions for the Toda system with Dirichlet boundary conditions on a unit ball centered at the origin, which exhibits a single blow-up point at the center. D'Aprile, Pistoia, and Ruiz extended this result to planar domains that possess a so-called ``$k$-symmetric" property for $k>2$, as $\rho_1\rightarrow 8\pi$, while keeping $\rho_2$ fixed within the interval $(4\pi, 8\pi)$ in  \cite{D'aprile_asymmetric_2016}. Additionally, Musso, Pistoia, and Wei constructed a family of blow-up solutions for the $SU(N+1)$ Toda system, applying a similar approach for $N\geq 2$ in \cite{musso_new_2016}.
	
	In the scenario of full blow-ups, the special case where $u_1=u_2$ and $\rho_1=\rho_2$ reduces the system to a mean field equation. For this equation, blow-up solutions have been constructed in both bounded domains and closed Riemann surfaces (see\cite{Esposito2005, Bartolucci2020, Esposito2014singular}, for instance). The general case, however, remains an open problem. Lin, Wei, and Zhao have examined blow-up solutions in the full blow-up scenario, providing five necessary conditions for the existence of such solutions, indicating the increased difficulty of constructing full blow-up solutions in \cite{lin2012sharp}. 
	
	To our knowledge, there are no published results on blow-up solutions for the Toda system with Neumann boundary conditions on surfaces with boundary. Our study aims to undertake an endeavor to construct blow-up solutions,  particularly focusing on scenarios of partial blow-up, on compact Riemann surfaces with boundary.

	The system \eqref{eq:toda} is variational and has a corresponding Euler-Lagrange functional 
	\[ J_{\rho}(u)= \int_{\Sigma} Q(u,u)\, dv_g + \sum_{i=1}^2 \rho_i\left( \int_{\Sigma} u_i\,  dv_g - \log \int_{\Sigma} V_i e^{u_i} \, dv_g\right), \]
	where $u=(u_1,u_2), \rho=(\rho_1,\rho_2), $ the quadratic forms $$ Q(v,w)=\frac 1 3 \lan \nabla v_1, \nabla w_1\ran_g +\frac 1 3  \lan \nabla v_2, \nabla w_2\ran_g+\frac 1 6  \lan \nabla v_2, \nabla w_1\ran_g+\frac 1 6\lan \nabla v_1, \nabla w_2\ran_g. $$ 
	The weak solutions of \eqref{eq:toda} correspond to the critical points of the functional $J_{\rho}$ on $H^1(\Sigma)\times H^1(\Sigma)$. 
	Let $$\oH:=\left\{h\in H^1(\Sigma): \int_{\Sigma} h\, dv_g=0\right\} \text{ and }
	\cH=\oH\times \oH. $$
	We equip the space $\oH$ with the norm $\|h\|=\left( \int_{\Sigma} |\nabla h|_g^2 \, dv_g\right)^{\frac 1 2}$ for any $h\in\oH$ and the space $\cH$ with the norm $\|u\|=\|u_1\|+\|u_2\|$ for any $u=(u_1,u_2)\in \cH.$
	Nevertheless, we focus on solutions in $\cH$, a subspace of $H^1(\Sigma) \times H^1(\Sigma)$ where $ \int_{\Sigma} u_i \, dv_g = 0 $ for  $i = 1, 2$. Both the problem~\eqref{eq:toda} and the energy functional $J_{\rho}$ exhibit invariance under the addition of a constant; thus, our assumption does not impose a significant restriction. 
	
	Given integers $ m \geq k \geq 0 $, consider a sequence $ \rho^n := (\rho^n_1, \rho^n_2) $ where $ \rho^n_2 = \rho_2$ fixed and $ \rho^n_1 $ approaches $ 2\pi(m+k) $. We aim to construct a sequence of solutions $ u^n = (u^n_1, u^n_2) $ for which $ u^n_1 $ exhibits partial blow-ups while $ u^n_2 $ remains bounded from above. To facilitate this, we introduce the following configuration set:
	\begin{eqnarray*}
		\Xi_{k,m}:=\intsigma^k\times(\partial\Sigma)^{m-k}\setminus \De,
	\end{eqnarray*}
	where $	\De:=\left\{\xi=(\xi_1,\cdots,\xi_m): \xi_i=\xi_j\text{ for some } i=j \right\}. $
	For any $\varepsilon>0$, we define 
	\begin{eqnarray*}
		&	\Xi_{k,m}^{\varepsilon}:= \left\{\xi\in\Xi_{k,m}: d_g(\xi_i,\partial\Sigma)\geq \varepsilon \text{ for } i=1,\cdots,k;  d_g(\xi_i,\xi_j)\geq \varepsilon \text{ for } i\neq j \right\}.&
	\end{eqnarray*}
	We define the following function:
	\begin{equation}
		\label{eq:def_reduced_function}
		\Lambda_{k,m}: \Xi_{k,m}\rightarrow\R, \xi\mapsto \frac 1 2 I_{\xi}(z(\cdot,\xi)) -\frac 1 4 \cF_{k,m}(\xi),
	\end{equation}
	where $\cF_{k,m}$ is a $m$-vortices Kirchhoff-Routh type function (see~\cite{lin1941motion,Bartsch2017TheMP,Ahmedou-Bartsch-Fiernkranz:2023}), i.e. 
	\begin{eqnarray}
		\label{eq:F_km}\cF_{k,m}(\xi)= \sum_{j=1}^m \varrho^2(\xi_j)R^g(\xi_j)+ \sum_{j'\neq j} \varrho(\xi_j)\varrho(\xi_{j'})G^g(\xi_j,\xi_{j'})+ \sum_{j=1}^m 2 \varrho(\xi_j)\log V_1(\xi_j),
	\end{eqnarray}
	where $ G^g $ is the Green's function and $ R^g $ is the Robin's function (for details, refer to Section~\ref{sec:pre}). 
	To address the second component, we introduce solutions of the following singular mean-field equations: 
	\begin{equation}\label{eq:singular_mf}
		\left\{
		\begin{aligned}
			-\Delta_g z(x) &= 2\rho_2 \left( \frac{ \tilde{V}_2(x,\xi)e^{z(x)}}{\int_{\Sigma} \tilde{V}_2(\cdot,\xi)e^{z}\, dv_g }- 1 \right) & x \in \Sigma \\
			\partial_{\nu_g}	z(x) &= 0 & x \in \partial\Sigma
		\end{aligned}
		\right.,
	\end{equation}
	where $\xi \in\Xi_{k,m}$ and
	\begin{equation*}
		\tilde{V}_2(x, \xi) = V_2(x) e^{- \sum_{i=1}^{m}\frac 1 2\varrho(\xi_i) G^g(x,\xi_{i})},
	\end{equation*}
	where $G^g$ is the Green's function (for details, refer to Section~\ref{sec:pre}). 
	For $2\rho_2< 4\pi$, the energy function
	\[ I_{\xi}(z)=\frac 1 2 \int_{\Sigma} |\nabla z|^2_g \, dv_g -2\rho_2 \log \left( \int_{\Sigma} \tilde{V}_2(\cdot,\xi)e^{z} \, dv_g\right)\]
	is coercive by the Moser-Trudinger inequality. Standard variational analysis indicates that~\eqref{eq:singular_mf} admits a solution in $\oH.$ However, the non-degeneracy of the solutions is not guaranteed. To tackle this issue, additional conditions are needed.
	
	As in~\cite{lee_degree_2018}, for fixed $\rho_2\notin 2\pi \N$ and some $\alpha\in  (0,1)$, we consider the following  shadow system in the space
	$ C^{2,\alpha}(\Sigma)\times \Xi_{k,m}$ with $\int_{\Sigma} w\, dv_g =0$:
	\begin{equation}\label{eq:shadow}
		\left\{
		\begin{array}{ll}
			-\Delta_g w =2\rho_2 \left( \frac{V_2 e^{ w - \sum_{j=1}^m  \frac {\varrho(\xi_j)}2 G^g(x,\xi_j)}} {\int_{\Sigma} V_2 e^{ w - \sum_{j=1}^m  \frac {\varrho(\xi_j)}2 G^g(x,\xi_j)}\, dv_g  } - 1\right)& \quad \text{ in } \intsigma\\
			\partial_{\nu_g} w=0 & \quad\text{ on } \partial\Sigma\\
			\partial_{x_i} f_0(x)|_{x=\xi} = 0& \quad\text{ in } \Xi_{k,m}
		\end{array}
		\right.,
	\end{equation}
	where	\[
	\begin{array}{ll}
		f_0(x_1, x_2, \ldots, x_m)=& \cF_{k,m}(x)
		-\sum_{j=1}^m\varrho(x_j) w(x_j). 
	\end{array}
	\]
	
	\begin{itemize}
		\item [(C1)]\label{item:G}{\it There exists a non-degenerate solution $(w,\xi)\in C^{2,\alpha}(\Sigma)\times \Xi_{k,m}$ with $\int_{\Sigma} w\, dv_g=0$ of the shadow system~\eqref{eq:shadow} for positive potential function $(V_1, V_2)\in  C^{2,\alpha}(\Sigma, \R_+)\times C^{2,\alpha}(\Sigma,\R_+) $. }
	\end{itemize}
	The hypothesis states the existence of non-degenerate solutions for the shadow system. 
	\begin{thm}
		~\label{thm:main}
		Give integers $m\geq k\geq 0$ and $\rho_2\notin 2\pi \N$.  If the hypothesis \hyperref[item:G]{(C1)} is satisfied for  $(w,\xi)$, then
		there exist an open neighborhood of $\xi$ denoted by $\Dc (\subset \Xi_{k,m})$, a sequence $\xi^n = (\xi^n_1, \ldots, \xi^n_m) \in \Dc$ converging to  $\xi:=(\xi_1,\cdots,\xi_m)$, and a family of solutions $\bu^n = (u^n_1, u^n_2)$ of the Toda system~\eqref{eq:toda} corresponding to the parameter $\rho^n = (\rho^n_1, \rho_2) \rightarrow (2\pi(k+m), \rho_2)$. As $n \rightarrow +\infty$, we have
		\begin{equation}
			\label{eq:expansion_u1}
			2\rho
			^n_1\frac{ V_1e^{u^{n}_1}} { \int_{\Sigma} V_1 e^{u^{n}_1} dv_g} \rightarrow \sum_{j=1}^m \varrho(\xi_j)\delta_{\xi_j}, 
		\end{equation}
		which is convergent in sense of measure on $\Sigma$,
		and 
		\begin{eqnarray}
			\label{eq:expansion_u2} u^n_2&=&z(\cdot,\xi^n)-\frac 1 2\sum_{j=1}^m  \varrho(\xi_j) G^g(\cdot,\xi^n_j)+o(1),
		\end{eqnarray}
		which is convergent in $\oH$. 
	\end{thm}
	For sufficiently small values of the parameter $\rho_2\in (0,2\pi)$, we can deduce the uniqueness and non-degeneracy~\eqref{eq:singular_mf} by the implicit theorem (refer to Section~\ref{sec:pre}). On the other hand, $\cF_{k,m}(\xi)\rightarrow +\infty$ as $\xi\rightarrow \partial \Xi_{k,m}$, which implies that the global minimum point of  $\Lambda_{k,m}$ exits and lies in the interior of $\Xi_{k,m}$. So, for $\rho_2\in (0,2\pi)$, the hypothesis~\hyperref[item:G]{(C1)} holds true. It leads to the following corollary:
	\begin{cor}\label{cor:1}
		Given integers $m\geq k\geq 0$.  There exists an open subset $\Dc\subset\overline{\Dc}\subset \Xi_{k,m}$ and  $\rho_0\in (0,2\pi)$ sufficiently small such that for any fixed $\rho_2\in (0, \rho_0)$,   a family of solutions $\bu^n = (u^n_1, u^n_2)$ of the Toda system~\eqref{eq:toda} corresponding to the parameter $\rho^n = (\rho^n_1, \rho_2) \rightarrow (2\pi(k+m), \rho_2)$ can be constructed and  they blow up at $k$ points in the interior and $(m-k)$ points on the boundary.
		
		Moreover, there exists a sequence  $\xi^n\in \Xi_{k,m}$ converges to some $\xi\in \overline{\Dc}$ such that  \eqref{eq:expansion_u1} is valid for 
		$\rho^n_1$ and $u^n_1$, 
		\eqref{eq:expansion_u2} is valid for $\rho_2$ and $u^n_2$, and 
		\[ \inf_{\Xi_{k,m}} \Lambda_{k,m}= \Lambda_{k,m}(\xi). \]
	\end{cor}
	
	Applying a transversality theorem, the shadow system~\eqref{eq:shadow} is non-degenerate for generic positive function $(V_1, V_2)\in  C^{2,\alpha}(\Sigma, \R_+)\times C^{2,\alpha}(\Sigma,\R_+)$ (refer to Theorem~\ref{thm:residual}). We select the positive function
	$(V_1, V_2)$ such that ~\eqref{eq:shadow} is  non-degenerate for any solutions $(w,\xi)$. Using the method of continuity, we deduce the existence of a solution  $(w,\xi)$ of the shadow system~\eqref{eq:shadow}, leading to the formulation of our second corollary:
	\begin{cor}\label{cor:2}
		Suppose $m \geq k \geq 0$ are integers and $\alpha\in (0,1)$. For  $(V_1, V_2)$ in a dense subset of $C^{2,\alpha}(\Sigma, \R_+) \times C^{2,\alpha}(\Sigma, \R_+)$ if one of following conditions holds:\\ 
		\begin{itemize}
			\item 	[a)]  $\rho_2\in (0,2\pi)$,
			\item 
			[b)] $ \chi(\Sigma)= 2-2\rg+ \rb<1$ and $\rho_2\in (2\pi, +\infty)\setminus 2\pi \N_+$, where $\rg$ is the genus of the surface $\Sigma$ and $\rb$ is the number of connected components of the boundaries $\partial\Sigma$,
		\end{itemize}
	then the conclusions of Theorem~\ref{thm:main} 
		hold.
	\end{cor}

	\section{Preliminary}
	Throughout this paper, we use the terms ``sequence" and ``subsequence" interchangeably.  The constant denoted by $C$ in our deduction may assume different values across various equations or even within different lines of equations. 
	We also denote $B_r(y)=\{y\in\R^2: |y|<r\}$.
	For any $v\in L^1(\Sigma)$ and $\Omega\subseteq \Sigma$, we denote $$\dashint _{\Omega} v:= \frac 1 {|\Omega|_g} \int_{\Omega} v\, dv_g\text{ and } \bar{v}:= \dashint_{\Sigma} v.  $$ 
	\subsection{The isothermal coordinates and Green's functions} \label{sec:pre} 
	
	To construct the approximation solutions of problem~\eqref{eq:toda}, we firstly introduce a family of isothermal coordinates (see \cite{chern1955, Esposito2014singular,yang2021125440}, for instance). For any $\xi\in\intsigma$,  there exists an isothermal coordinate system $\left(U(\xi), y_{\xi}\right)$ such that $y_{\xi}$  maps an open  neighborhood $U(\xi)$ around $\xi$  onto an open disk $B^{\xi}$ with radius $2r_{\xi}$ in which $g=\sum_{i=1}^2 e^{\hat{\varphi}_\xi(y_{\xi}(x))}  \mathrm{d} x^i \otimes \mathrm{d} x^i$. Without loss of generality, we may assume that $y_{\xi}(\xi)=(0,0)$ and 
	$\overline{U(\xi)}\subset \intsigma$. 
	For $\xi\in \partial\Sigma$ there exists an isothermal coordinate system $\left(U(\xi), y_{\xi}\right)$ around $\xi$ such that the image of $y_{\xi}$ is a half disk $B^{\xi}:=\{ y=(y_1,y_2)\in \R^2: |y|<2r_{\xi}, y_2\geq 0\}$ with a radius $2r_{\xi}$ and   $y_{\xi}\left(U(\xi)\cap \partial \Sigma\right)= \{ y=(y_1,y_2)\in \R^2: |y|<2r_{\xi}, y_2=0 \}$ with  $g=\sum_{i=1}^2 e^{\hat{\varphi}_\xi(y_{\xi}(x))}  \mathrm{d} x^i \otimes \mathrm{d} x^i$. W.l.o.g., we can assume that  $y_{\xi}(\xi)=(0,0)$. 
	Here $\hat{\varphi}_\xi:B^\xi\to\R$ is related to $K$, the Gaussian curvature of $\Sigma$, by the equation
	\begin{equation}
		\label{eq:Gauss}
		-\Delta\hat{\varphi}_\xi(y) = 2K\big(y^{-1}_\xi(y)\big) e^{\hat{\varphi}_\xi(y)} \quad\text{for all } y\in B^\xi. 
	\end{equation}
	For $\xi\in \Sigma$ and $0<r\le 2r_\xi$ we set
	\[
	B_r^\xi := B^\xi \cap \{ y\in\R^2: |y|< r\}\quad \text{and}\quad U_{r}(\xi):=y_\xi^{-1}(B_{r}^{\xi}).
	\]
	
	Both $y_\xi$ and $\hat{\varphi}_\xi$ are assumed to depend smoothly on $\xi$ as in~\cite{Esposito2014singular}. Additionally, $\hat{\varphi}_\xi$ satisfies $\hat{\varphi}_\xi(0)=0$ and $\nabla\hat{\varphi}_\xi(0)=0$. 
	Specifically, as in~\cite{yang2021125440} the Neumann boundary conditions preserved by the isothermal coordinates as for any $\xi\in \partial\Sigma$ and  $x\in y_{\xi}^{-1}\left( B^{\xi} \cap \partial \mathbb{R}^{2}_+\right)$, we have
	\begin{equation}\label{eq:out_normal_derivatives}
		\left(y_{\xi}\right)_*(\nu_g(x))=\left. -e^{ -\frac{\hat{\varphi}_{\xi}(y)}2} \frac {\partial} { \partial y_2 }\right|_{	y=y_{\xi}(x)}.
	\end{equation}
	
	Let $\chi$ be a radial cut-off function in  $C^{\infty}(\mathbb{R}, [0,1])$ such that 
	\begin{equation}
		\label{eq:cut_off} \chi(s)=\left\{\begin{array}{ll}
			1,& \text{ if } |s|\leq 1\\
			0, &\text{ if } |s|\geq 2
		\end{array}\right. .
	\end{equation} 
	And for fixed $r_0<\frac  1 4 r_{\xi}$ we denote that  $\chi_{\xi}(x)= \chi( |y_{\xi}(x)|/r_0)$ and $\varphi_{\xi}(x)= \hat{\varphi}_{\xi} ( y_{\xi}(x))$. 
	For any $ \xi \in \Sigma $, we define the Green's function for~\eqref{eq:toda} by the following equations:
	\begin{equation}~\label{eq:green}
		\left\{\begin{array}{ll}
			-\Delta_g G^g(x,\xi)  = \delta_{\xi} - \frac{1}{|\Sigma|_g} & x\in \intsigma \\
			\partial_{ \nu_g } G^g(x,\xi) = 0 & x\in \partial \Sigma \\
			\int_{\Sigma} G^g(x,\xi)\, dv_g(x) = 0
		\end{array}\right.,
	\end{equation}
	where $ \delta_{\xi} $ is the Dirac mass on $ \Sigma $ concentrated at $ \xi $.
	
	Let the function
	\begin{equation*}
		\Gamma^g_{\xi}(x)=\Gamma^g(x,\xi)= \left\{ \begin{array}{ll}	\frac{1}{2\pi} \chi_{\xi}(x)\log\frac{1}{|y_{\xi}(x)|}& \text{ if }\xi\in \intsigma \\	{\frac{1}{\pi}}\chi_{\xi}(x)\log\frac{1}{|y_{\xi}(x)|}	& \text{ if } \xi\in\partial\Sigma \end{array}\right..
	\end{equation*}  
	Decomposing the Green's function  
	$ G^g(x,\xi)= \Gamma^g_{\xi}(x)+ H^g_{\xi}(x),$ we have the function $H^g_{\xi}(x):= H^g(x,\xi)$ solves the following equations: 
	\begin{equation}\label{eqR}
		\left\{\begin{array}{ccll}
			-\Delta_g H^g_\xi
			&=&\frac{4}{\varrho(\xi)} (\Delta_g \chi_{\xi}) \log \frac 1 {|y_{\xi}|}+ \frac{8}{\varrho(\xi)} \lan\nabla\chi_{\xi}, \nabla\log \frac 1 {|y_{\xi}|}\ran_g 	-\frac 1 {|\Sigma|_g}, & \text{ in }\intsigma\\
			\partial_{ \nu_g} H^g_\xi&=&- \frac{4}{\varrho(\xi)}(\partial_{ \nu_g} \chi_{\xi}) \log\frac{1}{|y_{\xi}|}-\frac{4}{\varrho(\xi)}\chi_{\xi}\partial_{ \nu_g}  \log\frac{1}{|y_{\xi}|}, &\text{ on }\partial \Sigma\\
			\int_{\Sigma} H^g_\xi \,dv_g&=& - \frac{4}{\varrho(\xi)}\int_{\Sigma}  \chi_{\xi}\log\frac{1}{|y_{\xi}|} dv_g &
		\end{array}\right..\end{equation}
	By the regularity of the elliptic equations (see\cite{Nardi2014}), there is a unique smooth solution  $H^g(x,\xi)$, which solves \eqref{eqR} in $C^{2,\alpha}(\Sigma)$. 
	$H^g(x,\xi)$ is the regular part of  $G^g(x,\xi)$ and 
	$R^g(\xi):=H^g(\xi,\xi)$ is so-called Robin's function on $\Sigma$. It is easy to check that $H^g(\xi,\xi)$ is independent of the choice of the cut-off function $\chi$ and the local chart.
	\subsection{The stable critical points set}
	To state the main results, we have to define a ``stable" critical point set. We use the following definition of stability of critical points (see~\cite{del_pino_singular_2005,Esposito2005,Li1997OnAS}),
	\begin{defi}
		Let $D \subset \intsigma^k \times (\partial\Sigma)^{m-k}\setminus \Delta$, where $k\leq m\in \mathbb{N}_+$ and $\Delta:=\{\xi=(\xi_1,\xi_2,...,\xi_m)\in \Sigma^m: \xi_i=\xi_j \text{ for some } i\neq j\}$. 
		{ And $F: D \rightarrow \mathbb{R}$ be a $C^{1}$-function and  $K$ be  {a compact subset of critical points} of $F$, i.e 
			\[ K\subset \subset \{ x\in D: \nabla F(x)=0\}. \] 
			A critical set $K$ is stable if for any closed neighborhood $U$ of $K$ in $ \intsigma^k \times(\partial\Sigma)^{m-k}\setminus \Delta$, there exists $\varepsilon>0$
			such that, if $G: D \rightarrow \mathbb{R}$ is a $C^{1}$-function with $\|F-G\|_{C^1(U)}<\varepsilon$, then G has at least one critical point in $U$.}
	\end{defi}
	\begin{rem}
		A compact subset of critical points of $F$ is stable if one of the following conditions holds: 
		\begin{itemize}
			\item[a).] $K$ is a strict local maximum set of $F$, i.e. for any $x, y \in K$, $F(x)=F(y)$  and  for some open neighborhood $U$ of $K$, $F(x)>F(y)$ for any $x \in K$ and $y \in U \backslash K$;
			\item[b).] $K$ is a strict local minimum set of $F$;
			\item[c).] $K$ is an isolated critical point set with a nontrivial local degree. 
		\end{itemize}
	\end{rem}
	
	\subsection{The properties of a shadow system}
	Next, we study some important properties of the shadow system~\eqref{eq:shadow}. 
	
	Let $C^{\alpha}_0(\Sigma)$ and $C_0^{s,\alpha}(\Sigma)$ be the spaces of H\"{o}lder continuous functions of class $C^{\alpha}$ and $C^{s,\alpha}$ with zero average over $\Sigma$ for any $s=1,2, \alpha \in (0,1)$, respectively.

	The shadow system comprises a singular mean field equation coupled with a balance condition. An intuitive approach is to first analyze the non-degeneracy of the singular mean field equation and then address the balance condition, specifically where $\nabla f_0=0.$ We introduce the following hypothesis, which was originally introduced by \cite{DAprile_Pistoia_Ruiz2015} to deal with Dirichlet boundary conditions, and have modified it here to be compatible with Neumann boundary value conditions:
	\begin{itemize}
		\item[(C)]\label{item:H}{\it For some  $\xi^0\in \Xi_{k,m}$, there exists   a non-degenerate solution of the problem~\eqref{eq:singular_mf}, i.e. there exists  a  solution of~\eqref{eq:singular_mf}, denoted by $z(\cdot, \xi^0)$,  and  the following linear problem:
			\begin{equation}\label{eq:linear_u_2}
				\left\{
				\begin{aligned}
					-\Delta \psi(x) &= 2\rho_2 \left( \frac{\tilde{V}_2(x, \xi^0)e^{z(x,\xi^0)} \psi(x) \, }{\int_{\Sigma} \tilde{V}_2(\cdot, \xi^0)e^{z(\cdot,\xi^0)} \, dv_g} - \frac{\tilde{V}_2(x, \xi^0)e^{z(x,\xi^0)} \int_{\Si
						} \tilde{V}_2(\cdot, \xi)e^{z(\cdot,\xi^0)}\psi \, dv_g}{\left(\int_{\Sigma} \tilde{V}_2(\cdot, \xi^0)e^{z(\cdot,\xi^0)} \, dv_g\right)^2}  \right) & \text{in } \intsigma , \\
					\partial_{\nu_g}\psi &= 0 & \text{on } \partial\Sigma
				\end{aligned}
				\right.
			\end{equation}
			admits only the trivial solution $\psi=0$.}  
	\end{itemize}

	\begin{lem}\label{lem:C_1_z_x_xi}
		Assuming \hyperref[item:H]{(C)}  holds for $(\xi^0, z(\cdot,\xi^0))$, for any $\alpha\in (0,1)$, there exists an open neighborhood $\Dc$ around $\xi^0$ such that the map $\xi\mapsto z(\cdot,\xi)$ from $\Dc$ into $C^{2,\alpha}_0(\Sigma)$ is $C^1$-differentiable.
	\end{lem}
	
	\begin{proof}
		We observe that for any $\alpha\in (0,1)$,
		\[\tilde{V}_2(x,\xi)= V_2(x)   e^{-\sum_{j=1}^m\frac  {\varrho(\xi_j)} 2 H^g(x,\xi_j)}\Pi_{j=1}^m e^{2\chi_j(x)\log |y_{\xi_j}(x)|}\in C^{\alpha}(\Sigma),\]
		uniformly for $\Dc\subset \subset\Xi_{k,m}.$
		Let 
		\[
		\Psi : C^{2,\alpha}_0(\Sigma) \times \Xi_{k,m}\rightarrow C_0^{\alpha}(\Sigma)\times \{0\},
		\]
		\[
		\Psi(z, \xi) =\begin{bmatrix}
			-\Delta_g  z(x) - 2\rho_2\left( \frac{\tilde{V}_2(x, \xi)e^{z}}{\int_{\Omega} \tilde{V}_2(\cdot, \xi)e^{z} \, dv_g}-1\right)\\
			\partial_{\nu_g} z
		\end{bmatrix}. 
		\]
		We observe that  $z$ is a solution of \eqref{eq:singular_mf} if and only if $\Psi(z, \xi) = 0$.
		By \hyperref[item:H]{(C)}, for fixed $\xi^0$ and let $z=z(\cdot,\xi^0)$ be the non-degenerate solution of~\eqref{eq:singular_mf}. Then we have $\Psi(z, \xi^0) = 0$ and for any $\psi\in C^{2,\alpha}_0(\Sigma)$
		\[
		D_z\Psi(z, \xi^0)(\psi) =\begin{bmatrix}
			-\Delta_g \psi - K(\psi)\\
			\partial_{\nu} \psi
		\end{bmatrix} ,
		\]
		where  $K : C^{2,\alpha}_0(\Sigma) \rightarrow C_0^{\alpha}(\Sigma)$, 
		\[
		\psi\mapsto 2\rho_2 \left( \frac{\tilde{V}_2(x, \xi^0)e^{z} \psi(x)}{\int_{\Sigma} \tilde{V}_2(\cdot, \xi^0)e^{z} \, dv_g} - \frac{\tilde{V}_2(x, \xi^0)e^{z} \int_{\Sigma} \tilde{V}_2(\cdot, \xi^0)e^{z} \psi \, dv_g }{(\int_{\Sigma} \tilde{V}_2(\cdot, \xi^0)e^{z} \, dv_g)^2} \right).
		\]
		By the Schauder estimates (see \cite{yang2021125440,Nardi2014}), the mapping $(-\Delta_g,\partial_{\nu_g}) : C^{2,\alpha}_0(\Sigma) \rightarrow C^{\alpha}_0(\Sigma)\times \{0\}$ is  an isomorphism. It is noted that the operator $K$ is compact, a characteristic that preserves the Fredholm index upon addition to $(-\Delta_g,\partial_{\nu_g})$. So, $(-\Delta_g - K, \partial_{\nu_g})$ is a Fredholm operator of index zero. Applying hypothesis \hyperref[item:H]{(C)}, the derivative $D_z\Psi(z, \xi^0)(\psi)$ has a trivial kernel, thereby affirming the non-degeneracy of $D_z\Psi(z, \xi^0)$. The implicit function theorem yields the existence of a radius $r > 0$ and a continuously differentiable mapping 
		\[
		z : \tilde{B}_r(\xi^0) = \{\xi \in \Xi_{k,m} : d_g(\xi, \xi^0) < r\} \rightarrow C^{2,\alpha}_0(\Sigma),
		\]
		\[
		\xi \mapsto z_{\xi},
		\]
		such that  $z_{\xi^0} = z(\cdot,\xi^0)$ and $\Psi(z_{\xi}, \xi) = 0$ for any $\xi\in\{\xi\in \Xi_{k,m}: d_g(\xi,\xi^0)<r \} $. We take $\Dc=\tilde{B}_r(\xi^0).$
	\end{proof}
	
	Let $\i(\xi)= 2$ if $\xi\in\intsigma$ and equal $1$ if $\xi\in \partial\Sigma.$
	\begin{lem}\label{lem:pa_I_xi}
		For $\xi \in\Xi_{k,m}$, $j=1,\ldots,m, i=1,\ldots,\i(\xi)$
		\begin{eqnarray}
			\label{eq:pa_I_id_pa_z} \partial_{(\xi_j)_i} I_{\xi}(z(\cdot,\xi))=\frac 1 2 \varrho(\xi_j) \partial_{x_i} z(x,\xi)|_{x=\xi_j}. 
		\end{eqnarray}
	\end{lem}
	\begin{proof}
		Since $z(\cdot,\xi)$ solves the problem~\eqref{eq:singular_mf}, $I'_{\xi}(z(\cdot,\xi))=0$. 
		By the representation's formula, we deduce that 
		\begin{eqnarray*}
			&&	\partial_{(\xi_j)_i} I_{\xi}(z(\cdot,\xi))\\
			&=& I'_{\xi}(z(\cdot,\xi)) \partial_{(\xi_j)_i}z(\cdot,\xi) -2\rho_2 \int_{\Sigma} \frac{ \tilde{V}_2(\cdot,\xi)e^{z(\cdot,\xi)}}{\int_{\Sigma} \tilde{V}_2(\cdot,\xi)e^{z(\cdot,\xi)}\, dv_g}
			\left(- \frac{\varrho(\xi_j)}{2} \partial_{(\xi_j)_i} G^g(\cdot,\xi_j)\right)\,dv_g \\
			&=&\rho_2 \varrho(\xi_j) \int_{\Sigma} \frac{ \tilde{V}_2(\cdot,\xi)e^{z(\cdot,\xi)}}{\int_{\Sigma} \tilde{V}_2(\cdot,\xi)e^{z(\cdot,\xi)}\, dv_g}
			\partial_{(\xi_j)_i} G^g(\cdot,\xi_j)\,dv_g\\
			&=&\frac 1 2 \varrho(\xi_j) \partial_{(x_j)_i} z(x,\xi)|_{x=\xi_j}.
		\end{eqnarray*}
	\end{proof}
	It is clear that  \hyperref[item:G]{(C1)} implies  \hyperref[item:H]{(C)}. The hypothesis \hyperref[item:H]{(C)} describes the non-degeneracy of solutions of the singular mean field equations~\eqref{eq:singular_mf}. For $\rho_2\in (0,2\pi)$, it is generally observed that hypothesis \hyperref[item:H]{(C)} does not hold. Nonetheless, when $\rho_2>0$ is chosen to be sufficiently small, this hypothesis is satisfied for any $\Dc$ that is an open precompact subset of $\Xi_{k,m}$.
	\begin{lem}
		\label{lem:hypo_H_small_rho2}
		Let $\Dc \subset\overline{\Dc}\subset\Xi_{k,m}$ be an open subset. There exists $\rho_0\in (0,2\pi)$ sufficiently small such that for any $\rho_2\in (0,\rho_0)$,  there exists a unique solution $z(\cdot,\xi)$ of~\eqref{eq:singular_mf} satisfying  \hyperref[item:H]{(C)}. 
	\end{lem}
	\begin{proof}
		To formulate our argument, we select 
		$\Dc$ as an arbitrary open precompact subset of $\Xi_{k,m}$. We will prove the lemma by contradiction.
		
		Let us suppose that there exists a sequence $\rho^n \rightarrow 0$, $\xi^n \in \Dc$, and $z_n^{1}, z_n^{2}\in C^{2,\alpha}_0(\Sigma)$, two distinct solutions of the following problem: 
		\begin{equation}\label{eq:singular_MF_rho}
			\left\{\begin{array}{ll}
				-\Delta_g z(x) = 2 \rho^n\left( \frac{\tilde{V}_2(x, \xi^n)e^{z_n}}{\int_{\Sigma}\tilde{V}_2(\cdot, \xi^n)e^{z_n} \, dv_g } -1\right),&\quad \text{ in }  \intsigma\\
				\partial_{\nu_g}z(x) =0, &\text{ on } \partial \Sigma
			\end{array}\right.. 
		\end{equation}
		Given the compactness of the solution space for~\eqref{eq:singular_mf}, up to a subsequence, for $i=0,1$ $\xi^n \rightarrow \xi^0 \in \overline{\Dc}$, $z_n^{i} \rightarrow z^{i}$ in the H\"{o}lder space $C_0^{2,\alpha}(\Sigma)$.
		As the parameter $\rho^n\rightarrow 0$ in the context of~\eqref{eq:singular_MF_rho}, it follows  $z^{i} =0 $ for $i = 0, 1$.
		We introduce a mapping
		\begin{align*}
			\Psi: C_{0}^{2,\alpha}(\Sigma) \times \Xi_{k,m}\times \mathbb{R} &\rightarrow C_0^{\alpha}(\Sigma)\times \{ 0\}, \\
			(z, \xi, \rho) &\mapsto \begin{bmatrix}
				-\Delta_g z - 2\rho\left(\frac{\tilde{V}_2(x, \xi^n)e^{z_n}}{\int_{\Sigma}\tilde{V}_2(\cdot, \xi^n)e^{z_n} \, dv_g } -1\right)\\
				\partial_{\nu_g} z
			\end{bmatrix} .
		\end{align*}
		We note that  $z$ is a solution of ~\eqref{eq:singular_MF_rho} if and only if $\Psi(z, \xi, \rho) = 0$ and $$D_z\Psi(z^i, \xi^0, 0)(\psi)=\begin{bmatrix}
			-\Delta_g\psi\\
			\partial_{\nu_g}\psi
		\end{bmatrix} .$$ Given that $(-\Delta_g,\partial_{\nu_g})$ is an isomorphism from  $C_{0}^{2,\alpha}(\Sigma)$ to $C_0^{\alpha}(\Sigma)\times\{ 0\}$, applying the Implicit Function Theorem, 
		there exists $r > 0$, $\delta > 0$, $U$ a neighborhood of $0$ in $C_{0}^{2,\alpha}(\Sigma)$ and a $C^1$-diffeomorphism:
		\begin{align*}
			z: \tilde{B}_r(\xi^0):=\{ \xi\in\Xi_{k,m}: d_g(\xi,\xi^0)<r\} \times (-\delta, \delta) &\rightarrow U, \quad
			(\xi, \rho) \mapsto z_{\xi,\rho},
		\end{align*}
		with  $\Psi(z_{\xi,\rho}, \xi, \rho) = 0$. 
		Additionally, $z=z_{\xi,\rho}$ is the unique solution in $U$ satisfying    $\Psi(z, \xi, \rho) = 0$. 
		There exists $N_0>0$ sufficiently large such that $ z^i_n \in U$ and $\xi^n\in\tilde{B}_{r}(\xi^0)$ for any $i=0,1 \text{ and }n\geq N_0.$
		Given the uniqueness of 
		$z_{\xi,\rho}$, we deduce that  $z_n^{i} = z_{\xi^n,\rho^n}$ for $i=1,2$, leading to a contradiction.
		
		We now will establish the non-degeneracy. Suppose $\rho^n \rightarrow 0$, $\xi^n \in \Dc$, $z_n$ a solution of \eqref{eq:singular_MF_rho} and $\psi_n$ a nontrivial solution of the following problem:
		\begin{equation}\label{eq:singular_MF_rho2}
			\left\{\begin{array}{ll}
				-\Delta_g  \psi = K_n(\psi), &\text{ in } \intsigma, \\
				\partial_{\nu_g} \psi= 0 \quad &\text{ on }  \partial \Sigma,
			\end{array}\right.,
		\end{equation}
		where $K_n: C^{2,\alpha}_0(\Sigma) \rightarrow C_0^{\alpha}(\Sigma)$, 
		\[
		K_n(\psi) =2\rho^n \left( \frac{\tilde{V}_2(x, \xi)e^{z} \psi(x)}{\int_{\Sigma} \tilde{V}_2(\cdot, \xi)e^{z} \, dv_g} - \frac{\tilde{V}_2(x, \xi)e^{z} \int_{\Sigma} \tilde{V}_2(\cdot, \xi)e^{z} \psi \, dv_g }{(\int_{\Sigma} \tilde{V}_2(\cdot, \xi)e^{z} \, dv_g)^2} \right). \] 
		Taking the limit of~\eqref{eq:singular_MF_rho2} as $\rho^n\rightarrow 0$, we observe that  $z_n\rightarrow 0$ in $C^{2,\alpha}(\Sigma).$ Without loss of generality, we assume that $\|\psi_n\|_{L^{\infty}(\Sigma)}=1$. This assumption leads to the conclusion that 
		$\|K_n(\psi_n)\|_ {L^{\infty}(\Sigma)}\rightarrow 0$. By the regularity theory, we obtain that $\psi_n=0$, which is a contradiction. Therefore, we concluded the Lemma. 		
	\end{proof}
	On the other hand, we can show that the shadow system~\eqref{eq:shadow} is non-degenerate for generic $(V_1, V_2)\in C^{2,\alpha}(\Sigma, \R_+)\times C^{2,\alpha}(\Sigma,\R_+)$ and for any fixed $\rho_2\geq 0$, by a well-known transversality theorem which can be found in~\cite[Theorem~5.4]{henry2005perturbation}.
	\begin{thm}[see~\cite{henry2005perturbation}]~\label{thm:trans}
		Let $M,\cV,N$ be Banach manifolds of class $\cC^r$ for some $r\in \N$, let $\cD\subset M\times \cV$ be open, let $\cF: \cD\to N$ be a $\cC^r$ map, and fix a point $z\in N$. Assume for each $(y,\psi)\in \cF^{-1}(z)$ that: 
		\begin{itemize}
			\item[(1)] $D_y\cF(y,\psi): T_y M\to T_{z} N$ is semi-Fredholm with index $<r$;
			\item[(2)] $D\cF(y,\psi) : T_y M \times T_\psi \cV \to T_{z} N$ is surjective;
			\item[(3)]  $\cF^{-1}(z)\to \cV$, $(y,\psi)\mapsto \psi$, is $\sigma$-proper. 
		\end{itemize}
		Then for 	$\cD_{\psi}=\{ y\in M : (y,\psi)\in \cD\}$, 
		\[
		\cV_{0}:=\{ \psi\in\cV: z \text{ is a critical value of } \cF(\cdot, \psi): \cD_{\psi}\rightarrow N\}
		\]
		is a meager subset of $\cV$.
	\end{thm}
	For any $p>1,\alpha\in (0,1)$, define
	\[
	M:= C^{2,\alpha}_0(\Sigma) \times \Sigma^k\times (\partial\Sigma)^{m-k}, \quad \cV:=C^{2,\alpha}(\Sigma,\R_+) \times C^{2,\alpha}(\Sigma,\R_+), \quad N: =C_0^{\alpha}(\Sigma) \times\{ 0\}\times  \mathbb{R}^{m+k},
	\]
	and $\cD:=C^{2,\alpha}_0(\Sigma)  \times\Xi_{k,m} \times \cV$ is an open subset of $M$. 
	We note that 
	for any $\xi\in \intsigma^k\times (\partial\Sigma)^{m-k} $, the tangent space $T_{\xi}(\intsigma^k\times (\partial\Sigma)^{m-k})$ is isomorphic to $\R^{m+k}$, so we identify the elements in $T_{\xi}(\intsigma^k\times (\partial\Sigma)^{m-k})$ and $\R^{m+k}.$
	
	We define the following shadow system with parameter $t\in [0,1]$,
	\begin{equation}\label{eq:shadow_t}
		\left\{
		\begin{array}{ll}
			-\Delta_g w^t =2\rho_2 \left( \frac{V_2 e^{ w^t - \sum_{j=1}^m  \frac {\varrho(\xi_j)}2 G^g(x,\xi_j)}} {\int_{\Sigma} V_2 e^{ w^t - \sum_{j=1}^m  \frac {\varrho(\xi_j)}2 G^g(x,\xi_j)}\, dv_g  } - 1\right)& \quad \text{ in } \intsigma\\
			\partial_{\nu_g} w^t=0 & \quad\text{ on } \partial\Sigma\\
			\partial_{x_i} f_t(x)|_{x=\xi} = 0& \quad\text{ in } \Xi_{k,m}\tag{S$_t$}
		\end{array}
		\right.,
	\end{equation}
	where	\[
	\begin{array}{ll}
		f_t(x_1, x_2, \ldots, x_m)=& \cF_{k,m}(x)
		-\sum_{j=1}^m(1-t)\varrho(x_j) w(x_j). 
	\end{array}
	\]
	Consider the map for $t\in [0,1]$
	\begin{equation}
		\label{eq:T_t} 	\cT_t(w,\xi,V_1,V_2) = 
		\begin{bmatrix}
			\Delta_g w + 2\rho_2 \left( \frac{V_2 e^{w-\sum_{j=1}^m \frac {\varrho(\xi_j)} 2 G^g(\cdot,\xi_j)}} {\int_{\Sigma} V_2 e^{w-\sum_{j=1}^m \frac {\varrho(\xi_j)} 2 G^g(\cdot,\xi_j) } \, dv_g} - 1\right) 
			\\
			\partial_{\nu_g} w \\
			\nabla_{\xi_1} f_t(\xi) \\
			\vdots \\
			\nabla_{\xi_m} f_t(\xi)
		\end{bmatrix}.
	\end{equation}
	
	\begin{thm} \label{thm:residual}
		For $t\in[0,1]$, $\cT_t$ is $C^1$-differentiable. Moreover,
		\[
		\cV^t_{reg}:=\left\{ (V_1,V_2) \in \cV :   \text{any solution } (w, \xi) \text{ of } \cT_t(\cdot, V_1,V_2) = 0 \text{ is nondegenerate} \right\}
		\]
		is residual in $\cV$. 
	\end{thm}
	\begin{proof}
		It is easy to check the condition (3) in Theorem~\ref{thm:trans}.
		Let $$M_j:=\left \{(w,\xi)\in C^{2,\alpha}_0(\Sigma)\times \Xi_{k,m}: \text{dist}(\xi, \partial \Xi_{k,m})\geq \frac 1 {2^j} \text{ and } \|w\|_{C^{2,\alpha}(\Sigma)}\leq 2^j\right\},$$
		for $j\in\N$. It follows that 
		\[ C^{2,\alpha}_0(\Sigma)\times \Xi_{k,m}=\bigcup_{j=0}^{+\infty} M_j. \]
		We consider the map $ \cT_t^{-1}(0)\cap (M_j\times \cV) \rightarrow \cV, (t,w,\xi, V_1, V_2)\mapsto (V_1, V_2)$. 
		For any $(t^n,w^n,\xi^n, V^n)\in \cT_t^{-1}(0)\cap (M_j\times \cV)$ such that $V^n:=(V^n_1, V^n_2)\rightarrow V^0:=(V^0_1, V^0_2)$ in $\cV$, by the compactness of $\Xi_{k,m}^{\frac 1{2^j}}$, up to a subsequence,  $\xi^n:=(\xi^n_1,\cdots,\xi^n_m)\rightarrow \xi^0:=(\xi^0_1,\cdots,\xi^0_m)\in \Xi_{k,m}^{\frac 1{2^j}}$.  The Arzelà-Ascoli theorem implies that for any $\alpha'\in (0,\alpha)$, 
		there exists $w^0$ such that 
		\[ w^n\rightarrow w^0 \text{ in } C^{2,\alpha'}_0(\Sigma).\]
		Since $(w^n,\xi^n, V^n)\in\cT_t^{-1}(0)$, we have $(w^n,\xi^n, V^n)\in \cT_t^{-1}(0)$, too. 
		\begin{equation*}
			\left\{\begin{array}{ll}
				- \Delta_g (w^n-w^0)&= 2\rho_2 \frac{V^n_2 e^{w^n-\sum_{j=1}^m \frac {\varrho(\xi^n_j)} 2 G^g(\cdot,\xi^n_j)}} {\int_{\Sigma} V^n_2 e^{w^n-\sum_{j=1}^m \frac {\varrho(\xi^n_j)} 2 G^g(\cdot,\xi^n_j) } \, dv_g}-2\rho_2 \frac{V^0_2 e^{w^0-\sum_{j=1}^m \frac {\varrho(\xi^0_j)} 2 G^g(\cdot,\xi^0_j)}} {\int_{\Sigma} V^0_2 e^{w^0-\sum_{j=1}^m \frac {\varrho(\xi^0_j)} 2 G^g(\cdot,\xi^0_j) } \, dv_g} \text{ in }\intsigma,
				\\
				\partial_{\nu_g} (w^n-w^0)&=0 \text{ on } \partial\Sigma,\\
			\end{array}\right.
		\end{equation*}
		We have $$\|-\Delta_g(w^n-w_0)\|_{C^{\alpha}(\Sigma)}\leq C \|w^n-w^0\|_{C^{\alpha}(\Sigma)}\rightarrow 0,$$
		as $n\rightarrow +\infty.$
		Applying the Schauder estimates  we can derive the convergence
		$ w^n\rightarrow w^0 \text{ in } C^{2,\alpha}_0(\Sigma).$
		Thus, $\cT_t^{-1}(0)\rightarrow \cV, (w,\xi, V)\mapsto V$ is $\sigma$-proper.
		
		To compute the derivative of  $\cT_t$ in the direction of $(w,\xi)$, we proceed as follows. Let  $\phi\in C^{2,\alpha}_0(\Sigma), v=(v_1,\ldots, v_m)\in\R^{m+k}$. The derivative of
		$\cT_t$ with respect to $w$ and $\xi$ is given by 
		\begin{equation}\label{eq:D_w_xi_T_t}
			D_{w,\xi}\cT_t(w, \xi, V_1, V_2)[\phi, v] = 
			\begin{bmatrix}
				T_0(w, \xi, V_1, V_2)[\phi, v] \\
				T_1(w, \xi, V_1, V_2)[\phi, v] \\
				\vdots \\
				T_m(w, \xi, V_1, V_2)[\phi, v]
			\end{bmatrix},
		\end{equation}	
		where	\begin{eqnarray*}
			&&	T_0(w, \xi, V_1, V_2)[\phi, v] 		 =\\
			&&	\left[  \Delta_g \phi + 2\rho_2 \left( \frac{\tilde{V}_2 e^w \phi}{\int_\Sigma \tilde{V}_2 e^w \, dv_g} - \frac{\tilde{V}_2 e^w \int_\Sigma \tilde{V}_2 e^w \phi\,  d v_g}{\left(\int_\Sigma \tilde{V}_2 e^w\,  dv_g\right)^2} \right)
			-\rho_2 \sum_{j=1}^{m} \varrho(\xi_j)\left( \frac{\tilde{V}_2 e^w \nabla_{x_j} G^g(\cdot, x_j)|_{x_j=\xi_j} \cdot v_j }{\int_\Sigma \tilde{V}_2 e^w\, dv_g } \right) \right. \\
			&&\left. + \rho_2 \sum_{j=1}^{m} \left( \frac{\tilde{V}_2 e^w \int_\Sigma \tilde{V}_2 e^w \nabla_{x_j} G^g(\cdot, x_j)|_{x_j=\xi_j} \cdot v_j\, dv_g  }{\left( \int_\Sigma\tilde{V}_2 e^w \, dv_g  \right)^2} \right), \partial_{\nu_g}\phi\right]^{\top},
		\end{eqnarray*}
		and 
		\[\begin{aligned}	  			
			T_i(w, \xi, V_1, V_2)[\phi, v] =& \left. 
			\nabla_{x_i}^2 \left( \sum_{j=1}^{m}\varrho^2(x_j) R^g(\xi_j)
			+ \sum_{\substack{i,j=1 \\ i \neq j}}^{m}\varrho(x_i)\varrho(x_j) G(x_i, x_j)\right)\right|_{x=\xi}\cdot v_i\\
			&
			+ 2 \varrho^2(x_i)\nabla^2_{x_i
			}\log V_1(x_i
			)|_{x_j=\xi_i}\cdot v_i - \varrho(\xi_i) (1-t) \nabla_{x_i}^2w(x_i)\cdot v_i 
			\\
			&-\varrho(\xi_i)(1-t)\nabla_{\xi_i}\phi(\xi_i),  
		\end{aligned}
		\]
		for $i = 1, \ldots, m,$ 
		where
		$
		\tilde{V}_2= V_2 e^{- \sum_{j=1}^{m}\frac {\varrho(\xi_j)} 2 G^g(\cdot,\xi_j)}. $
		By the  formula of $C^1$-derivatives,	it is easy to know that $\cT_t$ is $C^1$-differentiable for $(w,\xi)$ uniformly for any $V_1,V_2$ bounded in $C^{2,\alpha}(\Sigma, \R_+)$ and $t\in [0,1]$. 
		We decompose $ D_{w,\xi}\cT_t$ into the following two linear operatorss: 
		\begin{equation}\label{eq:T_11}
			\cT^t_{11}(w, \xi, V_1, V_2)[\phi, v] = 	\begin{bmatrix}
				\Delta_g \phi + 2\rho_2 \left( \frac{\tilde{V}_2 e^w \phi}{\int_\Sigma \tilde{V}_2 e^w \, dv_g} - \frac{\tilde{V}_2 e^w \int_\Sigma \tilde{V}_2 e^w \phi\,  d v_g}{\left(\int_\Sigma \tilde{V}_2 e^w\,  dv_g\right)^2} \right) \\
				\partial_{\nu_g}\phi\\ 
				0\\
				\vdots \\
				0
			\end{bmatrix},
		\end{equation}
		and 
		\begin{eqnarray}\label{eq:T_12}
			&& \cT^t_{12}(w, \xi, V_1, V_2)[\phi, v] = \begin{bmatrix}
				T_{10} (w, \xi, V_1, V_2)[\phi, v] \\
				T_{11}(w, \xi, V_1, V_2)[\phi, v] \\
				\vdots \\
				T_{1m}(w, \xi, V_1, V_2)[\phi, v]
			\end{bmatrix},
		\end{eqnarray}
		where 
		$$
		\begin{aligned}
			&T_{10}(w, \xi, V_1, V_2)[\phi, v]\\
			& =\left[-\rho_2 \sum_{j=1}^{m} \varrho(\xi_j)\left( \frac{\tilde{V}_2 e^w}{\int_\Sigma \tilde{V}_2 e^w\, dv_g } \nabla_{x_j} G^g(\cdot, x_j)|_{x_j=\xi_j} \cdot v_j \right)\right. \\
			& \left. + \rho_2 \sum_{j=1}^{m} \left( \frac{\tilde{V}_2 e^w}{\left( \int_\Sigma\tilde{V}_2 e^w\, dv_g  \right)^2} \int_\Sigma \tilde{V}_2 e^w \nabla_{x_j} G^g(\cdot, x_j)|_{x_j=\xi_j} \cdot v_j\, dv_g  \right), 0\right]^{\top},
		\end{aligned}$$
		and 
		$T_{1j}= T_j$ for any $j=1,\ldots,m.$
		
		Since $(-\Delta_g, \partial_{\nu_g}): C^{2,\alpha}_0(\Sigma)\rightarrow C^{\alpha}_0(\Sigma)\times\{ 0\}$ is an isomorphism and for any fixed $\xi$, applying the compactness of $\Sigma$ and compact embedding theorem of Sobolev's spaces we can deduce that  $\phi\mapsto  2\rho_2 \left( \frac{\tilde{V}_2 e^w \phi}{\int_\Sigma \tilde{V}_2 e^w \, dv_g} - \frac{\tilde{V}_2 e^w \int_\Sigma \tilde{V}_2 e^w \phi\,  d v_g}{\left(\int_\Sigma \tilde{V}_2 e^w\,  dv_g\right)^2} \right)$ is a compact operator. So, we can obtain that 
		\[ \dim(\text{ker}(\cT^t_{11}(w,\xi, V_1, V_2)))= \text{codim}(\text{Im}(\cT^t_{11}(w,\xi, V_1, V_2)))=m+k, 
		\]
		i.e. $\cT^t_{11}(w,\xi,V_1, V_2)$ is a Fredholm operator with index $0$. Similarly, by the compactness of $\Sigma$ and the compact embedding theorem of Sobolev's space, 
		$\cT^t_{12}(w,\xi, V_1, V_2)$ is a compact operator, too. It follows that 
		$D_{w,\xi}\cT_t(w,\xi, V_1, V_2)$ is a Fredholm operator with index $0$. 
		The condition (1) in Theorem~\ref{thm:trans} is satisfied. 
		
		Next, we will show that condition (2) holds true. Let $(w,\xi, V_1, V_2)\in \cT_t^{-1}(0)$. We have the derivative of $\cT_t$ with respect to $V=(V_1, V_2)$ as follows: 
		\begin{eqnarray*}
			D_{V_1}	\cT_t(w, \xi, V_1, V_2)[h_1] = 
			\left[ \begin{array}{c}
				0 \\
				0\\
				2\rho^2(\xi_1)
				\left(\frac{\nabla_{\xi_1}h_1(\xi_1)}{V_1(\xi_1)}- \frac{ h_1(\xi_1)\nabla_{\xi_1} V_1(\xi_1)}{V_1^2(\xi_1)}\right) \\
				\vdots \\
				2\rho^2(\xi_m)
				\left(\frac{\nabla_{\xi_m}h_1(\xi_m)}{V_1(\xi_m)}- \frac{ h_1(\xi_m)\nabla_{\xi_m} V_1(\xi_m)}{V_1^2(\xi_m)}\right) 
			\end{array} \right],
		\end{eqnarray*}
		and 
		\begin{eqnarray*}
			&&D_{V_2}\cT_t(w, \xi, V_1, V_2)[h_2] = 
			\left[ \begin{array}{c}
				2\rho_2 \left( \frac{\tilde{V}_2 e^w }{\int_\Sigma \tilde{V}_2 e^w \, dv_g} \frac{h_2}{V_2}- \frac{\tilde{V}_2 e^w \int_\Sigma \tilde{V}_2 e^w \frac{h_2}{V_2}\,  d v_g}{\left(\int_\Sigma \tilde{V}_2 e^w\,  dv_g\right)^2}  \right) \\
				0\\
				0 \\
				\vdots \\
				0 
			\end{array} \right].\
		\end{eqnarray*}
		For 
		$v=(v_1,\ldots,v_m) = 0\in \R^{m+k},$ and  $h_1\in C^{2,\alpha}(\Sigma)$ satisfying that
		\[ 2\rho^2(\xi_i)
		\left(\frac{\nabla_{\xi_i}h_1(\xi_i)}{V_1(\xi_i)}- \frac{ h_1(\xi_i)\nabla_{\xi_i} V_1(\xi_i)}{V_1^2(\xi_i)}\right) =\rho(\xi_i) \nabla_{\xi_i}\phi(\xi_i), \]
		for $i=1,\ldots,m,$ we have 
		\begin{eqnarray*}
			&&D_{w, \xi}\cT_t(w, \xi, V_1,V_2)[\phi, v_1, \dots, v_m] + D_{V_1}\cT_t(w, \xi, V_1,V_2)[h_1] \\
			&=& 
			\left[ \begin{array}{c}
				\Delta_g \phi + 2\rho_2 \left( \frac{\tilde{V}_2 e^w \phi}{\int_\Sigma \tilde{V}_2 e^w \, dv_g} - \frac{\tilde{V}_2 e^w \int_\Sigma \tilde{V}_2 e^w \phi\,  d v_g}{\left(\int_\Sigma \tilde{V}_2 e^w\,  dv_g\right)^2} \right)\\
				\partial_{\nu_g}\phi\\
				0 \\
				\vdots \\
				0 
			\end{array} \right].
		\end{eqnarray*}
		\begin{claim}\label{calim:2.6}
			Let $L(\phi)= \begin{bmatrix}
				-\Delta_g \phi - 2\rho_2 \left( \frac{\tilde{V}_2 e^w \phi}{\int_\Sigma \tilde{V}_2 e^w \, dv_g} - \frac{\tilde{V}_2 e^w \int_\Sigma \tilde{V}_2 e^w \phi\,  d v_g}{\left(\int_\Sigma \tilde{V}_2 e^w\,  dv_g\right)^2} \right)\\
				\partial_{ \nu_g} \phi
			\end{bmatrix}.$
			If $\phi\in C^{2,\alpha}_0(\Sigma)$ such that 
			\begin{eqnarray*}
				L(\phi)=0,&\\
				\int_{\Sigma} \left( \frac{\tilde{V}_2 e^w }{\int_\Sigma \tilde{V}_2 e^w \, dv_g} \frac{h_2}{V_2}- \frac{\tilde{V}_2 e^w \int_\Sigma \tilde{V}_2 e^w \frac{h_2}{V_2}\,  d v_g}{\left(\int_\Sigma \tilde{V}_2 e^w\,  dv_g\right)^2}  \right)\phi\, dv_g =0,& \label{eq:othogonal_h2_v_2_phi}
			\end{eqnarray*}
			$\text{ for any } h_2\in C^{2,\alpha}(\Sigma),$	then $\phi=0.$
		\end{claim}
		
		We define that 
		\[  H= \text{span}\left\{ \text{Im}(D_{w, \xi}\cT_t(w, \xi, V_1,V_2)[\cdot, v] + D_{V_1}\cT_t(w, \xi, V_1,V_2)[h_1]) \bigcup \text{Im}(D_{V_2}\cT_t(w, \xi, V_1, V_2))\right\}.\]
		We will prove that 
		\begin{eqnarray}\label{eq:H_equiv_1space}
			\label{eq:induction} H=&C^{\alpha}_0(\Sigma)\times\{ 0\} \times\underbrace{\{0\}\times\ldots\times\{0\}}&\\
			& \quad \quad \quad \quad \quad \quad m+k &\nonumber.
		\end{eqnarray}
		For any $h\in C^{\alpha}_0(\Sigma)$, since $(-\Delta_g, \partial_{\nu_g}): C^{2,\alpha}(\Sigma)\rightarrow C^{\alpha}_0(\Sigma)\times\{0\}$ is an isomorphism, there exists $\phi_0\in C^{2,\alpha}_0(\Sigma)$ with $\partial_{ \nu_g}\phi_0=0$ such that 
		\[ \Delta_g \phi_0= h.\]
		We recall that  $K(\phi)= 2\rho_2 \left( \frac{\tilde{V}_2 e^w }{\int_\Sigma \tilde{V}_2 e^w \, dv_g} \phi- \frac{\tilde{V}_2 e^w \int_\Sigma \tilde{V}_2 e^w \phi\,  d v_g}{\left(\int_\Sigma \tilde{V}_2 e^w\,  dv_g\right)^2}  \right).$
		Then 
		\begin{eqnarray*}
			\begin{bmatrix}
				h\\
				0\\
				0\\
				\vdots
				\\
				0
			\end{bmatrix}=\begin{bmatrix}
				\Delta_g \phi_0\\
				0\\
				0\\
				\vdots
				\\
				0
			\end{bmatrix}=\begin{bmatrix}
				\Delta_g\phi_0+K(\phi_0)\\
				\partial_{\nu_g}\phi_0\\
				0\\
				\vdots
				\\
				0
			\end{bmatrix}- D_{V_2}\cT_t(w, \xi, V_1, V_2))[V_2 \phi_0]\in H. 
		\end{eqnarray*}
		Hence,~\eqref{eq:H_equiv_1space} is concluded.
		Next, we will show that 
		\begin{equation}
			\label{eq:H_equiv_2space}
			\text{Im}(D_{V_1}\cT_t(w, \xi, V_1, V_2))
			=\{0\}\times\{0\}\times\R^{m+k}. 
		\end{equation}
		We select functions  $ h_{11}, h_{12}\in C^{2,\alpha}(\Sigma) $, such that
		\begin{equation}
			\label{eq:null_h_1}
			h_{1i}(\xi_j) = 0, \quad\text{ and } \quad \nabla_{\xi_j} h_{1i}(\xi_j) = 0, \quad 2 \leq j \leq m, \quad i = 1, 2. 
		\end{equation}
		In the case where  $\xi_1\in \intsigma$,  it is possible to find functions  $h_{11}, h_{12}$ that
		satisfy~\eqref{eq:null_h_1} and fulfill the following conditions:
		\[
		\frac{\nabla_{\xi_1}h_{11}(\xi_1)}{V_1(\xi_1)}- \frac{ h_{11}(\xi_1)\nabla_{\xi_1} V_1(\xi_1)}{V_1^2(\xi_1)} = (1, 0),
		\]
		and
		\[
		\frac{\nabla_{\xi_1}h_{12}(\xi_1)}{V_1(\xi_1)}- \frac{ h_{12}(\xi_1)\nabla_{\xi_1} V_1(\xi_1)}{V_1^2(\xi_1)} = (0,1).
		\]
		If $\xi_1\in\partial\Sigma$, we also can find $h_{11}, h_{12}$ that satisfy~\eqref{eq:null_h_1} such that 
		\[ \frac{\nabla_{\xi_1}h_{11}(\xi_1)}{V_1(\xi_1)}- \frac{ h_{11}(\xi_1)\nabla_{\xi_1} V_1(\xi_1)}{V_1^2(\xi_1)} = 1.\]
		It follows that
		\[
		[0,0, c_1,0,\ldots, 0]^{\top}\in 	\text{Im}(D_{V_1}\cT_t(w, \xi, V_1, V_2)),
		\]
		for all $ c_1\in \R^2 $ when $\xi_1\in\intsigma$ and $c\in\R$ when $\xi_1\in\partial\Sigma$. Similarly,  appropriate selection of $h_{11}, h_{12}$ allows us to infer that
		\[
		[0, 0,c_{1},\ldots, c_m]^{\top} = \text{Im}(D_{V_1}\cT_t(w, \xi, V_1, V_2)),
		\]
		for all $ (c_1,\ldots,c_m) \in \R^{m+k} $. Consequently, this leads us to conclude equation~\eqref{eq:H_equiv_2space}, demonstrating that $D\cT_t(w, \xi, V_1, V_2)$ is surjective.
		
		Invoking  Theorem~\ref{thm:trans},  we ascertain that
		\begin{eqnarray*}
			\cV_0&=\{V\in \cV: 0 \text{ is a critical value of } \cT_t(\cdot,V): C^{2,\alpha}_0(\Sigma)\times \Xi_{k,m}\rightarrow C^{\alpha}_0(\Sigma)\times\{0\}\times\R^{m+k}\}
		\end{eqnarray*}
		is a meager subset of $\cV$. Consequently, the complement
		$ \cV^t_{reg}=\cV\setminus \cV_{0}$
		is a residual set in $\cV.$
		\par 
		{\it Proof of Claim~\ref{calim:2.6}.}
		If $\phi\in \text{ker}(L)$, for any $h\in L^2(\Sigma)$ we have 
		\begin{eqnarray}\label{eq:orth_phi_h}
			\int_{\Sigma} (-\Delta_g\phi) hdv_g = 2\rho_2 \left( \frac{\int_{\Sigma}\tilde{V}_2 e^w \phi h\, dv_g }{\int_\Sigma \tilde{V}_2 e^w \, dv_g} - \frac{\int_{\Sigma} \tilde{V}_2 e^w h\, dv_g \int_\Sigma \tilde{V}_2 e^w \phi\,  d v_g}{\left(\int_\Sigma \tilde{V}_2 e^w\,  dv_g\right)^2} \right).
		\end{eqnarray}
		Since $C^{2,\alpha}(\Sigma)$ is dense in $L^2(\Sigma)$, \eqref{eq:othogonal_h2_v_2_phi} implies that 
		the right hand side of~\eqref{eq:orth_phi_h} vanishes, i.e. 
		\[ 	\int_{\Sigma} (-\Delta_g \phi) h\, dv_g=0\]
		for any $h\in L^2(\Sigma)$. So, $-\Delta_g\phi=0$ in $\intsigma$ and $\partial_{\nu_g}\phi=0$ on $\partial\Sigma.$
		By the Schauder estimate in~\cite{yang2021125440}, $\phi=0$ in $ C^{2,\alpha}_0(\Sigma)$. 
	\end{proof}
	
	Using the method of continuity, we can conclude the following proposition: 
	\begin{prop}\label{prop:ex_sols_shadow}
		For generic $(V_1, V_2)\in C^{2,\alpha}(\Sigma,\R_+)\times C^{2,\alpha}(\Sigma,\R_+)$, if one of the following condition holds: 
		\begin{itemize}
			\item [a)]$\rho_2\in (0,2\pi)$ 
			\item [b)]$\rho_2\in (2\pi, +\infty)\setminus 2\pi \N_+$ and the Euler characteristic $ \chi(\Sigma)<1$,
		\end{itemize} 
		then, there exists a non-degenerate solution $(w,\xi)$ of the shadow system~\eqref{eq:shadow}. 
	\end{prop}
	\begin{proof}
			We recall the shadow system for $(w,\xi)$ in $ C_0^{2,\alpha}(\Sigma)\times \Xi_{k,m}$:
		\begin{equation*}
			\left\{
			\begin{array}{ll}
				-\Delta_g w =2\rho_2 \left( \frac{V_2 e^{ w - \sum_{i=1}^m  \frac {\varrho(\xi_i)}2 G^g(x,\xi_i)}} {\int_{\Sigma} V_2 e^{ w - \sum_{i=1}^m  \frac {\varrho(\xi_i)}2 G^g(x,\xi_i)} dv_g  } - 1\right)& \quad \text{ in } \intsigma\\
				\partial_{\nu_g} w=0 & \quad\text{ on } \partial\Sigma\\
				\nabla f_0(\xi)= 0& \quad\text{ in } \Xi_{k,m}
			\end{array}
			\right.,
		\end{equation*}
		where	\[
		\begin{array}{ll}
			f_0(x_1, x_2, \cdots, x_m)=& \cF_{k,m}(x)
			-\sum_{i=1}^m\varrho(x_i) w(x_i). 
		\end{array}
		\]
		To apply the method of continuity, we introduce a parameter $t\in [0,1]$ to deform~\eqref{eq:shadow} to a decoupled system. 	Clearly,  when $t=0$, the shadow system \eqref{eq:shadow_t} equals~\eqref{eq:shadow}, and when  $t=1$, \eqref{eq:shadow_t} is a decoupled system,
		\begin{equation}
			\label{eq:singular_mf_t} \left\{
			\begin{array}{ll}
				-\Delta_g w =2\rho_2 \left( \frac{V_2 e^{ w - \sum_{i=1}^m  \frac {\varrho(\xi_i)}2 G^g(x,\xi_i)}} {\int_{\Sigma} V_2 e^{ w - \sum_{i=1}^m  \frac {\varrho(\xi_i)}2 G^g(x,\xi_i)} dv_g  } - 1\right)& \quad \text{ in } \intsigma\\
				\partial_{\nu_g} w=0 &\text{ on } \partial\Sigma\\
			\end{array}
			\right.,
		\end{equation}
		with $
		\nabla  \cF_{k,m}(\xi) = 0$ for $\xi \text{ in }\Xi_{k,m}$. 
		Observing that 
		$\cF_{k,m}(\xi)\rightarrow +\infty$ as $\xi\rightarrow \partial\Xi_{k,m}$ by~\cite[Lemma 7.3 ]{HBA2024}, we can deduce the existence of a global minimum point of 
		$\cF_{k,m}$, denoted by $\xi^1\in \Xi_{k,m}$.
		On the one hand, 
		for any $\rho_2\in (0,2\pi)$ and $\xi\in \Xi_{k,m}$,  the singular mean field equation~\eqref{eq:singular_mf_t} is solvable via the Moser-Trudinger inequality and standard variational methods. On the other hand, applying  the existence result of singular Toda system in \cite[Theorem 1.2.]{Hu2024}, we deduce \eqref{eq:singular_mf_t} has a solution for $2\rho_2 \in (4\pi, +\infty)\setminus 4\pi \N_+$ with the assumption $\chi(\Sigma)<1.$
		Consequently, a solution $(w^1,\xi^1)$ of \eqref{eq:shadow_t} exists for $t=1$ when conditions $a)$ or $b)$ are met. 
		
		Let $\Q_0:=\{t\in [0,1]: t \text{ is a rational number}\}$. Since $\Q_0$ is countable, we rewrite $\Q_0 $  as $\{ t_n: n\in \N_+\}$. 
		Theorem~\ref{thm:residual} implies that for any $n\in\N_+,$  $\cV^{t_n}_{reg}$ is a residual set in $C^{2,\alpha}(\Sigma,\R_+)\times C^{2,\alpha}(\Sigma,\R_+)$. Then the intersection set $ \cV_{reg}:= \cap_n\cV^{t_n}_{reg}$ remains a residual set in $C^{2,\alpha}(\Sigma,\R_+)\times C^{2,\alpha}(\Sigma,\R_+)$, too.  By fixing arbitrary $(V_1,V_2)\in \cV_{reg} $, it follows that any solution $(w^t,\xi^t)$ of~\eqref{eq:shadow_t} is non-degenerate for any $t\in \Q_0$.

		We define that 
		\[ T=\left\{ t\in [0,1]: \exists (w^t,\xi^t)\in C^{2,\alpha}_0\times \Xi_{k,m} \text{ is a non-degenerate solution of } \eqref{eq:shadow_t}\right\}.\]
		Considering that $1\in \Q_0$, it follows that $1\in T(\neq\emptyset)$ based on the analysis above. By the method of continuity, it is sufficient to show that $T$ is both close and open in $[0,1].$
		
		Suppose that $t_n\in T$ such that $t_n\rightarrow t_0$ as $n\rightarrow+\infty$. Then there exists $(w^n, \xi^n):=(w^{t_n}, \xi^{t_n})$ solves $(\mathrm{S}_{t_n})$ with $\xi^n$ is a critical point of $f_{t_n}$ in $\Xi_{k,m}$.  Proposition~\ref{prop:cpt_shadow}  implies that for same $\alpha\in (0,1)$
		\[ \| w^n\|_{C^{2,\alpha}(\Sigma)}\leq C,\]
		for constant $C>0$. By the Arzelà–Ascoli theorem, for any $\alpha'\in (0,\alpha)$, there exists a $w^0\in C^{2,\alpha'}(\Sigma)$ such that 
		\[ w^n\rightarrow w^0 \text{ as } n\rightarrow +\infty, \]
		strongly in $C^{2,\alpha'}(\Sigma)$ and $\xi^n\rightarrow \xi^0$ with $\xi^0\in \overline{\Xi}_{k,m}$.
		\cite[Lemma 4.1]{BartschHuSubmitted} implies that  $$|\nabla \cF_{k,m}(\xi)|_g\rightarrow +\infty$$ as $\xi\rightarrow \partial\Xi_{k,m}$. Combining with the uniformly bounded of $\|w^n\|_{C^1(\Sigma)}$, there exists $\delta>0$ such that $\dist(\xi^n, \partial\Xi_{k,m})\geq \delta.$ It follows that $\xi^0\in \Xi_{k,m}.$
		
		For any $\varphi\in \oH$, as $n\rightarrow +\infty$
		\begin{eqnarray*}
			\int_{\Sigma} \lan \nabla w^n,\nabla \varphi\ran_g  dv_g \rightarrow \int_{\Sigma} \lan \nabla w^0,\nabla \varphi\ran_g  dv_g,
		\end{eqnarray*} and 
		\begin{eqnarray*}
			&&2\rho_2	\int_{\Sigma} \left( \frac{V_2 e^{ w^n - \sum_{i=1}^m  \frac {\varrho(\xi^n_i)}2 G^g(x,\xi^n_i)}} {\int_{\Sigma} V_2 e^{ w - \sum_{i=1}^m  \frac {\varrho(\xi^n_i)}2 G^g(x,\xi^n_i)} dv_g  } - 1\right) \varphi  dv_g\\
			&
			\longrightarrow & 2\rho_2	\int_{\Sigma} \left( \frac{V_2 e^{ w^0 - \sum_{i=1}^m  \frac {\varrho(\xi^0_i)}2 G^g(x,\xi^0_i)}} {\int_{\Sigma} V_2 e^{ w^0- \sum_{i=1}^m  \frac {\varrho(\xi^0_i)}2 G^g(x,\xi^0_i)} dv_g  } - 1\right) \varphi  dv_g.
		\end{eqnarray*}
		By Schauder estimates and $\|\Delta_g w^0\|_{C^{\alpha}}< +\infty$, we have $w^0\in C^{2,\alpha}_0(\Sigma).$
		
		Observe that $\nabla\cF_{k,m}, \nabla w^0$ is continuous at $\xi^0$ 
		and for some constant $C>0$
		\[ |\nabla f_t(x)-\nabla f_{t_0}(x)|\leq C|t-t_0|. \] 
		
		It is easy to verify that $\xi^0$  is a critical point of $f_{t_0}$. Indeed, 
		\begin{eqnarray*}
			\nabla	f_{t_0}(\tilde{\xi}^0)&=&\nabla ( f_{t_0}(\xi^0)- f_{t_0}(\xi^n))+\nabla (f_{t_0}(\xi^n)- f_{t_n}(\xi^n))+ \nabla f_{t_n}(\xi^n)\\
			&=& \nabla ( f_{t_0}(\xi^0)- f_{t_0}(\xi^n))+\nabla (f_{t_0}(\xi^n)- f_{t_n}(\xi^n)).
		\end{eqnarray*}
		By the arbitrariness of $n$, we have $\nabla f_{t_0}(\xi^0)=0.$
		Hence, $(w^0,\xi^0)$ solves~($\mathrm{S}_{t_0}$).
		It remains to show the non-degeneracy of the solution $(w^0,\xi^0)$.  In the proof of Theorem~\ref{thm:residual}, it is established that $D_{w,\xi}\cT_{t}$ is a Fredholm operator with index $0$. Given $(h, 0, \zeta)\in C^{\alpha}_0(\Sigma)\times \{ 0\}\times \R^{m+k}$, the non-degeneracy yields that there exists 
		$(\phi_n, v_n)\in C^{2,\alpha}_0(\Sigma)\times \R^{m+k}$
		such that 
		\[ D_{w,\xi}\cT_{t_n}(w^n,\xi^n, V_1, V_2)[ \phi_n, v_n]= (h, 0, \zeta). \]
		Considering that $\|w^n\|_{C^{2,\alpha}(\Sigma)}\leq C$, the operator norm of $D_{w,\xi}\cT_{t_n}(w^n,\xi^n, V_1, V_2)$ is uniformly bounded. It follows that $\phi_n$ is uniformly bounded in $C^{2,\alpha}(\Sigma)$. 
		By the Arzelà–Ascoli theorem, up to a subsequence,  $\phi_n\rightarrow \phi_0$ weakly in $C^{2,\alpha}(\Sigma)$ and strongly in $C^2(\Sigma)$ for some $\phi_0\in C_0^{2,\alpha}(\Sigma)$ and 
		$v_n\rightarrow v^*$ for same $v_0\in \R^{m+k}$.
		Passing the limit $n\rightarrow+\infty$, 
		\[  D_{w,\xi}\cT_{t_n}(w^n,\xi^n, V_1, V_2)[ \phi_n, v_n] \rightarrow  D_{w,\xi}\cT_{t_0}(w^0,\xi^0, V_1, V_2)[ \phi_0, v_0],\]
		which implies that $D_{w,\xi}\cT_{t_0}(w^0,\xi^0, V_1, V_2)[ \phi_0, v_0]= (h, 0, \zeta).$ Hence, we prove that $(w^0,\xi^0)$ is a non-degenerate solution of ($\mathrm{S}_{t_0}$).

		It is clear that $(w,\xi)$ solves~\eqref{eq:shadow_t} if and only if $\cT_t(w,\xi,V_1, V_2)=0$ (see~\eqref{eq:T_t}).
		Due to the non-degeneracy of $\cT_{t_0}$  at $(w^0,\xi^0)$, the implicit function theorem yields in a small open neighborhood of $t_0$ in $[0,1]$, denoted by $I_{t_0}$, 
		\[ t\mapsto (w^t,\xi^t)\]
		is continuous satisfying that $(w^t,\xi^t)$ is the unique solution of $\cT_t(\cdot, V_1, V_2)=0$ and $(w^{t_0}, \xi^{t_0})=(w^0,\xi^0)$. 
		Since for any $t\in I_{t_0}\cap\Q_0$, $(w^t,\xi^t)$ is a non-degenerate solution of~$\cT_t(\cdot, V_1, V_2)=0$, 
		it follows $I_{t_0}\cap\Q_0.$ Using the closeness argument we proved before, we can deduce that $I_{t_0}\subset T$.  
		
		To sum up, we obtain that the non-empty set $T$ is both close and open in $[0,1]$. Hence, $T=[0,1]$. Consequently, there exists a non-degenerate solution of the shadow system~\eqref{eq:shadow}.
		
	\end{proof}

	\subsection{The approximation solutions} 
	To construct blow-up solutions of \eqref{eq:toda}, it is sufficient to consider the following problem: 
	\begin{equation}
		\label{eq:toda_equi}
		\left\{\begin{array}{lc}
			-\Delta_g u_1= (2\lambda V_1e^{u_1}-\overline{2\lambda V_1e^{u_1}})-\left(\rho_2 \frac{V_2e^{u_2}}{\int_{\Si} V_2e^{u_2} dv_g} -\rho_2\right)& \text{ in } \intsigma \\
			-\Delta_g u_2 = (2\rho_2 \frac{V_2e^{u_2}}{\int_{\Si} V_2e^{u_2} dv_g} -2\rho_2)- \left(\lambda V_1 e^{u_1}-\overline{\lambda V_1 e^{u_1}}\right)& \text{ in }\intsigma\\
			\partial_{ \nu_g } u_1=\partial_{ \nu_g } u_2=0 & \text{ on } \partial\Sigma
		\end{array}\right.,
	\end{equation}
	where $\rho_2\in (0, 2\pi)$ and $\lambda$ is a positive parameter.
	We are going to construct a family blow-up solutions $u_{\lambda}=(u_{1,\lambda}, u_{2,\lambda})$ of ~\eqref{eq:toda_equi} as $\lambda\rightarrow 0$  which blows up exactly at $\{\xi_1,\ldots,\xi_m\}$ where $\xi=(\xi_1,\ldots,\xi_m)\in \intsigma^k\times(\partial\Sigma)^{m-k}$ is a $C^1$-stable critical point of $\Lambda_{k,m}$ with the limit mass 
	\[ \lim_{\lambda\rightarrow 0}  \lambda \int_{\Sigma} V_1 e^{u_{1,\lambda}}= \sum_{j=1}^m \frac  1 2 \varrho(\xi_j)=2\pi(k+m).\]
	Setting $\rho_1= \lambda\int_{\Sigma} V_1 e^{u_1}$, then we construct a family blow-up solutions of~\eqref{eq:toda} as $ \rho_1\rightarrow 2\pi(k+m).$
	We only have the first component $u_1$ which blows up and its profile is related to the solutions of the Liouville problem. 
	
	All  solutions of the Liouville equation \begin{equation}\label{eq:anatz}
		\left\{ \begin{array}{lc}
			-\Delta w= e^{w}& \text{ in }\R^2\\
			\int_{\R^2}  e^{w}<\infty&
		\end{array}
		\right.,
	\end{equation}
	can be expressed as follows:
	\[
	u_{\eta,\delta}(y) := \log \frac{ 8 \delta^{2}}{(\delta^{2} + |y-\eta|^{2})^2} \quad y \in \mathbb{R}^2, \delta > 0, \eta\in\R^2. 
	\]
	Moreover, we have
	$
	\int_{\mathbb{R}^2}  e^{u_{0,\delta}(y)} \, dy = 8\pi.
	$
	Applying  the isothermal coordinate $(y_{\xi}, U(\xi))$, we can pull-back $u_{\delta,0}$ to the Riemann surface around $\xi$,
	\[U_{\xi,\delta}:=\log \frac{8  \delta^{2}}{(\delta^2 + |y_{\xi}(x)|^{2})^2} \text{ in } U(\xi).  \]
	Then we project the local bubbles into the functional space $\oH$ by following equations:
	\begin{equation}
		\label{eq:proj}
		\left\{\begin{array}{lc}
			-\Delta_gPU_{\xi,\delta} = \chi_{\xi} e^{-\varphi_{\xi}}  e^{U_{\xi,\delta}}-\overline{\chi_{\xi} e^{-\varphi_{\xi}}  e^{U_{\xi,\delta}}} &\text{ in } \intsigma\\
			\partial_{\nu_g} PU_{\xi,\delta}=0 &\text{ on } \partial\Sigma\\
			\int_{\Sigma} PU_{\xi,\delta} \,dv_g=0&
		\end{array}\right..
	\end{equation}
	
	For any fixed $\varepsilon>0$, for any $\xi\in \Xi^{\varepsilon}_{k,m}$, by the compactness of $\Sigma$,  we can choose a uniformly $r_0>0$ which only depends on $\varepsilon$ such that $r_{\xi_i}\geq 4 r_0$, there exists an isothermal chart $(y_{\xi_j}, U_{4r_0}(\xi_i))$ around $\xi_i$ such that $U_{4r_0}(\xi_i)\cap U_{4r_0}(\xi_j)=\emptyset$ for any $i\neq j$ and $U_{4r_0}(\xi_i)\cap \partial\Sigma=\emptyset$ for $i=1,\ldots,k.$
	We take the concentration parameter 
	\begin{equation}
		\label{eq:def_delta_i} \delta_j= d_j\lambda^{\frac 1 2 },
	\end{equation}
	for some $d_j>0$ which will be chosen later. 
	For simplicity of the notations,
	we denote that $U_j= U_{\xi_j,\delta_j}, PU_j=PU_{\xi_j,\delta_j}, \chi_j=\chi(y_{\xi_j}/r_0)$ and $\varphi_j=\hat{\varphi}_{\xi_j}(y_{\xi_j})$.
	Assuming  \hyperref[item:H]{(C1)}, the approximation solution $\bW_{\lambda}=(W_{1,\lambda}, W_{2,\lambda})$ is defined by 
	\begin{equation*}
		W_{1,\lambda}=\sum_{j=1}^m PU_j -\frac 1 2 z(\cdot,\xi),\text{ and } W_{2,\lambda}= z(\cdot,\xi) -\frac 1 2 \sum_{j=1}^m PU_j,
	\end{equation*}
	where $z(\cdot,\xi)$ is the unique solution of~\eqref{eq:singular_mf}. 
	Next, we are going to construct the solutions with the form
	\[ \bu_{\lambda}=\bW_{\lambda}+\bphi_{\lambda},\]
	where $\bphi_{\lambda}=(\phi_{1,\lambda}, \phi_{2,\lambda})$ is the error term. 
	By Lemma~\ref{lem:extension_PU}, we have as $\lambda\rightarrow 0$
	\begin{equation}
		\label{eq:expansion_W1} W_{1,\lambda}=\sum_{j=1}^m\chi_j\log\frac{1}{(\delta_j^2+|y_{\xi_j}|^2)^2}+\sum_{j=1}^m \varrho(\xi_j) H^g(\cdot,\xi_j)-\frac 1 2 z(\cdot,\xi)+ \mathcal{O}(\lambda|\log \lambda|)
	\end{equation}
	and 
	\begin{eqnarray}
		\label{eq:expansion_W2} W_{2,\lambda}&=&z(\cdot,\xi)-\frac 1 2\sum_{j=1}^m\chi_j\log\frac{1}{(\delta_j^2+|y_{\xi_j}|^2)^2} -\frac 12 \sum_{j=1}^m \varrho(\xi_j) H^g(\cdot,\xi_j) +  \mathcal{O}(\lambda|\log \lambda|)\\
		&=& \log \tilde{V}_2(\cdot,\xi)-\log V_2 +z(\cdot,\xi)+\mathcal{O}(\lambda|\log\lambda|).\nonumber
	\end{eqnarray}
	For $j=1,\ldots,m$, we define  
	\begin{equation}
		\label{eq:def_theta}
		\begin{array}{lcl}
			\Theta_{j}(y)= W_{1,\lambda} \circ y_{\xi_j}^{-1}(\delta_j y) + \hat{\varphi}_{\xi_j}(\delta_j y) -\chi(\delta_j|y|/r_0) U_j \circ y_{\xi_j}^{-1}(\delta_j y) +\log V_1\circ y_{\xi_j}^{-1}(\delta_j y)+\log(2\lambda),  
		\end{array}
	\end{equation}
	for  $y\in \Omega_j:=\frac 1 {\delta_j} B^{\xi_j}_{2r_0}$. 
	Lemma~\ref{lem:extension_PU} and \eqref{eq:def_delta_i} imply that 
	\begin{eqnarray}
		\label{eq:est_theta}  \Theta_{j}(y)&=&-2\log d_j -2\log 2 +\varrho(\xi_j)H^g(\xi_j,\xi_j)+\sum_{l\neq j}\varrho(\xi_l)G^g(\xi_j,\xi_l)+ \log V_1(\xi_j)\\
		&&-\frac 1 2 z(\xi_j,\xi)+ \mathcal{O}(\lambda|\log\lambda|+ \lambda^{\frac 1 2}|y|).\nonumber
	\end{eqnarray}
	To ensure that $\Theta_{j}$ is sufficiently small, for $j=1,\ldots,m$, we choose
	\begin{equation}\label{eq:def_d_ij}
		\begin{aligned}
			d_j = \sqrt{\frac 1 8 e^{\varrho(\xi_j)H^g(\xi_j,\xi_j)+\sum_{l\neq j}\varrho(\xi_l)G^g(\xi_j,\xi_l) +\log V_1(\xi_j)-\frac 12 z(\xi_j,\xi)}}.
		\end{aligned}
	\end{equation}

	
	\section{Finite dimensional reduction}
	\subsection{The linearized problem}
	Utilizing the Moser-Trudinger type inequality on compact Riemann surfaces, as in~\cite{yang2006extremal}, we have 
	\[\sup _{\int_{\Sigma}|\nabla_g u|^2 d v_g=1, \int_{\Sigma} u d v_g=0} \int e^{2\pi u^2} dv_g <+\infty.
	\]
	It follows that 	
	\begin{eqnarray*}
		\log \int_{\Sigma} e^u dv_g &\leq& \log \int_{\Sigma}  e^{ 2\pi\frac{ u^2}{\|u\|^2} + \frac 1 {8\pi}\|u\|^2  } dv_g \quad \text{(by Young's Inequality)}\\
		&=& 
		\frac 1 {8\pi} \int_{\Sigma} |\nabla_g u|^2 dv_g  +C, \text{ for any } u\in \overline{\mathrm{H}}^1, 
	\end{eqnarray*}
	where $C>0$ is a constant. 
	Consequently, $\overline{\mathrm{H}}^1\rightarrow L^p(\Sigma), u\mapsto e^u$ is continuous.  
	For any $p>1$, let $i^*_p: L^p(\Sigma)\rightarrow \overline{\mathrm{H}}^1$  be the adjoint operator corresponding to the immersion $i: \overline{\mathrm{H}}^1 \rightarrow L^{\frac{p}{p-1}}$ and $i^* : \cup_{p>1} L^p(\Sigma)\rightarrow \overline{\mathrm{H}}^1$. For any $f\in L^p(\Sigma)$, we define that 
	$i^*(f):=i^*( f-\bar{f})$, i.e. for any $h\in \overline{\mathrm{H}}^1 $, 
	$\langle i^*(f), h\rangle =\int_{\Sigma}(f-\bar{f}) h dv_g.$
	Let 
	\begin{equation}
		\label{eq:def_f12} f_1(u_1)= 2\lambda V_1 e^{u_1}-\overline{2\lambda V_1 e^{u_1}} \text{ and } f_2(u_2)=2\rho_2 \left(\frac{V_2 e^{u_2}}{\int_{\Sigma}V_2 e^{u_2}\, dv_g} -1\right)
	\end{equation}
	and 
	\begin{equation}
		\label{eq:def_F}
		F(\bu)=\left(\begin{array}{c}
			f_1(u_1)-\frac 1 2 f_2(u_2)\\
			f_2(u_2)-\frac 1 2 f_1(u_1)
		\end{array}\right). 
	\end{equation}
	The problem~\eqref{eq:toda_equi} has the following equivalent form,
	\begin{equation} \label{eq:toda_equi2}
		\left\{\begin{array}{ll}
			\bu =(u_1,u_2)=i^*(F(\bu)), \\
			u\in \cH
		\end{array}\right.. 
	\end{equation}
	
	We consider the following linear operator associated with the problem~\eqref{eq:toda_equi}: 
	\begin{equation}
		\label{eq:linear_op}  \cL_{\xi,\lambda}(\bphi):=( L^1_{\xi,\lambda}(\bphi),L^2_{\xi,\lambda}(\bphi) ),
	\end{equation}
	where
	\begin{eqnarray*}
		L^1_{\xi,\lambda}(\bphi):&=& -\Delta_g\phi_1 -\sum_{j=1}^m\left( \chi_j e^{-\varphi_j} e^{U_j} \phi_1-\overline{\chi_j e^{-\varphi_j} e^{U_j} \phi_i}\right)\\
		&& +\rho_2 \left( \frac{V_2 e^{W_{2,\lambda}} \phi_2 } {\int_{\Sigma} V_2e^{W_{2,\lambda}} \, dv_g}  - 
		\frac{V_2 e^{W_{2,\lambda}} \int_{\Sigma} V_2 e^{W_{2,\lambda}}  \phi_2\, dv_g }   {\left(\int_{\Sigma} V_2e^{W_{2,\lambda}} \, dv_g\right)^2 }\right) 
	\end{eqnarray*}
	and 
	\begin{eqnarray*}
		L^2_{\xi,\lambda}(\bphi):&=& -\Delta_g\phi_2 -2\rho_2 \left( \frac{V_2 e^{W_{2,\lambda}} \phi_2 } {\int_{\Sigma} V_2e^{W_{2,\lambda}} \, dv_g}  - 
		\frac{V_2 e^{W_{2,\lambda}} \int_{\Sigma} V_2 e^{W_{2,\lambda}}  \phi_2\, dv_g }   {\left(\int_{\Sigma} V_2e^{W_{2,\lambda}} \, dv_g\right)^2 }\right) \\
		&&+\frac 1 2 \sum_{j=1}^m\left( \chi_j e^{-\varphi_j} e^{U_j} \phi_1-\overline{\chi_j e^{-\varphi_j} e^{U_j} \phi_i}\right).
	\end{eqnarray*}
	Formally, for $i=1,2, j=1,\ldots,m$, we deduce that the limiting operator of 
	$L^i_{\xi,\lambda}$  is given by
	$$-\Delta \phi -\frac{8}{(1+|y|^2)^2}\phi,$$  
	through appropriate scaling around $\xi_j$ in an isothermal chart (for details, see Lemma~\ref{lem:invertible}).
	It is well known that 
	the kernel space is generated by (refers to~\cite{DelPino2012Nondegeneracy,baraket_construction_1997}, for instance) 
	\[z^0(y):=\frac{1-|y|^{2}}{1+|y|^{2}},  z^i(y):=\frac{4y_i}{1+|y|^2}, i=1,\ldots, \i(\xi_j). \] 
	We define that 
	\[ Z^0_j=	\left\{\begin{array}{ll}2 \frac{ \delta_j^2-|y_{\xi_j}|^2}{\delta_j^2+|y_{\xi_j}|^2} & \text{ in } U_{4r_0}(\xi_j)\\
		0& \text{ in } \Sigma\setminus U_{4r_0}(\xi_j)
	\end{array}\right., Z^i_j=	\left\{\begin{array}{ll} \frac{ 4(y_{\xi_j})_i}{\delta_j^2+|y_{\xi_j}|^2} & \text{ in } U_{4r_0}(\xi_j)\\
		0& \text{ in } \Sigma\setminus U_{4r_0}(\xi_j)
	\end{array}\right.,\]
	for any $j=1,\ldots, m$ and $i=1,\ldots, \i(\xi_j)$ and then projects $Z^i_j$ into the Hilbert space $\oH$ by following equations: 
	\begin{equation}
		\left\{ \begin{array}{ll}
			-\Delta_g PZ^i_j= \chi_j e^{-\varphi_j}e^{U_j} Z^i_j -\overline{\chi_j e^{-\varphi_j}e^{U_j} Z^i_j} & 
			\text{ in }\intsigma\\
			\partial_{\nu_g}  PZ^i_j =0& \text{ on } \partial\Sigma\\
			\int_{\Sigma} PZ^i_j =0 &
		\end{array}\right.. 
	\end{equation}
	Define the subspace \[ K_{\xi}= span \{ PZ^i_j: j=1,\ldots,m, i=1,\ldots, \i(\xi_j)\}\times \{ 0\}. \]	
	To ensure the invertibility of the linear operator 
	$\cL_{\xi,\lambda}$, we confine the error term $\bphi$ to the orthogonal complement of $\in K_{\xi}$, denoted by $K_{\xi}^{\perp}$, where 
	\[  K_{\xi}^{\perp}:=\left\{ \bphi\in\cH: \int_{\Sigma}\lan \phi_i, h_i\ran_g\, dv_g=0 \text{ for any } \bh=(h_1,h_2)\in K_{\xi}\right\}. \]
	Furthermore, we introduce the orthogonal projections 
	$\Pi_{\xi}: \cH\rightarrow K_{\xi}$ and $\Pi_{\xi}^{\perp}:\cH\rightarrow K_{\xi}.$
	
	\begin{lem}
		\label{lem:invertible}
		Let $\Dc$ be a compact subset of $\Xi_{k,m}$. 
		For any $p>1$,  
		there exist $\lambda_0 > 0$ and $C > 0$ such that for any $\lambda \in (0, \lambda_0)$,any $\xi\in \Dc$, any $\bh=(h_1,h_2) \in (L^p(\Sigma))^2$ and   $\bphi=(\phi_1, \phi_2) \in (W^{2,p}(\Sigma))^2 \cap K_{\xi}^{\perp}$ is the unique solution of 
		\begin{equation}
			\label{eq:linear_key} \left\{ \begin{array}{ll}
				\cL_{\xi,\lambda}(\bphi)=\bh & \text{ in }\intsigma\\
				\partial_{\nu_g} \bphi =0 & \text{ on  }\partial\Sigma\\
				\int_{\Sigma}\bphi\, dv_g=0 
			\end{array}\right.,
		\end{equation}
		the following estimate holds 
		\[
		\|\bphi\| \leq C |\log \lambda| \|\bh\|_p,
		\]
		where $\|\bh\|_p=\|h_1\|_{L^p(\Sigma)} + \|h_2\|_{L^p(\Sigma)}$.
	\end{lem}
	\begin{proof}
		We will prove it by contradiction. Suppose Lemma~\ref{lem:invertible} fails, i.e., there exist $p>1$ and a sequence of $\lambda_n\rightarrow 0, \xi^n\rightarrow \xi^*\in \Xi_{k,m}$ and $\bh_n:=(h_{1,n}, h_{2,n})\in (L^p(\Sigma))^2$ and $\bphi_n:=(\phi_{1,n},\phi_{2,n})\in  (W^{2,p}(\Sigma))^2\cap K_{\xi}^{\perp}$ solving  ~\eqref{eq:linear_key} for $\bh_n$ such that 
		\begin{equation}
			\label{eq:asumme_linear} \|\bphi_n\|=1\text{ and } |\log \lambda_n|\|\bh_n\|_p:=|\log\lambda_n|\sum_{i=1}^2 \|h_i\|_{L^p(\Sigma)}\rightarrow 0 \quad(n\rightarrow +\infty).
		\end{equation}  
		For simplicity, we still use the notations $\phi_i, h_i,\xi,\lambda$ instead of $\phi_{i,n},h_{i,n},\xi^n, \lambda_n$ for $i=1,2$. 
		Let $\cR_{\xi}=	\left\{\begin{array}{ll}
			\R^2 &\text{ if }\xi\in \intsigma\\
			\R^2_+&\text{ if }\xi\in \partial\Sigma
		\end{array}\right..$
		We define that for $i=1,2, j=1,\ldots,m$
		\[ \tilde{\phi}_{1j}(y)= 	\left\{\begin{array}{ll}
			\chi(\delta_j |y|/ r_0)\phi_i\circ y_{\xi_j}^{-1}(\delta_j y),&  y \in \Omega_j:= \frac{1}{\delta_j} B^{\xi_j}_{2r_0}\\
			0&  y\in\cR_{\xi_j}\setminus\Omega_j
		\end{array}\right.. \]
		Then we consider the following spaces for  $\xi\in\Sigma$
		\[ \rL_{\xi}:=\left\{ u: \left\| \frac{ u}{1+|y|^{2}} \right\|_{L^2(\cR_{\xi})} <+\infty\right\}\]
		and 
		\[ \rH_{\xi}:=\left\{u: \|\nabla u\|_{L^2(\cR_{\xi})}+\left\| \frac{u}{1+|y|^{2}}  \right\|_{L^2(\cR_{\xi})}<+\infty \right\}.\]
		\begin{itemize}
			\item[Step 1.]\label{item:step1} {\it As $\lambda\rightarrow 0$, for any $j=1,\ldots,m,$ $$\tilde{\phi}_{1j}\rightarrow a_{1j} \frac{1-|y|^{2}}{1+|y|^{2}}, $$
				for some $a_{1j}\in\R$ weakly in $\rH_{\xi_j}$ and strongly in $\rL_{\xi_j}$ and  $$\tilde{\phi}_2\rightarrow 0$$ weakly in $\oH$ and strongly in $L^q(\Sigma)$ for any $q\geq 2$.}
		\end{itemize}
		We first estimate the second component $\tilde{\phi}_{2j}$. 
		Let $\psi_2\in C^{\infty}_c(\Sigma\setminus \{\xi_1^*,\ldots,\xi^*_m\})$ with $\int_{\Sigma}\psi_2 \, dv_g=0$. We use $\psi_2$ a test function for the equation $L^2_{\xi,\lambda}(\bphi)=h_2$. Then it follows that 
		\begin{eqnarray}\label{eq:inner_psi_2}
			&&	\int_{\Sigma} \lan\nabla \phi_2,\nabla\psi_2\ran_g\,dv_g -2\rho_2 \int_{\Sigma} \left( \frac{V_2 e^{W_{2,\lambda}} \phi_2 } {\int_{\Sigma} V_2e^{W_{2,\lambda}} \, dv_g}  - 
			\frac{V_2 e^{W_{2,\lambda}} \int_{\Sigma} V_2 e^{W_{2,\lambda}}  \phi_2\, dv_g }   {\left(\int_{\Sigma} V_2e^{W_{2,\lambda}} \, dv_g\right)^2 }\right) \psi_2\, dv_g \\
			&&+\frac 1 2 \sum_{j=1}^m\int_{\Sigma} \chi_j e^{-\varphi_j} e^{U_j} \phi_1\psi_2\, dv_g = \int_{\Sigma} h_2\psi_2\, dv_g.\nonumber
		\end{eqnarray}
		The asymptotic expansion for~\eqref{eq:expansion_W2} implies that 
		\begin{eqnarray}
			\label{eq:expansion_eW2}
			V_2	e^{W_{2,\lambda}}&=& V_2 e^{z(\cdot,\xi)-\frac 1 2\sum_{j=1}^m\chi_j\log\frac{1}{(\delta_j^2+|y_{\xi_j}|^2)^2} -\frac 12 \sum_{j=1}^m \varrho(\xi_j) H^g(\cdot,\xi_j) +  \mathcal{O}(\lambda|\log \lambda|)}\\
			&=&\tilde{V}_2(\cdot,\xi) e^{z(\cdot,\xi)}+ \mathcal{O}(\lambda|\log\lambda|),\nonumber
		\end{eqnarray}
		in view of $G^g(\cdot,\xi_j)= H^g(\cdot,\xi_j)+\frac{4}{\varrho(\xi_j)}\chi_j\log \frac 1 {|y_{\xi_j}|}$. 
		By the assumption  \eqref{eq:asumme_linear}, we have  $\phi_1 \rightharpoonup \phi_1^*$ and $\phi_2 \rightharpoonup \phi_2^*$ weakly in $\oH$ and strongly in $L^q(\Sigma)$ for any $q \geq 2$.
		The Moser-Trudinger inequality yields that 
		\[ \left| \int_{\Sigma}h_2\phi_2\right|\leq \|\bh\|_p \|\phi_2\|_{L^{p'}(\Sigma)}\leq \|\bh\|_p \|\phi_2\|=o(|\log\lambda|)\rightarrow 0,\]
		where $p, p'>1$ with $\frac 1 p+\frac 1 {p'}=1. $ 
		Passing the limit of \eqref{eq:inner_psi_2}, \eqref{eq:expansion_W2}, \eqref{eq:expansion_eW2} and~\eqref{eq:asumme_linear} yield that  
		\begin{eqnarray*}
			&&	\int_{\Sigma} \lan \nabla\phi_2^*, \nabla\psi_2\ran_g \,dv_g \\
			& =&2\rho_2 \left( \int_{\Sigma}\frac{\tilde{V}_2(\cdot,\xi) e^{z(\cdot,\xi)}\psi_2}{\int_{\Sigma} \tilde{V}_2(\cdot,\xi) e^{z(\cdot,\xi)}\, dv_g }\, dv_g - \frac{\int_{\Sigma} \tilde{V}_2(\cdot,\xi) e^{z(\cdot,\xi)}\psi_2 \, dv_g \int_{\Sigma}\int_{\Sigma} \tilde{V}_2(\cdot,\xi) e^{z(\cdot,\xi)}\phi_2^*\, dv_g  }{(\int_{\Sigma} \tilde{V}_2(\cdot,\xi) e^{z(\cdot,\xi)}\, dv_g )^2}\right)    
		\end{eqnarray*}
		By assumption, $\|\phi_2^*\|\leq 1$. Then,  $\phi_2^* \in \oH$ solves the problem~\eqref{eq:linear_u_2}. Through the hypothesis \hyperref[item:H]{(C1)}, we obtain that $\phi^*_2=0.$\par 
		Applying $\phi_1$ as a test function for $L^1_{\xi,\lambda}(\bphi)- L^2_{\xi,\lambda}(\bphi)= h_1-h_2$, since $\int_{\Sigma}\phi_1\, dv_g=0,$ \begin{eqnarray}\label{eq:diff_h}
			\int_{\Sigma} (h_1-h_2)\phi_1 \, dv_g 	&=&	\int_{\Sigma} \lan \nabla (\phi_1-\phi_2), \nabla\phi_1\ran_g \, dv_g -\frac 3 2\sum_{j=1}^m\int_{\Sigma} \chi_je^{-\varphi_j}  e^{U_j} \phi_1^2 \, dv_g \\
			&&+3 \rho_2 \int_{\Sigma} \left( \frac{V_2 e^{W_{2,\lambda}} \phi_2 } {\int_{\Sigma} V_2e^{W_{2,\lambda}} \, dv_g}  - 
			\frac{V_2 e^{W_{2,\lambda}} \int_{\Sigma} V_2 e^{W_{2,\lambda}}  \phi_2\, dv_g }   {\left(\int_{\Sigma} V_2e^{W_{2,\lambda}} \, dv_g\right)^2 }\right) \phi_1\, dv_g  .\nonumber
		\end{eqnarray}
		The Moser-Trudinger inequality yields that 
		\[ \left| \int_{\Sigma}(h_1-h_2)\phi_1\right|\leq \|\bh\|_p \|\phi_1\|_{L^{p'}(\Sigma)}\leq \|\bh\|_p \|\phi_1\|=o(|\log\lambda|)\rightarrow 0,\]
		where $p, p'>1$ with $\frac 1 p+\frac 1 {p'}=1. $  $\left|\int_{\Sigma}\lan\nabla \phi_1,\nabla\phi_2\ran_g\, dv_g\right|\leq \|\phi_1\|\|\phi_2\|\leq 1. $
		Hence, we derive that
		\begin{equation}\label{eq:1}
			\sum_{j=1}^m \int_{\Sigma} \chi_j e^{-\varphi_j}  e^{U_j} \phi_1^2 \, dv_g = 2\rho_2 \int_{\Sigma} \left( \frac{V_2 e^{W_{2,\lambda}} \phi_2 } {\int_{\Sigma} V_2e^{W_{2,\lambda}} \, dv_g}  - 
			\frac{V_2 e^{W_{2,\lambda}} \int_{\Sigma} V_2 e^{W_{2,\lambda}}  \phi_2\, dv_g }   {\left(\int_{\Sigma} V_2e^{W_{2,\lambda}} \, dv_g\right)^2 }\right) \phi_1\, dv_g  +\mathcal{O}(1). 
		\end{equation}
		Applying 
		$\phi_1$ as a test function for $ L^1_{\xi,\lambda}(\bphi)=h_1$, we can deduce that 
		\begin{eqnarray}\label{eq:test_phi1}
			\int_{\Sigma} h_1\phi_1 \, dv_g 	&=&	\int_{\Sigma} \lan \nabla\phi_1,  \nabla\phi_1\ran_g \, dv_g -\sum_{j=1}^m\int_{\Sigma} \chi_je^{-\varphi_j}  e^{U_j} \phi_1^2 \, dv_g \\
			&&+\rho_2 \int_{\Sigma} \left( \frac{V_2 e^{W_{2,\lambda}} \phi_2 } {\int_{\Sigma} V_2e^{W_{2,\lambda}} \, dv_g}  - 
			\frac{V_2 e^{W_{2,\lambda}} \int_{\Sigma} V_2 e^{W_{2,\lambda}}  \phi_2\, dv_g }   {\left(\int_{\Sigma} V_2e^{W_{2,\lambda}} \, dv_g\right)^2 }\right) \phi_1\, dv_g, \nonumber
		\end{eqnarray}
		in view of $\int_{\Sigma}\phi_1=0$. 
		Similarly, we have 
		\begin{eqnarray*}
			\sum_{j=1}^m \int_{\Sigma} \chi_je^{-\varphi_j} e^{U_j} \phi_1^2 \, dv_g& =&\rho_2 \int_{\Sigma} \left( \frac{V_2 e^{W_{2,\lambda}} \phi_2 } {\int_{\Sigma} V_2e^{W_{2,\lambda}} \, dv_g}  - 
			\frac{V_2 e^{W_{2,\lambda}} \int_{\Sigma} V_2 e^{W_{2,\lambda}}  \phi_2\, dv_g }   {\left(\int_{\Sigma} V_2e^{W_{2,\lambda}} \, dv_g\right)^2 }\right) \phi_1\, dv_g\\
			&& +\mathcal{O}(1)\\
			&\stackrel{\eqref{eq:1}}{=}&\frac 1 2 \sum_{j=1}^m \int_{\Sigma} \chi_j e^{-\varphi_j}  e^{U_j} \phi_1^2+\mathcal{O}(1). 
		\end{eqnarray*}
		Consequently, 
		\begin{equation}
			\label{eq:bouned_exp_phi}
			\sum_{j=1}^m \int_{\Sigma} \chi_j e^{-\varphi_j}  e^{U_j} \phi_1^2 \, dv_g=\mathcal{O}(1).
		\end{equation}
		By a straightforward calculation, 
		\[\sum_{j=1}^m \int_{\cR_{\xi_j}}\frac{1}{(1+|y|^2)^2}(\tilde{\phi}_{1j}(y))^2\,dy  =	\sum_{j=1}^m \int_{\Sigma} \chi_j e^{-\varphi_j}  e^{U_j} \phi_1^2 \, dv_g=\mathcal{O}(1).\]
		By the assumption $\|\bphi\|=1$, we immediately  have 
		\[ \int_{\cR_{\xi_j}} |\nabla \tilde{\phi}_{1j}|^2 \leq \int_{\Sigma} |\nabla\phi_1|_g^2 \, dv_g\leq 1.  \]
		It follows that $\tilde{\phi}_{1j}$ is uniformly bounded in $\rH_{\xi_j}$. Applying Proposition~A.1 in~\cite{musso_new_2016},  $\rL_{\xi_j}\hookrightarrow \rH_{\xi_j}$ is a compact embedding, then  up to a subsequent that
		\[ \tilde{\phi}_{1j}  \rightharpoonup \tilde{\phi}_{1j}^*, \]
		which is weakly convergent in $\rH_{\xi_j}$ and strongly in $\rL_{\xi_j}$. 
		For any $q>1$, 
		\begin{equation}
			\label{eq:L_p_exp}
			\int_{\Sigma} \left|\chi_j e^{-\varphi_j}  e^{U_j}  \right|^q \,dv_g = \int_{U_{2r_0}(\xi_j)}\frac{ \chi^q(|y|/r_0) \delta_j^{2-2q} }{(1+|y|^2)^{2q}} dy =\mathcal{O}(\delta_j^{2(1-q)}). 
		\end{equation}
		For any $\varphi\in C_c^{\infty}(\cR_{\xi_j})$, assume that  $$\text{ supp }\varphi\subset B_{R_0}(0).$$
		If $\delta_j <\frac {r_0}{R_0}$, then 
		$\text{ supp }\nabla \chi\left(\frac{|y|}{r_0}\right) \cap \text{{ supp }} \varphi \left(\frac{1}{\delta_j} y\right) =\emptyset. $
		For any $\Phi\in \oH$, 
		\begin{eqnarray}
			~\label{eqp1}
			0&=&\int_{B_{2r_0}^{\xi_i}} \Phi\circ y_{\xi_j}^{-1}(y)\nabla \chi\left(\frac{|y|}{r_0}\right) \cdot \nabla  \varphi\left(\frac{1}{\delta_j} y\right) dy\\
			&=&
			\int_{B_{2r_0}^{\xi_i}} h\left(\frac{1}{\delta_j} y\right)\nabla( \Phi\circ y_{\xi_j}^{-1}(y))\cdot \nabla \chi\left(\frac{|y|}{r_0}\right)   dy. \nonumber
		\end{eqnarray}
		We observe that for any $q>1$ 
		\begin{eqnarray}~\label{eqp2}
			\int_{\Sigma} \chi_j\varphi^q\left(\frac{1}{\delta_j} y_{\xi_j}(x) \right) \,dv_g(x)= O(\delta_j^q) \quad(\delta_j\rightarrow 0)
		\end{eqnarray}
		{and }
		\begin{eqnarray}\label{eqp3}
			\int_{\Sigma} \chi_j \left |\nabla \varphi\left(\frac{1}{\delta_j} y_{\xi_j}(x) \right) \right|_g^2 \,  dv_g(x)= \mathcal{O}(1)\quad  (\delta_j\rightarrow 0). 
		\end{eqnarray}
		Applying~\eqref{eq:expansion_eW2}, we have 
		\begin{eqnarray}
			\label{eq:linear_phi_2}
			&&	2\rho_2\left( \frac{V_2 e^{W_{2,\lambda}} \phi_2 } {\int_{\Sigma} V_2e^{W_{2,\lambda}}  \, dv_g}  - 
			\frac{V_2 e^{W_{2,\lambda}} \int_{\Sigma} V_2 e^{W_{2,\lambda}}  \phi_2\, dv_g }   {\left(\int_{\Sigma} V_2e^{W_{2,\lambda}} \, dv_g\right)^2 }\right) \\
			&=& 2\rho_2\left( \frac{\tilde{V}_2(\cdot,\xi) e^{z(\cdot,\xi)}\phi_2 }{\int_{\Sigma} \tilde{V}_2(\cdot,\xi) e^{z(\cdot,\xi)}\, dv_g }- \frac{ \tilde{V}_2(\cdot,\xi) e^{z(\cdot,\xi)} \int_{\Sigma} \tilde{V}_2(\cdot,\xi) e^{z(\cdot,\xi)}\phi_2\, dv_g  }{(\int_{\Sigma} \tilde{V}_2(\cdot,\xi) e^{z(\cdot,\xi)}\, dv_g )^2}\right) +\mathcal{O}(\lambda|\log\lambda|)\nonumber.
		\end{eqnarray}
		It follows that 
		\begin{eqnarray}
			\label{eqp4}
			&&	2\rho_2\int_{\Sigma}\left( \frac{V_2 e^{W_{2,\lambda}} \phi_2 } {\int_{\Sigma} V_2e^{W_{2,\lambda}}  \, dv_g}  - 
			\frac{V_2 e^{W_{2,\lambda}} \int_{\Sigma} V_2 e^{W_{2,\lambda}}  \phi_2\, dv_g }   {\left(\int_{\Sigma} V_2e^{W_{2,\lambda}} \, dv_g\right)^2 }\right) \chi_j \varphi\left(\frac 1 {\delta_j} y_{\xi_j}\right)\, dv_g
			\\
			&	=& \mathcal{O}(\delta_j^{2(1-\frac 1 {q})} \|\phi_2\|_{L^q(\Sigma)}+ \lambda|\log\lambda|)=o(1)\nonumber,
		\end{eqnarray}
		in which the last inequality, we applied that 
		$\phi_2\rightarrow 0$ in $L^q(\Sigma)$ for any $q\geq 2. $
		Assume that $0<\lambda< \frac{r^2_0}{d_j^2R^2_0}$, as $\lambda\rightarrow 0$, \eqref{eq:L_p_exp}-\eqref{eqp4} imply that 
		\begin{eqnarray*}
			&&\int_{\cR_{\xi_j}} \nabla\tilde{\phi}_{1j} \nabla \varphi\,  dy =
			\int_{B_{2r_0}^{\xi_j}} \nabla\left( \chi\left(\frac{|y|}{r_0}\right) \phi_1\circ y_{\xi_j}^{-1}(y)\right)\cdot\nabla \varphi\left(\frac{1}{\delta_j} y\right) \,dy\\
			&=&  \int_{B_{2r_0}^{\xi_j}} \nabla \phi_1\circ y_{\xi_j}^{-1}(y)\cdot\nabla\left( \chi\left(\frac{|y|}{r_0}\right)  \varphi\left(\frac{1}{\delta_j} y\right) \right)\, dy \\
			&=& \int_{\Sigma} \left \lan\nabla \phi_1, \nabla \left( \chi_j \varphi\left(\frac{1}{\delta_j} y_{\xi_j}\right)\right) \right\ran_g dv_g \\
			&=&  \int_{\Sigma} \chi_j e^{-\varphi_j} e^{U_j} \phi_1 \varphi \left(\frac{1}{\delta_j} y_{\xi_j} \right)dv_g- \int_{\Sigma} \chi_j e^{-\varphi_j} e^{U_j} \phi_1 dv_g \int_{\Sigma}\chi_j \varphi\left(\frac{1}{\delta_j} y_{\xi_j} \right)dv_g \\
			&&- \rho_2 \int_{\Sigma} \left( \frac{V_2 e^{W_{2,\lambda}} \phi_2 } {\int_{\Sigma} V_2e^{W_{2,\lambda}}  \, dv_g}  - 
			\frac{V_2 e^{W_{2,\lambda}} \int_{\Sigma} V_2 e^{W_{2,\lambda}}  \phi_2\, dv_g }   {\left(\int_{\Sigma} V_2e^{W_{2,\lambda}} \, dv_g\right)^2 }\right) \chi_j \varphi\left(\frac{1}{\delta_j} y_{\xi_j} \right) \, dv_g \\
			&&+ \int_{\Sigma}  h_1 \left( \chi_j \varphi\left(\frac{1}{\delta_j} y_{\xi_j}\right)\right)  \, dv_g \\
			&=&\int_{\cR_{\xi_j}} \frac 8{(1+|y|^2)^2}\tilde{\phi}_{1j}(y)\varphi(y)\, dy+ \mathcal{O}(\delta_j^{1+2(1-q)/q})+o(1) +o(|\log \lambda|)\\
			&& \text{\small (by the H\"{o}lder inequality)}\\
			&=& \int_{\cR_{\xi_j}} \frac 8{(1+|y|^2)^2}\tilde{\phi}_{1j}(y)\varphi(y)\, dy+o(1), 
		\end{eqnarray*}
		for $q>1$ sufficiently close to $1$ such that  $1+ 2(1-q)/q>0$. 
		Thus, 
		$\tilde{\phi}_{1j}$ converges to the solution of 
		\begin{equation}
			\label{eq:limit_linear}
			\left\{\begin{array}{ll}
				-\Delta \phi= \frac{8}{\left(1+|y|^2\right)^2} \phi & \text { in }\cR_{\xi_j}\\
				\partial_{y_2} \phi=0 &\text{ on } \partial\cR_{\xi_j}\\
				\int_{\cR_{\xi_j}}|\nabla \phi(y)|^2 d y<+\infty &
			\end{array}\right.,
		\end{equation}
		in the distribution sense. 
		According to the regularity theory,
		$\phi^*_{1j}$ is a smooth solution on the whole space $\cR_{\xi_j}$, for any $j=1,\ldots,m.$
		The result \cite{DelPino2012Nondegeneracy} implies that the solutions space of \eqref{eq:limit_linear} is generated by the following functions 
		\[ z^0(y):=\frac{1-|y|^2}{1+|y|^2}, \quad z^i(y):=\frac{4y_i}{1+|y|^2} \text{ for }i=1,\ldots, \i(\xi_j).\]
		Since $\bphi\in K_{\xi}^{\perp}$, 
		for any $j=1,\ldots,m$ and $j=1,\ldots, \i(\xi_j)$, 
		\begin{eqnarray*}
			&&\frac{32}{\delta_j}	\int_{\cR_{\xi_j}} \frac{y_j}{  (|y|^2+1)^3} \tilde{\phi}^*_{1j}dy = \lim_{\lambda\rightarrow 0} \frac{32}{\delta_j}
			\int_{\Omega_{j}} \frac{y_j}{  (|y|^2+1)^3}  \tilde{\phi}_{1j}\chi(\delta_j|y|/r_0) dy \\
			&=&\lim_{\lambda\rightarrow 0}32  \int_{B_{2r_0}^{\xi_j}} \frac{y_j}{(|y|^2 +\delta_j^2)^3} \phi_i\circ  y^{-1}_{\xi_j}(y)\chi\left(\frac{|y|}{r_0}\right) dy  \\
			&=& \lim_{\lambda\rightarrow 0}\int_{U_{2r_0}(\xi_i)} \chi_j e^{-\varphi_j} e^{U_j}  Z^i_j \phi_i dv_g =\lim_{\lambda\rightarrow 0} \lan PZ^i_j,\phi_i\ran=0. 
		\end{eqnarray*}
		Consequently, 
		$
		\int_{\cR_{\xi_j}} \frac{|y|^2-1}{(|y|^2+1)^{3}} \tilde{\phi}_{1j}^* dy  = \int_{\cR_{\xi_j} } \frac{y_j}{(|y|^2+1)^{3}} \tilde{\phi}_{1j}^* dy =0.
		$
		It follows that 
		\[ \tilde{\phi}^*_{1j}=a_{1j} \frac{1-|y|^2}{1+|y|^2}.\]
		\begin{itemize}
			\item [Step 2.]\label{item:step2} {\it For $j=1,\ldots
				,m,i=1,\ldots, \i(\xi_j),$ $$\int_{\Omega_j} \frac{ 8}{(1+|y|^2)^2}\tilde{\phi}_{1j}(y) \,dy =o(|\log\lambda|^{-1}).$$} 
		\end{itemize}
		Applying that $PZ^0_{j}$ as a test function of $L^1_{\xi,\lambda}(\bphi)=h_1$, we have
		\begin{eqnarray*}
			&&\int_{\Sigma} 8 \chi_je^{-\varphi_j}  \frac{\delta_j^2}{\left(\delta_j^2+|y_{\xi_j}|^2\right)^2} \phi_1 Z^0_{j}\, dv_g=\int_{\Sigma} 8 \chi_je^{-\varphi_j}  \frac{\delta_j^2}{\left(\delta_j^2+|y_{\xi_j}|^2\right)^2} \phi_i PZ^0_j\, dv_g \\
			&&-\rho_2 \int_{\Sigma} \left( \frac{V_2 e^{W_{2,\lambda}} \phi_2 } {\int_{\Sigma} V_2e^{W_{2,\lambda}}  \, dv_g}  - 
			\frac{V_2 e^{W_{2,\lambda}} \int_{\Sigma} V_2 e^{W_{2,\lambda}}  \phi_2\, dv_g }   {\left(\int_{\Sigma} V_2e^{W_{2,\lambda}} \, dv_g\right)^2 }\right)  PZ^0_j\, dv_g +\int_{\Sigma}h_1 PZ^0_j\,dv_g. 
		\end{eqnarray*}
		It follows that 
		\begin{eqnarray*}
			0&=&\int_{\Sigma} 8 \chi_je^{-\varphi_j}  \frac{\delta_j^2}{\left(\delta_j^2+|y_{\xi_j}|^2\right)^2} \phi_i\left( PZ^0_j-Z^0_j\right) \, dv_g \\
			&&-\rho_2 \int_{\Sigma} \left( \frac{V_2 e^{W_{2,\lambda}} \phi_2 } {\int_{\Sigma} V_2e^{W_{2,\lambda}}  \, dv_g}  - 
			\frac{V_2 e^{W_{2,\lambda}} \int_{\Sigma} V_2 e^{W_{2,\lambda}}  \phi_2\, dv_g }   {\left(\int_{\Sigma} V_2e^{W_{2,\lambda}} \, dv_g\right)^2 }\right)PZ^0_j\, dv_g +\int_{\Sigma}h_1 PZ^0_j\,dv_g.  
		\end{eqnarray*}
		By Lemma~\ref{lem:extension_PZ},  we have $\|PZ^0_j\|_{L^{p'}(\Sigma)}=\mathcal{O}(1)$, for $\frac 1 p+\frac 1 {p'}=1$. Then by the assumption $\|\bh\|_p=o(|\log \lambda|^{-1}),$
		\[ \left|\int_{\Sigma}h_1 PZ^0_j\,dv_g\right|\leq\|PZ_i\|_{L^{p'}(\Sigma)}\|h_1\|_{L^p(\Sigma)}=o(|\log \lambda|^{-1}). \]
		Lemma~\eqref{lem:extension_PZ} with~\eqref{eq:linear_phi_2} implies that 
		\begin{eqnarray*}
			&&\rho_2 \int_{\Sigma} \left( \frac{V_2 e^{W_{2,\lambda}} \phi_2 } {\int_{\Sigma} V_2e^{W_{2,\lambda}}  \, dv_g}  - 
			\frac{V_2 e^{W_{2,\lambda}} \int_{\Sigma} V_2 e^{W_{2,\lambda}}  \phi_2\, dv_g }   {\left(\int_{\Sigma} V_2e^{W_{2,\lambda}} \, dv_g\right)^2 }\right)PZ^0_j\, dv_g\\
			&=&  \mathcal{O}\left(\delta_j^{2(1-\frac 1 q)}|\log\delta_j| \|\phi_2\|_{L^q(\Sigma)} +\lambda|\log\lambda|\right)=o(\lambda^{\frac 1 2}|\log\lambda|)\nonumber,
		\end{eqnarray*}
		for $q=2.$
		Applying Lemma~\ref{lem:extension_PZ} again, we derive that 
		\begin{eqnarray*}
			&&	(\log \lambda)\int_{\Sigma} 8 \chi_je^{-\varphi_j}  \frac{\delta_j^2}{\left(\delta_j^2+|y_{\xi_j}|^2\right)^2} \phi_i\left( PZ^0_j-Z^0_j\right) \, dv_g\\ &=& (\log \lambda)
			\int_{\Sigma} 8 \chi_je^{-\varphi_j}  \frac{\delta_j^2}{\left(\delta_j^2+|y_{\xi_j}|^2\right)^2} \phi_i\left(1 +\mathcal{O}(\delta_j^2|\log \delta_j|) \right) \, dv_g\\
			&=&(1+\mathcal{O}(\delta_j^2|\log \delta_j|))(\log\lambda)\int_{\Omega_j} \frac 8{\left(1+|y|^2\right)^2} \tilde{\phi}_{1j}(y) \, dy+\mathcal{O}(|\log \lambda|\delta_j^2)\\
			&=&   (\log\lambda)\int_{\Omega_j}  \frac{8}{\left(1+|y|^2\right)^2} \tilde{\phi}_{1j}(y) \, dy+o(1), \quad \lambda\rightarrow 0.
		\end{eqnarray*}
		Thus, we have  $(\log\lambda) \int_{\Omega_j} \frac 8{\left(1+|y|^{\alpha_{i}}\right)^2} \tilde{\phi}_{1j}(y)\, dy=o(1)$ as $\lambda\rightarrow 0.$\\
		\begin{itemize}
			\item[Step 3.]\label{item:step3} {\it  Construct the contradiction.}
		\end{itemize}	
		Using $PU_j$ as a test function for~\eqref{eq:linear_key}, we derive that 
		\begin{eqnarray}\label{eq:test_PU_j}
			&&\int_{\Sigma} 8 \chi_je^{-\varphi_j}  \frac{\delta_j^2}{\left(\delta_j^2+|y_{\xi_j}|^2\right)^2} \phi_i\, dv_g=\int_{\Sigma} 8 \chi_je^{-\varphi_j}  \frac{\delta_j^2}{\left(\delta_j^2+|y_{\xi_j}|^2\right)^2} \phi_i PU_j\, dv_g \\
			&&-\rho_2 \int_{\Sigma} \left( \frac{V_2 e^{W_{2,\lambda}} \phi_2 } {\int_{\Sigma} V_2e^{W_{2,\lambda}}  \, dv_g}  - 
			\frac{V_2 e^{W_{2,\lambda}} \int_{\Sigma} V_2 e^{W_{2,\lambda}}  \phi_2\, dv_g }   {\left(\int_{\Sigma} V_2e^{W_{2,\lambda}} \, dv_g\right)^2 }\right)PU_j\, dv_g+\int_{\Sigma} h_i PU_j\, dv_g. \nonumber
		\end{eqnarray}
		The L.H.S. of~\eqref{eq:test_PU_j} 
		\begin{eqnarray*}
			\int_{\Sigma} 8 \chi_je^{-\varphi_j}  \frac{\delta_j^2}{\left(\delta_j^2+|y_{\xi_j}|^2\right)^2} \phi_i\, dv_g&=&\int_{\Omega_j} 8 \frac{ 1}{(1+|y|^2)^2}\tilde{\phi}_{1j}(y) \,dy+\mathcal{O}(\delta_j^2) \\
			&=& o(1),
		\end{eqnarray*}
		by the result of \hyperref[item:step2]{Step 2}, 
		Lemma~\ref{lem:extension_PU} yields that  $\|PU_j\|_{L^{\infty}(\Sigma)}= \mathcal{O}(|\log\lambda|)$. We drive that 
		\begin{eqnarray*}
			\left|\int_{\Sigma} h_i PU_j\, dv_g \right| \leq \|h_i\|_{L^p(\Sigma)} \|PU_j\|_{L^{p'}(\Sigma)}
			=o(1),
		\end{eqnarray*}
		where $\frac 1 p+\frac{1}{p'}=1.$
		By~\eqref{eq:linear_phi_2} and Lemma~\ref{lem:extension_PU},
		\begin{eqnarray*}
			&&	\rho_2 \int_{\Sigma} \left( \frac{V_2 e^{W_{2,\lambda}} \phi_2 } {\int_{\Sigma} V_2e^{W_{2,\lambda}}  \, dv_g}  - 
			\frac{V_2 e^{W_{2,\lambda}} \int_{\Sigma} V_2 e^{W_{2,\lambda}}  \phi_2\, dv_g }   {\left(\int_{\Sigma} V_2e^{W_{2,\lambda}} \, dv_g\right)^2 }\right)PU_j\, dv_g \\
			&=&\mathcal{O}\left(\delta_j^{\frac 1 2 }\|\phi_2\|_{L^q(\Sigma)}+\lambda|\log\lambda|\right)=o(1),
		\end{eqnarray*}
		for $q=2.$
		By a straightforward calculation, we have
		\begin{eqnarray*}
			\delta_j\int_{\Omega_j} \frac 8{\left(1+|y|^2 \right)^2}|\tilde{\phi}_{1j}(y)| |y| &\leq & 8 
			\delta_j \delta_j^{\frac{2(1-q)}{q}}\left\|\phi_i\right\|_{L^{q'}(\Sigma)}\left(\int_{\mathbb{R}^2}\left(\frac{1}{\left(1+|y|^2\right)^2}\right)^q d y\right)^{1 / q} \\
			&=&O\left(\delta_j^{\frac{2-q}{q}}\right)=o(1).
		\end{eqnarray*}
		where $q\in(1,2)$ such that $\frac 1 q +\frac 1 {q'}=1$. 
		Applying Lemma~\ref{lem:extension_PU}, \hyperref[item:step1]{Step 1} and \hyperref[item:step2]{Step 2},  we deduce that 
		\begin{eqnarray*}
			&&\int_{\Sigma} 8 \chi_je^{-\varphi_j}  \frac{\delta_j^2}{\left(\delta_j^2+|y_{\xi_j}|^2\right)^2} \phi_1 PU_j\, dv_g\\ &=&\int_{\Sigma} 8 \chi_je^{-\varphi_j}  \frac{\delta_j^2}{\left(\delta_j^2+|y_{\xi_j}|^2\right)^2} \phi_1 \left(-2\chi_j\log\left( \delta_j^2+ |y_{\xi_j}|^2 \right) +\varrho(\xi_j) H^g(\cdot,\xi_j)\right. \\
			&&\left. + \mathcal{O}(\delta_j^2|\log \delta_j|)\right) \, dv_g \\
			&=&
			\int_{\Omega_j} \frac 8{\left(1+|y|^2\right)^2} \tilde{\phi}_{1j}(y)\left(-4 \log \delta_j-2 \log \left(1+|y|^2\right)+\varrho(\xi_j)H^g(\xi_j,\xi_j)\right) d y+ \\
			&&+O\left(\int_{\Omega_j} \frac 8{\left(1+|y|^2 \right)^2}\left|\tilde{\phi}_{1j}(y)\right|\left(\delta_j|y|+\delta_j^2|\log\delta_j|\right) d y\right)+\mathcal{O}(\delta_j^2)\\
			&\rightarrow & -2 a_{ij} \int_{\cR_{\xi_j}} \frac 8{\left(1+|y|^2\right)^2}\frac{1- |y|^2}{1+|y|^2}\log(1+|y|^2)\, dy= \varrho(\xi_j)a_{1j}, \quad (\lambda\rightarrow 0). 
		\end{eqnarray*}
		in which the last equality used the fact that  $\int_{\R^2} \frac 8{\left(1+|y|^2\right)^2}\frac{1- |y|^2}{1+|y|^2}\log(1+|y|^2)\, dy=-4\pi.$
		
		Consequently,  $a_{ij}=0$ for any  $j=1,\ldots,m, i=1\ldots,\i(\xi_j).$
		We $\phi_i$ as test functions for $L^i_{\xi,\lambda}(\bphi)=h_i$.
		For $q\geq 2,$ since $\phi_2\rightarrow 0$ strongly in $L^q(\Sigma)$ and $\tilde{\phi}_{1j}\rightarrow 0$ strongly in $\rL_{\xi_j}$, 
		\begin{eqnarray*}
			&&1=\sum_{i=1}^2 \|\phi_i\|^2=\sum_{j=1}^m\int_{\Sigma} 8 \chi_je^{-\varphi_j}  \frac{\delta_j^2}{\left(\delta_j^{2}+|y_{\xi_j}|^2\right)^2} \phi_1 ^2\, dv_g \\
			&&-	\rho_2 \int_{\Sigma} \left( \frac{V_2 e^{W_{2,\lambda}} \phi_2 } {\int_{\Sigma} V_2e^{W_{2,\lambda}}  \, dv_g}  - 
			\frac{V_2 e^{W_{2,\lambda}} \int_{\Sigma} V_2 e^{W_{2,\lambda}}  \phi_2\, dv_g }   {\left(\int_{\Sigma} V_2e^{W_{2,\lambda}} \, dv_g\right)^2 }\right)\phi_1\, dv_g \\
			&&+ 2\rho_2 \left( \int_{\Sigma}\frac{V_2 e^{W_{2,\lambda}} \phi^2_2\, dv_g } {\int_{\Sigma} V_2e^{W_{2,\lambda}}  \, dv_g}  - 
			\frac{\left( \int_{\Sigma} V_2 e^{W_{2,\lambda}}  \phi_2\, dv_g \right)^2}   {\left(\int_{\Sigma} V_2e^{W_{2,\lambda}} \, dv_g\right)^2 }\right) 
			\\&& - \frac 1 2 \sum_{j=1}^m\int_{\Sigma} 8 \chi_je^{-\varphi_j}  \frac{\delta_j^2}{\left(\delta_j^{2}+|y_{\xi_j}|^2\right)^2} \phi_1 \phi_2\, dv_g+ \sum_{i=1}^2\int_{\Sigma}h_i \phi_i \,dv_g\\
			&\leq&  \sum_{j=1}^m \|\tilde{\phi}_{1j}\|^2_{\rL^2_{\xi}} +\mathcal{O}\left(\|\phi_2\|_{L^q(\Sigma)}\right)+o\left(\frac 1 {|\log \lambda|}\right)\rightarrow 0\quad(\lambda\rightarrow 0). 
		\end{eqnarray*}
	Therefore, we arrive at a contradiction.
	\end{proof}
	\subsection{Nonlinear problem}
	The expected solution $\bW_{\lambda}+\bphi_{\lambda}$ solves~\eqref{eq:toda_equi} if and only if 
	$\bphi_{\lambda}$ solves the following problem in $\cH_{\xi,r_0}$:
	\begin{equation}
		\label{eq:toda_nonlinear}
		\cL_{\xi,\lambda}(\bphi) =\cS_{\xi,\lambda}(\bphi)+ \cN_{\xi,\lambda}(\phi)+\cR_{\xi,\lambda}.
	\end{equation}
	Here, $\cL_{\xi,\lambda}$  is the linear operator defined by~\eqref{eq:linear_op},
	the higher order linear operator $ \mathcal{S}_{\xi,\lambda}(\bphi):=\left(S_{\xi,\lambda}^1(\bphi),  S_{\xi,\lambda}^2(\bphi)\right),$ where for 
	$$
	\begin{aligned}
		& S_{\xi,\lambda}^1(\bphi):= \left(-\sum_{j=1}^m\chi_j e^{-\varphi_j} e^{U_j}+2 \lambda V_1e^{ W_{1,\lambda}}\right) \phi_1-\overline{\left(-\sum_{j=1}^m\chi_j e^{-\varphi_j} e^{U_j}+2 \lambda V_1e^{ W_{1,\lambda}}\right) \phi_1} \quad\text{ and }
	\end{aligned}
	$$
	\[S_{\xi,\lambda}^2(\bphi)=-\frac 1 2 \left(-\sum_{j=1}^m\chi_j e^{-\varphi_j} e^{U_j}+2 \lambda V_1e^{ W_{1,\lambda}}\right) \phi_1+\frac 1 2\overline{\left(-\sum_{j=1}^m\chi_j e^{-\varphi_j} e^{U_j}+2 \lambda V_1e^{ W_{1,\lambda}}\right) \phi_1}, \]
	the nonlinear term $\mathcal{N}_{\xi,\lambda}(\bphi):=\left(N_{\xi,\lambda}^1(\bphi), N_{\xi,\lambda}^2(\bphi)\right)$, where 
	\begin{eqnarray*}
		N_{\xi,\lambda}^1(\bphi):&=&2 \lambda V_1 e^{W_{1,\lambda}}\left(e^{\phi_1}-1-\phi_1\right)-\overline{2 \lambda V_1  e^{W_{1,\lambda}}\left(e^{\phi_1}-1-\phi_1\right)}\\
		&&- \rho_2\left(  \frac{V_2e^{W_{2,\lambda}+\phi_2}}{\int_{\Sigma} V_2e^{W_{2,\lambda}+\phi_2} dv_g } -\frac{V_2e^{W_{2,\lambda}}\phi_2}{\int_{\Sigma} V_2e^{W_{2,\lambda}} dv_g } + \frac{V_2e^{W_{2,\lambda}} \int_{\Sigma}V_2e^{W_{2,\lambda}}\phi_2 dv_g}{(\int_{\Sigma}V_2e^{W_{2,\lambda}} dv_g )^2}-\frac{V_2e^{W_{2,\lambda}}}{\int_{\Sigma}V_2e^{W_{2,\lambda}} dv_g }\right),\\
		N_{\xi,\lambda}^2(\bphi):&=&2 
		\rho_2\left(  \frac{V_2e^{W_{2,\lambda}+\phi_2}}{\int_{\Sigma} V_2e^{W_{2,\lambda}+\phi_2} dv_g } -\frac{V_2e^{W_{2,\lambda}}\phi_2}{\int_{\Sigma} V_2e^{W_{2,\lambda}} dv_g } + \frac{V_2e^{W_{2,\lambda}} \int_{\Sigma}V_2e^{W_{2,\lambda}}\phi_2 dv_g}{(\int_{\Sigma}V_2e^{W_{2,\lambda}} dv_g )^2}-\frac{V_2e^{W_{2,\lambda}}}{\int_{\Sigma}V_2e^{W_{2,\lambda}} dv_g }\right)\\
		&&- \lambda V_1 e^{W_{1,\lambda}}\left(e^{\phi_1}-1-\phi_1\right)-\overline{2 \lambda V_1  e^{W_{1,\lambda}}\left(e^{\phi_1}-1-\phi_1\right)},
	\end{eqnarray*}
	and the error term $\cR_{\xi,\lambda}=(R^1_{\xi,\lambda}, R^2_{\xi,\lambda} )$, where 
	\[\begin{aligned}
		R^{1}_{\xi,\lambda}&=\Delta_g W_{1,\lambda}+2\lambda V_1 e^{W_{1,\lambda}}-\overline{2\lambda V_1 e^{W_{1,\lambda}} }- \rho_2 \left(\frac{V_2 e^{W_{2,\lambda}}}{\int_{\Sigma}V_2W_{2,\lambda}\, dv_g}-1\right),\\
		R^{2}_{\xi,\lambda}&=\Delta_g W_{2,\lambda}+2 \rho_2 \left(\frac{V_2 e^{W_{2,\lambda}}}{\int_{\Sigma}V_2W_{2,\lambda}\, dv_g}-1\right)-\lambda V_1 e^{W_{1,\lambda}}+ \overline{\lambda V_1 e^{W_{1,\lambda}} } .
	\end{aligned}\]
	Firstly, we will show that the error term goes to zero along $\lambda\rightarrow 0$.  
	\begin{lem}\label{lem:est_error}
		There exist $p_0>1$ and $\lambda_0>0$ such that for any $p \in\left(1, p_0\right)$  we have as $\lambda\rightarrow 0$
		\begin{equation}
			\label{eq:est_error}	\left\|\mathcal{R}_{\xi,\lambda}\right\|_p:=\sum_{i=1}^2 \|R^i_{\xi,\lambda}\|_{L^P(\Sigma)}=O\left(\lambda^{\frac{2-p}{2p}}\right).
		\end{equation}
	\end{lem}
	\begin{proof}
		By~\eqref{eq:expansion_eW2},
		\begin{eqnarray*}
			R^1_{\xi,\lambda}&=& \sum_{j=1}^m\Delta_gPU_j -\frac 1 2\Delta_g z(\cdot,\xi) +2\lambda V_1 e^{W_{1,\lambda}}-\overline{2\lambda V_1 e^{W_{1,\lambda}} }- \rho_2 \left(\frac{\tilde{V}_2(\cdot,\xi) e^{z(\cdot,\xi)}}{\int_{\Sigma}\tilde{V}_2(\cdot,\xi) e^{z(\cdot,\xi)\, dv_g}}-1\right)\\
			&&+ \mathcal{O}(\lambda|\log\lambda|)\\
			&=& 2\lambda V_1 e^{W_{1,\lambda}}-\sum_{j=1}^m \chi_j e^{-\varphi_j}e^{U_j}-\overline{ 2\lambda V_1 e^{W_{1,\lambda}}-\sum_{j=1}^m \chi_j e^{-\varphi_j}e^{U_j}}+\mathcal{O}(\lambda|\log\lambda|).
		\end{eqnarray*}
		Similarly, 
		\[R^2_{\xi,\lambda}= -\lambda V_1 e^{W_{1,\lambda}}+\frac 1 2\sum_{j=1}^m \chi_j e^{-\varphi_j}e^{U_j}+\overline{ \lambda V_1 e^{W_{1,\lambda}}-\frac 12 \sum_{j=1}^m \chi_j e^{-\varphi_j}e^{U_j}}+\mathcal{O}(\lambda|\log\lambda|).\]
		For any $p\in (1,2)$, the H\"{o}lder inequality yields that 
		\[  \left|\int_{\Sigma}\sum_{j=1}^m  \chi_j e^{-\varphi_j}  e^{U_j}-2\lambda  V_1 e^{W_{1,\lambda}}\, dv_g \right|\leq \left\|\sum_{j=1}^m  \chi_j e^{-\varphi_j}  e^{U_j}-2\lambda  V_1 e^{W_{1,\lambda}}\right\|_{L^p(\Sigma)}|\Sigma|_g^{1-\frac 1 p}. \]
		Then Lemma~\ref{lem:diif_e^W_sum_e^u} implies that for any $p\in(0,1)$
		\[ \|R^i_{\xi,\lambda}\|_{L^p(\Sigma)}=\mathcal{O}(\lambda^{\frac{2-p}{2p}}), \quad i=1,2.\]
		We set $p_0=2$, and the proof is complete. 
	\end{proof}
	The following lemma shows that the higher order linear operator $\cS_{\xi,\lambda}$ is bounded on $\cH$ and the operator norm vanishes as $\lambda\rightarrow 0$. 
	\begin{lem}\label{lem:esti_S}
		There exist $s_0> 1$  such that for any $p,r \in (1, 2)$ with  $pr\in (1,s_0) $, as $\lambda\rightarrow 0$
		\[
		\| \cS_{\xi,\lambda}(\bphi) \|_p = O\left( \lambda^{\frac{2- pr}{2 pr}} \| \bphi \|  \right), \text{ for }\bphi\in \cH.
		\]
	\end{lem}
	\begin{proof}
		By Lemma~\ref{lem:diif_e^W_sum_e^u} and the H\"{o}lder inequality, we  derive that 
		\begin{align*}
			\| \cS_{\xi,\lambda}(\bphi) \|_p &= \sum_{i=1}^{2} \| S^{i}_{\xi,\lambda}(\bphi) \|_{L^p(\Sigma)}\\
			& =O\left( \left\| \left(-\sum_{j=1}^m\chi_j e^{-\varphi_j} e^{U_j}+2 \lambda V_1e^{ W_{1,\lambda}}\right) \phi_1\right\|_{L^p(\Sigma)}\right)\\
			&= O\left( \left( \left\|2 \lambda V_1e^{W_{1,\lambda}}-\sum_{j=1}^m\chi_j e^{-\varphi_j} e^{U_1} \right\|_{L^{pr}(\Sigma)} \right) \left| \phi_1\right|_{\frac{pr}{r-1}} \right)  \\
			&= O\left(\lambda^{\frac{2-pr}{2pr}}  \left\| \phi \right\| \right) = O\left(\lambda^{\frac{2-pr}{2pr}} \|\bphi\|\right)\quad(\lambda\rightarrow 0),
		\end{align*}
		for any $p,r\in(1,2)$ with $pr<2.$
	\end{proof}

	To study the asymptotic behavior of  the non-linear part $\cN_{\xi,\lambda}$ we have the following lemma: 
	\begin{lem}\label{lem:non_linear_term}
		There exist $c_2,\varepsilon_0,\lambda_0>0$ and  $s_0 > 1$such that for any $p > 1$, $r > 1$ with $pr \in (1, s_0), \lambda\in(0, \lambda_0)$
		\[
		\|\mathcal{N}_{\xi,\lambda}(\bphi)\|_p = O\left( \lambda^{\frac{1-pr}{pr}} e^{c_2\|\bphi\|^2 } \|\bphi\|^{2} \right) 
		\]
		and
		\[
		\|\mathcal{N}_{\xi,\lambda}(\bphi^1) - \mathcal{N}_{\xi,\lambda}(\bphi^0)\|_p =\mathcal{O}\left(\lambda^{\frac{1-pr}{pr}} e^{c_2\sum_{h=0}^1\|\bphi^h\|^2 } \|\bphi^1 - \bphi^0\|(\|\bphi^1\| + \|\bphi^0\|) \right) 
		\]
		hold true for any   $\bphi, \bphi^1, \bphi^0 \in \{ \bphi=(\phi_1,\phi_2) \in \cH : \|\phi_i\| \leq \varepsilon_0, i=1,2\}$. 
	\end{lem}
	\begin{proof}
		It is sufficient to prove the second estimate since we can take $\bphi^0=0$ to deduce the first one. 
		Let $f_1, f_2$ be defined in~\eqref{eq:def_f12}.
		We observe that 
		\begin{eqnarray*}
			&&	f_1(W_{1,\lambda}+\phi^1_1)-f_1(W_{1,\lambda}+\phi^0_1)- f'_1( W_{1,\lambda})(\phi^1_1-\phi^0_1) \\
			&
			= & 2\lambda V_1e^{W_{1,\lambda}+\phi^1_1} -2\lambda V_1e^{W_{1,\lambda}+\phi^0_1}-2\lambda  V_1e^{W_{1,\lambda}}(\phi^1_1-\phi^0_1) ,\\	&&f_2(W_{2,\lambda}+\phi^1_2)-f_2(W_{2,\lambda}+\phi^0_2)- f'_2( W_{2,\lambda})(\phi^1_2-\phi^0_2) \\
			&=& \frac{\rho_2 V_2 e^{W_{2,\lambda}+\phi^1_2} }{\int_{\Sigma} V_2 e^{W_{2,\lambda}+\phi^1_2} d v_g}-\frac{\rho_2 V_2 e^{W_{2,\lambda}+\phi^0_2} }{\int_{\Sigma} V_2 e^{W_{2,\lambda}+\phi^0_2} d v_g}
			- \frac{\rho_2 V_2 W_{2,\lambda}(\phi_2^1-\phi^0_2)}{\int_{\Sigma} V_2W_{2,\lambda} d v_g}\\
			&&+ 
			\frac{\rho_2 V_2 W_{2,\lambda} \int_{\Sigma} V_2 e^{W_{2,\lambda}} (\phi^1_2-\phi^0_2) d v_g}{\left(\int_{\Sigma} V_2 e^{W_{2,\lambda}} d v_g\right)^2}.
		\end{eqnarray*}
		For $i=1,2,$ $\theta,\gamma\in (0,1)$, the mean value theorem yields that for any $p>1$
		\begin{eqnarray*}
			&&\|f_i(W_{i,\lambda}+\phi^1_i)-f_i(W_{i,\lambda}+\phi^0_i)- f'_i( W_{i,\lambda})(\phi^1_i-\phi^0_i)\|_{L^p(\Sigma)}\\
			&=& \|(f'_i(W_{i,\lambda}+ \theta\phi^1_i+(1-\theta)\phi^0_i) -f'_i(W_{i,\lambda}) (\phi^1_i-\phi^0_i)\|_{L^p(\Sigma)}\\
			&=&\|f^{\prime\prime}_i(W_{i,\lambda}+ \gamma\theta\phi^1_i+\gamma(1-\theta)\phi^0_i) (\theta\phi^1_i+(1-\theta)\phi^0_i) (\phi^1_i-\phi^0_i)\|_{L^p(\Sigma)}.
		\end{eqnarray*}
		For $i=1$, by the H\"{o}lder inequality and the  Moser-Trudinger inequality, we derive that 
		\begin{eqnarray*}
			&&	\|f^{\prime\prime}_1(W_{1,\lambda}+ \gamma\theta\phi^1_1+\gamma(1-\theta)\phi^0_1) (\theta\phi^1_1+(1-\theta)\phi^0_1) (\phi^1_1-\phi^0_1)\|_{L^p(\Sigma)}  \\
			&\leq& C \sum_{h=0}^1
			\left(  \int_{\Sigma}
			\lambda^p V_1^p	e^{pW_{1,\lambda} }(  e^{|\phi^0_1|+|\phi^0_1|}  |\phi^1_1-\phi^0_1||\phi^h_1| ) ^p\, dv_g \right)^{1/p}  \\
			&\leq &  C \sum_{h=0}^1 \left(\int_{\Sigma} \lambda^{pr}V_1^{pr}e^{pr W_{1,\lambda} }\, dv_g\right)^{\frac{1}{pr}} \left( \int_{\Sigma}   e^{ps (|\phi^0_1|+|\phi^1_1|)} \,dv_g\right)^{\frac{1}{ps}}\\
			&& \left(\int_{\Sigma} |\phi^1_1-\phi^0_1|^{pt} |\phi^h_1|^{pt} dv_g\right)^{\frac{1}{pt}}  \\
			&\leq & C \sum_{h=0}^1 \| \lambda V_1e^{ W_{1,\lambda}} \|_{L^{pr}(\Sigma)} e^{\frac{ps}{8\pi}(\|\phi^1_1\|^2+\|\phi^0_1\|^2)} \|\phi^1_1-\phi^0_1\| \|\phi^h_1\|,
		\end{eqnarray*}
		where $ r,s,t \in (1, +\infty), {  \frac{1}{r}+\frac{1}{s}+\frac{1}{t}=1}$.
		Applying Lemma~\ref{lem:diif_e^W_sum_e^u}, we deduce that 
		\begin{eqnarray*}
			\|2\lambda V_1 e^{W_{1,\lambda}}\|_{L^{pr}(\Sigma)}&\leq& \left\|2 \lambda V_1e^{W_{1,\lambda}}-\sum_{j=1}^m\chi_j e^{-\varphi_j} e^{U_j} \right\|_{L^{pr}(\Sigma)}+ \sum_{j=1}^m\left\|\chi_j e^{-\varphi_j} e^{U_j} \right\|_{L^{pr}(\Sigma)}\\
			&\leq& \sum_{j=1}^m\left\|\chi_j e^{-\varphi_j} e^{U_j} \right\|_{L^{pr}(\Sigma)}+\mathcal{O}(\lambda^{\frac{2-pr}{2}}).
		\end{eqnarray*}
		Lemma~\ref{lem:extension_PU} implies that 
		\begin{eqnarray*}
			\int_{\Sigma} \chi^{pr}_j e^{-pr\varphi_j}e^{pr U_j}dv_g&=& \mathcal{O}\left(\delta_j^{2-2pr}\int_{\Omega_j} \left(\frac{1}{(1+|y|^2)^2} \right)^{pr}\, dy \right)+\mathcal{O}(\delta_j^2)\\
			&=& \mathcal{O}(\delta_j^{2-2pr})= \mathcal{O}( \lambda^{1-pr}).
		\end{eqnarray*}
		Hence, 
		\begin{eqnarray}
			\label{eq:est_1st_component_nonlin_term}
			&&	\left\|   2\lambda V_1e^{W_{1,\lambda}+\phi^1_1} -2\lambda V_1e^{W_{1,\lambda}+\phi^0_1}-2\lambda  V_1e^{W_{1,\lambda}}(\phi^1_1-\phi^0_1)\right\|_{L^p(\Sigma)}\\
			&\leq& C \sum_{h=0}^1 \lambda^{\frac{1-pr}{pr}} e^{\frac{ps}{8\pi}(\|\phi^1_1\|^2+\|\phi^0_1\|^2)} \|\phi^1_1-\phi^0_1\| \|\phi^h_1\|. \nonumber
		\end{eqnarray} 
		
		For $u,v,w\in\oH$, 
		\begin{eqnarray*}
			&&	(2\rho_2)^{-1}	f_2^{\prime\prime}(W_{2,\lambda} + u)(v)(w)\\
			&=& \frac{V_2e^{W_{2,\lambda}+u}vw}{\int_{\Sigma}V_2 e^{W_{2,\lambda}+u}\, dv_g} - \frac{V_2e^{W_{2,\lambda}+u}v}{\left(\int_{\Sigma} V_2e^{W_{2,\lambda}+u}\, dv_g\right)^2} \int_{\Sigma}V_2 e^{W_{2,\lambda}+u}w \, dv_g \\
			&& - \frac{V_2e^{W_{2,\lambda}+u}}{\left(\int_{\Sigma}V_2 e^{W_{2,\lambda}+u}\, dv_g \right)^2}w \int_{\Sigma}V_2 e^{W_{2,\lambda}+u}v \, dv_g\\
			&&- \frac{V_2e^{W_{2,\lambda}+u}}{\left(\int_{\Sigma}V_2 e^{W_{2,\lambda}+u}\, dv_g\right)^3} \left(\int_{\Sigma}V_2 e^{W_{2,\lambda}+u}v\, dv_g \right) \left(\int_{\Sigma}V_2 e^{W_{2,\lambda}+u}w\, dv_g\right).
		\end{eqnarray*}
		
		Given that $\|u\|,\lambda
		$ are sufficiently small, 
		\begin{eqnarray}
			\label{eq:low_bound_veW2}
			\int_{\Sigma} V_2 e^{W_{2,\lambda}+u}\, dv_g &\geq & \int_{\Sigma}V_2e^{W_{2,\lambda}}\, dv_g- \int_{\Sigma} V_2e^{W_{2,\lambda}}|e^u -1|\, dv_g\\
			&\geq&\int_{\Sigma}V_2e^{W_{2,\lambda}}\, dv_g- C\int_{\Sigma} V_2e^{W_{2,\lambda}}|u|\, dv_g\nonumber\\
			&\geq  &\int_{\Sigma}\tilde{V}_2(\cdot,\xi)e^{z(\cdot,\xi)}\, dv_g- C\|\tilde{V}_2(\cdot,\xi)e^{z(\cdot,\xi)}\|_{L^p(\Sigma)}\|u\| +o(1)\nonumber\\
			&\geq& C,\nonumber
		\end{eqnarray}
		where $C>0$ is a constant. 
		The H\"older's inequality  and~\eqref{eq:low_bound_veW2} imply that  for $a,b,c,d, r,s,t,e,f\in (1,+\infty)$ with $\frac 1 a +\frac 1 b+\frac 1 c+\frac 1 d =1,\frac 1 r +\frac 1 s +\frac 1 t =1, \frac 1 e+\frac 1 f =1, $
		\begin{eqnarray*}
			& &\|f^{\prime\prime}_2(W_{2,\lambda}+ \gamma\theta\phi^1_2+\gamma(1-\theta)\phi^0_2) (\theta\phi^1_2+(1-\theta)\phi^0_2) (\phi^1_2-\phi^0_2)\|_{L^p(\Sigma)}\\
			&&\leq   C\sum_{h=0}^1\|V_2e^{W_{2,\lambda}}\|_{L^{pa}(\Sigma)} \|e^{ |\phi^1_2|+ |\phi^0_2|}\|_{L^{pb}(\Sigma)} \| |\phi^h_2|\|_{L^{pc}(\Sigma)} \|\phi^1_2-\phi^0_2\|_{L^{pd}(\Sigma)} \\
			&&+  C\sum_{h=0}^1\|V_2e^{W_{2,\lambda}}\|^2_{L^{pr}(\Sigma)} \|e^{ |\phi^1_2|+ |\phi^0_2|}\|^2_{L^{ps}(\Sigma)} \| |\phi^h_2|\|_{L^{pt}(\Sigma)} \|\phi^1_2-\phi^0_2\|_{L^{pt}(\Sigma)}\\
			&&+C\sum_{h=0}^1\|V_2e^{W_{2,\lambda}}\|_{L^{pe}(\Sigma)} \|e^{ |\phi^1_2|+ |\phi^0_2|}\|_{L^{pf}(\Sigma)}\|V_2e^{W_{2,\lambda}}\|_{L^{pa}(\Sigma)} \|e^{ |\phi^1_2|+ |\phi^0_2|}\|_{L^{pb}(\Sigma)} \| |\phi^h_2|\|_{L^{pc}(\Sigma)} \|\phi^1_2-\phi^0_2\|_{L^{pd}(\Sigma)}\\
			&&+C\sum_{h=0}^1\|V_2e^{W_{2,\lambda}}\|_{L^{pe}(\Sigma)} \|e^{ |\phi^1_2|+ |\phi^0_2|}\|_{L^{pf}(\Sigma)} \|V_2e^{W_{2,\lambda}}\|^2_{L^{pr}(\Sigma)} \|e^{ |\phi^1_2|+ |\phi^0_2|}\|^2_{L^{ps}(\Sigma)} \| |\phi^h_2|\|_{L^{pt}(\Sigma)} \|\phi^1_2-\phi^0_2\|_{L^{pt}(\Sigma)},
		\end{eqnarray*}
		for $\|\bphi^i\|$ sufficiently small $i=0,1.$
		Moser-Trudinger inequality yields that 
		\[ \| e^{ |\phi^1_2|+ |\phi^0_2|} \|^2_{L^{q}(\Sigma)}\leq C e^{\frac{q}{8\pi} (\|\phi^1_2\|+ \|\phi^0_2\|)}, q>1.
		\]
		
		By~\eqref{eq:expansion_eW2}, 
		\[ \|V_2e^{W_{2,\lambda}}\|^2_{L^{pr}(\Sigma)} \leq C\|\tilde{V}_2(\cdot,\xi)e^{z(\cdot,\xi)}\|_{L^{q}(\Sigma)}\leq C .\]
		Combining the estimates above, we deduce that  
		\begin{eqnarray}
			\label{eq:non_linear_2rd}
			&&\left\|f_2(W_{2,\lambda}+\phi^1_2)-f_2(W_{2,\lambda}+\phi^0_2)- f'_2( W_{2,\lambda})(\phi^1_2-\phi^0_2)\right\|_{L^p(\Sigma)}\\
			&&\leq C e^{c_2\sum_{h=0}^1 \|\phi^h_2\|}(\|\phi^1_2\|+\|\phi^0_2\|)(\|\phi^1_2-\phi^0_2\|). \nonumber
		\end{eqnarray}
		
		Therefore, we proved that  there exists constants $c_2,\varepsilon_0,\lambda_0>0$ and $s_0>1$ for any $p,r>0$ satisfying $pr\in (1, s_0)$ and $\|\bphi^i\|\leq \varepsilon_0$ for $i=1,2$
		\begin{equation}~\label{est_N_s}
			\|\mathcal{N}_{\xi,\lambda}(\bphi^1) - \mathcal{N}_{\xi,\lambda}(\bphi^0)\|_p =\mathcal{O}\left(\lambda^{\frac{1-pr}{pr}} e^{c_2\sum_{h=0}^1\|\bphi^h\|^2 }(\|\bphi^1\|+\|\bphi^0\|)\|\bphi^1-\bphi^0\|\right).
		\end{equation}
	\end{proof}
	Next, for the fixed approximation solution $\bW_{\lambda}$, we will find $\bphi_{\lambda}$ to solve the problem~\eqref{eq:toda_equi2} in $K^{\perp}_{\xi}$, i.e. 
	\begin{equation}
		\label{eq:toda_inf_dim} 
		\bphi_{\lambda}=\Pi_{\xi}^{\perp}\circ \cL_{\xi,\lambda}^{-1}(\cS_{\xi,\lambda}(\bphi_{\lambda})+\cN_{\xi,\lambda}+\cR_{\xi,\lambda})
	\end{equation}
	for $\bphi_{\lambda}\in K_{\xi}^{\perp}$.

	\begin{thm}~\label{thm2} Let $\Dc$ be a compact subset of $\Xi_{k,m}$,
		and $\xi=(\xi_1,...,\xi_m)\in \Dc$. There exist $p_0>1, \lambda_0>0$ and $R>0$ (uniformly in $\xi$) such that for any $p\in (1, p_0)$ and  any $\lambda\in (0,\lambda_0)$ there is a unique $\bphi_{\xi,\lambda}\in K^{\perp}_{\xi}$ solves~\eqref{eq:toda_inf_dim}
		satisfying that 
		\[ \|\bphi_{\xi,\lambda}\|\leq R \lambda^{ \frac{ 2-p}{2p}}|\log \lambda|. \]
		Furthermore, 
		the map $\xi\mapsto \bphi_{\xi,\lambda}$ is $C^1$ map with respect to $\xi$.
	\end{thm} 
	\begin{proof}
		Given that $\xi\in \Dc$, we define the linear operator 
		\[ \cT_{\xi,\lambda}(\bphi):=\Pi_{\xi}^{\perp}\circ \cL_{\xi,\lambda}^{-1}(\cS_{\xi,\lambda}(\bphi_{\lambda})+\cN_{\xi,\lambda}+\cR_{\xi,\lambda}) \]
		on $K_{\xi}^{\perp}$. 
		For any $\bphi\in K_{\xi}^{\perp}$, 
		by Lemma~\ref{lem:invertible}-\ref{lem:non_linear_term}, there exist constants $s_0>1$, $C_0, C, c_2>0$ such that 
		\begin{eqnarray*}
			\|	\cT_{\xi,\lambda}(\bphi)\|&\leq& C |\log \lambda|\| \cS_{\xi,\lambda}(\bphi)+\cN_{\xi,\lambda}(\bphi)+\cR_{\xi,\lambda}\|_p\\
			&\leq & C_0 |\log\lambda| \left( \lambda^{\frac{2-pr}{2pr}}\|\bphi\|+ \lambda^{\frac{1-pr}{pr}} e^{c_2\|\bphi\|^2 } \|\bphi\|^{2} + \lambda^{\frac{2-p}{2p}}\right),
		\end{eqnarray*}
		for any $p,r\in(1,2)$ with $pr\in (1,s_0)$.
		We take $r=\frac 5 4, R=3C_0$. Then for arbitrary fixed $p\in (1,  \frac 3 2 )$, there exists $\lambda_0>0$ such that for any $\lambda\in (0,\lambda_1)$ we have 
		\[ \max\{3C_0 (\sqrt{c_2}+1)\lambda
		^{\frac{2-p}{2p}}|\log\lambda|, 3C_0e^{2}\lambda^{ \frac {1-pr}{pr}+ \frac{2-p}{2p}}|\log\lambda|\}\leq 1. \]
		Thus, for any $\bphi\in \{ \bphi\in K_{\xi}^{\perp}: \|\bphi\|\leq R\lambda^{\frac{2-p}{2p}}|\log\lambda|\}$,
		\[\|\cT_{\xi,\lambda}(\bphi)\|\leq R\lambda^{\frac{2-p}{2p}}|\log\lambda|. \]
		For any $\bphi^0,\bphi^1\in \{ \bphi\in K_{\xi}^{\perp}: \|\bphi\|\leq R\lambda^{\frac{2-p}{2p}}|\log\lambda|\}$, Lemma~\ref{lem:esti_S} and Lemma~\ref{lem:non_linear_term} yield that 
		\begin{eqnarray*}
			\|\cT_{\xi,\lambda}(\bphi^1)-\cT_{\xi,\lambda}(\bphi^0)\|&\leq& C |\log \lambda|\| \cS_{\xi,\lambda}(\bphi^1-\bphi^0)+\cN_{\xi,\lambda}(\bphi^1)-\cN_{\xi,\lambda}(\bphi^0)\|_p\\
			&\leq&C_1|\log\lambda|\left( \lambda^{\frac{2-pr}{2pr}}\|\bphi^1-\bphi^0\|  + R\lambda^{\frac{1-pr}{pr}+\frac{2-pr}{2pr}}|\log\lambda|\|\bphi^1 - \bphi^0\|\right) 
		\end{eqnarray*}
		Since $ \frac{1-pr}{pr}+\frac{2-pr}{2pr}>0$, there exists $\lambda_2<\lambda_1$ such that for any $\lambda\in(0,\lambda_2)$, 
		\[ \max\{ C_1 |\log\lambda| \lambda^{\frac{2-pr}{2pr}}, C_1R\lambda^{\frac{1-pr}{pr}+\frac{2-pr}{2pr}}|\log\lambda| \}\leq \frac 1 4. \]
		We choose $\lambda_0=\lambda_2$. Consequently, we obtain that $\cT_{\xi,\lambda}$ is a contract mapping on $ \{ \bphi\in K_{\xi}^{\perp}: \|\bphi\|\leq R\lambda^{\frac{2-p}{2p}}|\log\lambda|\}$ satisfying that 
		\[\|\cT_{\xi,\lambda}(\bphi^1)-\cT_{\xi,\lambda}(\bphi^0)\|\leq \frac 1 2\|\bphi^1-\bphi^0\|. \]
		The Banach fixed-point theorem deduces that there exists a unique $\bphi_{\xi,\lambda}\in\{ \bphi\in K_{\xi}^{\perp}: \|\bphi\|\leq R\lambda^{\frac{2-p}{2p}}|\log\lambda|\} $ that solves the problem~\eqref{eq:toda_inf_dim}. 
		Let $F(\bu)$ be defined by~\eqref{eq:def_F}.
		We define a function $\Phi: \Dc\times \cH\rightarrow \cH,$ $$(\xi,\bphi)\mapsto \bphi+\Pi_{\xi}^{\perp}\left(W_{\lambda} -i^*\circ F(W_{\lambda}+\Pi_{\xi}^{\perp}(\bphi) \right).$$ 
		We observe that 
		$\Phi(\xi,\bphi_{\xi,\lambda})=0$ and  for any $\psi\in\cH$
		\[ \frac{\partial \Phi}{\partial\bphi}(\xi,\bphi_{\xi,\lambda})(\psi)
		=\psi- \Pi_{\xi}^{\perp}\circ i^*(F^{\prime}(\bW_{\lambda}+\bphi_{\xi,\lambda})(\Pi_{\xi}^{\perp}\psi)).\]
		\begin{claim}
			$ \frac{\partial F}{\partial\bphi}(\xi,\bphi_{\xi,\lambda})$ is non-degenerate.
		\end{claim}
		{\it Indeed, }
		\begin{eqnarray*}
			\frac{\partial \Phi}{\partial\bphi}(\xi,\bphi_{\xi,\lambda})(\psi)&=&\Pi_{\xi}\psi-\Pi_{\xi}^{\perp}\circ i^*\circ (-\cL_{\xi,\lambda}+\cS_{\xi,\lambda})(\Pi_{\xi}^{\perp}\psi)\\
			&&- \Pi_{\xi}^{\perp}\circ i^*\left( (F^{\prime}(\bW_{\lambda}+\bphi_{\xi,\lambda})-F^{\prime}(\bW_{\lambda})(\Pi_{\xi}^{\perp}\psi)\right).
		\end{eqnarray*}
		By the mean value theorem, there exists $\theta\in (0,1)$ such that 
		\begin{eqnarray*}
			\|F^{\prime}(\bW_{\lambda}+\bphi_{\xi,\lambda})-F^{\prime}(\bW_{\lambda})\Pi_{\xi}^{\perp}\psi \|_{p}&=& \|F^{\prime}(\bW_{\lambda}+\theta\bphi_{\xi,\lambda})\Pi_{\xi}^{\perp}\psi\|_p \\
			&\leq &C \lambda^{\frac{1-pr}{pr}}\|\bphi_{\xi,\lambda}\| \|\Pi_{\xi}^{\perp}\psi\|
		\end{eqnarray*}
		Then Lemma~\ref{lem:invertible} and Lemma~\ref{lem:esti_S} imply that for some constant $c>0$
		\begin{eqnarray*}
			&&\left\|	\frac{\partial \Phi}{\partial\bphi}(\xi,\bphi_{\xi,\lambda})(\psi)\right\|\\
			&\geq& \|\Pi_{\xi}^{\perp}\psi\|+c\|\cL_{\xi,\lambda}\|\|\Pi_{\xi}^{\perp}\psi\| - \|\cS_{\xi,\lambda}\|\|\Pi_{\xi}^{\perp}\psi\|- \|(F^{\prime}(\bW_{\lambda}+\bphi_{\xi,\lambda})-F^{\prime}(\bW_{\lambda}))\Pi_{\xi}^{\perp}\psi\|_p\\
			&\geq&\|\Pi_{\xi}^{\perp}\psi\|+\frac{c}{|\log\lambda|}\|\Pi_{\xi}^{\perp}\psi\| - \mathcal{O}(\lambda^{\frac{2-pr}{2pr}}\|\Pi_{\xi}^{\perp}\psi\|)+\mathcal{O}(\lambda^{\frac{1-pr}{pr}} \|\bphi_{\xi,\lambda}\|\|\Pi_{\xi}^{\perp}\psi\|)\\
			&\geq& \frac{c}{|\log\lambda|}\|\psi\|,
		\end{eqnarray*}
		for $p, r>1$ sufficiently close to $1$. 
		So, we obtain that $\frac{\partial \Phi}{\partial\bphi}(\xi,\bphi_{\xi,\lambda})$ is invertible with $$\left\| \left(\frac{\partial \Phi}{\partial\bphi}(\xi,\bphi_{\xi,\lambda})\right)^{-1}\right\|\leq \frac 1 c |\log \lambda|.$$
		By the implicit function theorem, we have 
		$\xi\mapsto \bphi_{\xi,\lambda}$ is $C^1$-differentiable. 
	\end{proof}
	\section{ The reduced functional and its expansion}
	We calculate the energy functional for approximation solutions $\bW_{\lambda}$. 
	We define the energy functional corresponding to~\eqref{eq:toda_equi} as follows 
	\[ E_{\lambda}(\bu)=\int_{\Sigma} Q(u,u)\, dv_g - \lambda\int_{\Sigma} V_1 e^{u_1}\, dv_g- \rho_2\log\left( \int_{\Sigma} V_2e^{u_2 }\, dv_g \right),\quad \bu=(u_1,u_2)\in\cH.\]
	\begin{lem}\label{lem:key_energy_app}
		Given $m\geq k\geq 0$, we assume that  that~\eqref{eq:def_delta_i} and \eqref{eq:def_d_ij} are valid, there exists $\lambda_0>0$ such that  we have the expansion holds for $\lambda\in(0,\lambda_0)$
		\begin{equation}\label{expansion_JW}
			\begin{array}{lcl}
				E_{\lambda}(\bW_{\lambda})= \Lambda_{k,m}(\xi)
				-6\pi(k+m)+ 2\pi(k+m)\log 8-2\pi(k+m)\log\lambda +o(1)
			\end{array}
		\end{equation}
		and 
		\begin{equation}\label{expansion_pa_JW}
			\begin{array}{lcl}
				\partial_{\xi}	E_{\lambda}(\bW_{\lambda})
				&=&\partial_{\xi}\Lambda_{k,m}(\xi)+o(1),
			\end{array}
		\end{equation}
		which are convergent  in  $C(\Xi_{k,m})$ uniformly for any $\xi$ in a compact subset of $\Xi_{k,m}$. 
	\end{lem}
	\begin{proof}
		Let $\Dc$ be a compact subset of $\Xi_{k,m}$. Then there exists $\varepsilon_0>0$ such that $\Dc\subset\Xi_{k,m}^{\varepsilon_0}$. We consider $\xi\in \Xi_{k,m}^{\varepsilon_0}$. 
		\begin{eqnarray*}
			Q(\bW_{\lambda},\bW_{\lambda})= \frac 1 4 \int_{\Sigma}\left( |\nabla z(\cdot,\xi)|_g^2+\left|\nabla \left(\sum_{j=1}^m PU_j\right)\right|_g^2-\sum_{j=1}^m \lan \nabla PU_j,\nabla z(\cdot,\xi)\ran_g \right)\, dv_g 
		\end{eqnarray*}
		The estimate	\eqref{eq:expansion_eW2} yields that 
		\begin{eqnarray*}
			&&\frac 1 4 \int_{\Sigma}|\nabla z(\cdot,\xi)|^2_g - \rho_2\log\int_{\Sigma} V_2 e^{W_{\lambda, 2}}  \, dv_g \\
			&=&\frac 1 2\underbrace{\left(\frac 1 2 \int_{\Sigma}|\nabla z(\cdot,\xi)|^2_g -2\rho_2 \log\int_{\Sigma}\tilde{V}_2(\cdot,\xi) e^{z(\cdot,\xi)}\, dv_g\right)}_{=I_{\xi}(z(\cdot,\xi))}+ \mathcal{O}(\lambda)\\
			&=& \frac 1 2 I_{\xi}(z(\cdot,\xi))+\mathcal{O}(\lambda),
		\end{eqnarray*}
		as $\lambda\rightarrow 0.$
		By~Lemma~\ref{lem:extension_PU} and Lemma~\ref{lem:diif_e^W_sum_e^u}, we deduce that for $p\in (1,2)$
		\begin{eqnarray}\label{eq:mass_1st_component}
			&&\lambda \int_{\Sigma} V_1 e^{W_{1,\lambda}}\, dv_g\\
			&= & \mathcal{O}\left( \left| \int_{\Sigma} 2V_1 e^{W_{1,\lambda}}-\sum_{j=1}^m \chi_j e^{-\varphi_j} e^{U_j}\, dv_g\right|\right)   + \frac 1 2 \sum_{j=1}^m\int_{\Sigma} \chi_j e^{-\varphi_j} e^{U_j}\, dv_g \nonumber\\
			&=& \mathcal{O}\left(\|2V_1 e^{W_{1,\lambda}}-\sum_{j=1}^m \chi_j e^{-\varphi_j} e^{U_j}\|_{L^p(\Sigma)}|\Sigma|^{1-\frac 1p}\right)+ \frac 1 2 \sum_{j=1}^m \varrho(\xi_j) + \mathcal{O}(\lambda)\nonumber\\
			&=&\frac 1 2  \sum_{j=1}^m \varrho(\xi_j)+\mathcal{O}\left(\lambda^{\frac{2-p}{2p}}+ \lambda\right),\nonumber
		\end{eqnarray}
		where  we applied that $\int_{|y|<r}\frac{ 8}{(1+|y|^2)^2}\, dy= 8\pi - \frac{8\pi}{1+r^2}$. 
		Using $z(\cdot,\xi)$ as a test function for the equation~\eqref{eq:proj},
		\begin{eqnarray*}
			\frac 1 4 \int_{\Sigma} \lan \nabla PU_j, \nabla z(\cdot,\xi)\ran_g \, dv_g &=& \frac 1 4 \int_{\Sigma} \chi_j e^{-\varphi_j} e^{U_j} z(\cdot,\xi)\, dv_g \\
			&=& \frac 1 4 \int_{\Omega_j} \chi(\delta_j|y|/r_0) \frac{ 8}{(1+|y|^2)^2} z(y_{\xi_j}^{-1}(\delta_j y),\xi)\, dy \\
			&=&\frac 1 4 \varrho(\xi_j)z(\xi_j,\xi)+ \mathcal{O}(\lambda^{\frac 1 2}).
		\end{eqnarray*}
		For any $j,j'=1,\ldots,m,$,  using $PU_{j'}$ as a test function for the equation~\eqref{eq:proj}, Lemma~\ref{lem:extension_PU} with~\eqref{eq:def_d_ij} implies that 
		\begin{eqnarray*}
			&&\frac 1 4 	\int_{\Sigma}\lan PU_{j'}, PU_j\ran_g\, dv_g =\frac1 4\int_{\Sigma} \chi_je^{-\varphi_j} e^{U_j} PU_{j'}\, dv_g \\
			&=&\frac 1 4  \int_{\Sigma}\chi_je^{-\varphi_j} e^{U_j}  \left(  \chi_{j'} ( U_{j'}-\log( 8 \delta_{j'}^2))+\varrho(\xi_{j'}) H^g(\cdot,\xi_{j'})+ \mathcal{O}(\delta_{j'}^2|\log \delta_{j'}|)\right)\\
			&=&	\left\{\begin{array}{ll}
				-\frac 1 2 \int_{\Omega_j} \chi(\delta_j |y|/r_0)\frac{ 8(2\log\delta_j+\log(1+|y|^2) )}{(1+|y|^2)^2}\, dy+\frac 1 4\varrho^2(\xi_j)H^g(\xi_j,\xi_j)+\mathcal{O}(\lambda^{\frac 1 2})& j'=j\\
				\frac 1 4\varrho(\xi_j) \varrho(\xi_{j'})G^g(\xi_j,\xi_{j'}) +\mathcal{O}(\lambda^{\frac 1 2})	& j'\neq j 
			\end{array}\right.\\
			&=& 	\left\{\begin{array}{ll}
				-\frac 1 2\varrho(\xi_j)\log \delta_j^2-\frac 1 2\varrho(\xi_j)+\frac{1}{4} \varrho^2(\xi_j)H^g(\xi_j,\xi_j)+ \mathcal{O}(\lambda^{\frac 1 2}) & j'=j\\
				\frac 1  4 \varrho(\xi_j)\varrho(\xi_{j'})G^g(\xi_j,\xi_{j'}) & j'\neq j
			\end{array}\right.,
		\end{eqnarray*}
		where we applied that $\int_{|y|\leq r} \frac{8\log(1+|y|^2)}{(1+|y|^2)^2}= 8\pi+8\pi\frac{\log(1+r^2)+1}{1+r^2} $ and $ \int_{|y|<r}\frac{ 8}{(1+|y|^2)^2}\, dy= 8\pi - \frac{8\pi}{1+r^2}.$
		Combining all the estimates above, we conclude that 
		\begin{eqnarray*}
			&&E_{\lambda}(\bW_{\lambda})\\
			&=& \frac 1 2 I_{\xi}(z(\cdot,\xi))-\frac 1 4 \mathcal{F}_{k,m}(\xi)
			-4\pi(k+m)+ 2\pi(k+m)\log 8-2\pi(k+m)\log\lambda +o(1),
		\end{eqnarray*}
		as $\lambda\rightarrow 0.$																					
		For any $j=1,\ldots,m$ and $i=1,\ldots,\i(\xi_j)$, 
		\begin{eqnarray*}
			\partial_{(\xi_j)_i} E_{\lambda}(\bW_{\lambda})
			&=& \int_{\Sigma}\left(-\frac 2 3 \Delta_g W_{1,\lambda}-\frac 1 3 \Delta_g W_{2,\lambda}-\frac 1 2 f_1(W_{1,\lambda})\right) \partial_{(\xi_j)_i} W_{1,\lambda}\, dv_g \\
			&+& \int_{\Sigma}\left(-\frac 2 3 \Delta_g W_{2,\lambda}-\frac 1 3 \Delta_g W_{1,\lambda}-\frac 1 2 f_2(W_{2,\lambda})\right) \partial_{(\xi_j)_i} W_{2,\lambda}\, dv_g\\
			&=& \frac 1 2 \int_{\Sigma} \left( \sum_{l=1}^{m}\chi_l e^{-\varphi_l} e^{U_l
			} - f_1(W_{1,\lambda
			})\right) \partial_{(\xi_j)_i} W_{1,\lambda}\, dv_g \\
			&&+ \frac 1 2 \int_{\Sigma} \left( \frac{2\rho_2 \tilde{V}_2(\cdot,\xi) e^{z(\cdot,\xi)}}{\int_{\Sigma} \tilde{V}_2(\cdot,\xi) e^{z(\cdot,\xi)}\, dv_g }- f_2(W_{2,\lambda})\right)\partial_{(\xi_j)_i} W_{2,\lambda}\, dv_g,
		\end{eqnarray*}
		in view of $\int_{\Sigma} W_{i,\lambda}\, dv_g=0$ for $i=1,2.$
		By Lemma~\ref{lem4} and Lemma~\ref{lem:diff_PZ_partial_PU}, we have 
		$\|\partial_{(\xi_j)_i} W_{2,\lambda}\|=\mathcal{O}(\lambda^{-\frac 1 2})$. 
		Using Lemma~\ref{eq:expansion_eW2}, we  deduce that 
		\begin{eqnarray*}
			\left|\int_{\Sigma} \left( \frac{2\rho_2 \tilde{V}_2(\cdot,\xi) e^{z(\cdot,\xi)}}{\int_{\Sigma} \tilde{V}_2(\cdot,\xi) e^{z(\cdot,\xi)}\, dv_g }- f_2(W_{2,\lambda})\right)\partial_{(\xi_j)_i} W_{2,\lambda}\, dv_g\right|&\leq&\mathcal{O}\left(\lambda \|\partial_{(\xi_j)_i} W_{2,\lambda}\|_{L^1(\Sigma)}\right) \\
			&\leq& \mathcal{O}\left( \lambda^{\frac 1 2}\right). 
		\end{eqnarray*}
		Lemma~\ref{lem:extension_PZ}-\ref{lem:diff_PZ_partial_PU} and Lemma~\ref{lem:diif_e^W_sum_e^u} yield that for any $i=1,\ldots,\i(\xi_j)$, and for any $p\in (1,2)$
		\begin{eqnarray*}
			&&\frac 1 2 \int_{\Sigma} \left( \sum_{l=1}^{m}\chi_l e^{-\varphi_l} e^{U_l
			} - f_1(W_{1,\lambda
			})\right) \partial_{(\xi_j)_i} W_{1,\lambda}\, dv_g\\
			&=& \frac 1 2\sum_{l=1}^m \int_{\Sigma} \chi_le^{-\varphi_l} e^{U_l} Z^i_j \, dv_g -\lambda\int_{\Sigma} V_1 e^{W_{1,\lambda}} Z^i_j\, dv_g \\
			&&+  
			\mathcal{O}\left(\left\| \sum_{l=1}^{m}\chi_l e^{-\varphi_l} e^{U_l
			} - f_1(W_{1,\lambda
			})\right\| _{L^p(\Sigma)} |\Sigma|^{1-\frac 1 p}\right)\\
			&=& \frac 1 2\sum_{l=1}^m \int_{\Sigma} \chi_le^{-\varphi_l} e^{U_l} \chi_j Z^i_j \, dv_g -\lambda\int_{\Sigma} V_1 e^{W_{1,\lambda}}\chi_j Z^i_j\, dv_g\\
			&&+o(1),
		\end{eqnarray*}
		as $\lambda\rightarrow 0$. 
		By a direct calculation, we have for $i=1,\ldots,\i(\xi_j)$
		\begin{eqnarray*}
			\int_{\Sigma} \chi_l e^{-\varphi_l} e^{U_l} \chi_j Z^i_j \, dv_g 
			&=& 	\left\{\begin{array}{ll}
				\frac 1 {\delta_j}	\int_{\Omega_j}\chi^2\left(\frac{|y|}{r_0}\right)\frac{32 y_i }{(1+|y|^2)^3} \, dy  &\text{ for } l=j  \\
				0& \text{ for } l\neq j 
			\end{array}\right.=0,
		\end{eqnarray*}
		where the last equality applied the symmetric property of $ \Omega_j$. 
		It is sufficient to calculate the integral $\lambda\int_{\Sigma} V_1 e^{W_{1,\lambda}}\chi_j Z^i_j\, dv_g. $
		Let $\tau_j(x)= \varrho(\xi_j)H^g(x,\xi_j) + \sum_{l\neq j}\varrho(\xi_l)G^g(x,\xi_l)
		-\frac 1 2 z(x,\xi) +\log V_1(x).$
		Recall that $d_j^2 = \frac 1 4 e^{\tau_j(\xi_j)}$. 
		Using Lemma~\ref{lem:extension_PU} with~\eqref{eq:def_d_ij}, we can derive that $\lambda\rightarrow 0$
		\begin{eqnarray*}
			&&\lambda\int_{\Sigma} V_1 e^{W_{1,\lambda}}\chi_j Z^i_j\, dv_g=2\lambda \int_{\Sigma} \chi_j e^{\sum_{l=1}^m PU_l -\frac 1 2 z(\cdot,\xi) +\log V_1  }\frac{4(y_{\xi_j})_i}{\delta_j^2+|y_{\xi_j}|^2}\, dv_g \\
			&=&\lambda \int_{\Sigma} \chi_j e^{
				\chi_j(U_j-\log(8\delta_j^2))+ \varrho(\xi_j)H^g(\cdot,\xi_j) + \sum_{l\neq j}\varrho(\xi_l)G^g(\cdot,\xi_l)
				-\frac 1 2 z(\cdot,\xi) +\log V_1 +\mathcal{O}(\delta_j^2|\log\delta_j|) }\frac{4(y_{\xi_j})_i}{\delta_j^2+|y_{\xi_j}|^2}\, dv_g\\
			&=& \frac {\lambda} {\delta_j^3}\int_{\Omega_j}\chi\left(\frac{\delta_j|y|}{r_0}\right)\frac{ 4y_i}{(1+|y|^2)^3} e^{ \tau_j(\xi_j)
			}\left(1 + \sum_{s=1}^2 \delta_j\partial_{y_s} \tau_j\circ y_{\xi_j}^{-1}(0) y_s+\mathcal{O}( \delta^2_j|y|^2+\delta_j^2|\log\delta_j|)\right)\, dy\\
			&=& \frac{ e^{\tau_j(\xi_j)}\partial_{y_i}\tau_j\circ
				y_{\xi_j}^{-1}(0) \varrho(\xi_j)}{8d_j^2}+o(1)=\frac 1 4 \partial_{(\xi_j)_i} \mathcal{F}_{k,m}(\xi) -\frac 14\varrho(\xi_j) \partial_{(x_j)_i} z(x,\xi)|_{x=\xi_j}+o(1),
		\end{eqnarray*}
		where we applied the symmetric property of $\Omega_j$ and $\int_{\R^2} \frac{4y_i^2}{(1+|y|^2)^3}\, dy=\pi. $
		Since $z(\cdot,\xi)$ solves~\eqref{eq:singular_mf}, we have $I'_{\xi}(z(\cdot,\xi))=0.$
		By the representation's formula, we deduce that 
		\begin{eqnarray*}
			&&	\partial_{(\xi_j)_i} I_{\xi}(z(\cdot,\xi))\\
			&=& I'_{\xi}(z(\cdot,\xi)) \partial_{(\xi_j)_i}z(\cdot,\xi) -2\rho_2 \int_{\Sigma} \frac{ \tilde{V}_2(\cdot,\xi)e^{z(\cdot,\xi)}}{\int_{\Sigma} \tilde{V}_2(\cdot,\xi)e^{z(\cdot,\xi)}\, dv_g}
			\left(- \frac{\varrho(\xi_j)}{2} \partial_{(\xi_j)_i} G^g(\cdot,\xi_j)\right)\,dv_g \\
			&=&\rho_2 \varrho(\xi_j) \int_{\Sigma} \frac{ \tilde{V}_2(\cdot,\xi)e^{z(\cdot,\xi)}}{\int_{\Sigma} \tilde{V}_2(\cdot,\xi)e^{z(\cdot,\xi)}\, dv_g}
			\partial_{(\xi_j)_i} G^g(\cdot,\xi_j)\,dv_g\\
			&=&\frac 1 2 \varrho(\xi_j) \partial_{(x_j)_i} z(x,\xi)|_{x=\xi_j}.
		\end{eqnarray*}
		Hence, we obtain that as $\lambda\rightarrow 0$
		\[ \partial_{(\xi_j)_i}E_{\lambda}(\bW_{\lambda}) =\frac 12\partial_{(\xi_j)_i} I_{\xi}(z(\cdot,\xi))-\frac 1 4 \partial_{(\xi_j)_i} \mathcal{F}_{k,m}(\xi)+o(1)=\partial_{(\xi_j)_i}\Lambda_{k,m}(\xi)+o(1).  \]
	\end{proof}
	Next, we consider the reduced functional 
	$\tilde{E}_{\lambda}(\xi):= E_{\lambda}(\bW_{\lambda}+\bphi)$  for any $\xi\in \Xi_{k,m}, $  where $\bphi$ is a solution of  Theorem~\ref{thm2}. 
	\begin{thm}
		\label{thm:expansion_E_reduced}
		Given $m\geq k\geq 0$, we assume that  that~\eqref{eq:def_delta_i} and \eqref{eq:def_d_ij} are valid, there exists $\lambda_0>0$ such that  we have the expansion holds for $\lambda\in(0,\lambda_0)$
		\begin{equation}\label{expansion_E_re}
			\begin{array}{lcl}
				\tilde{E}_{\lambda}(\xi)= \Lambda_{k,m}(\xi)
				-6\pi(k+m)+ 2\pi(k+m)\log 8-2\pi(k+m)\log\lambda +o(1)
			\end{array}
		\end{equation}
		and 
		\begin{equation}\label{expansion_pa_E_re}
			\begin{array}{lcl}
				\partial_{\xi}\tilde{E}_{\lambda}(\xi)
				&=&\partial_{\xi}\Lambda_{k,m}(\xi)+o(1),
			\end{array}
		\end{equation}
		which are convergent  in  $C(\Xi_{k,m})$ uniformly for any $\xi$ in a compact subset of $\Xi_{k,m}$. 
	\end{thm}
	\begin{proof}
		\begin{eqnarray*}
			&&\tilde{E}_{\lambda}(\xi)= E_{\lambda}(\bW_{\lambda})+  \frac{1}{2} \left( \int_{\Sigma
			} \left( \sum_{l=1}^{m} \chi_l e^{-\varphi_l} e^{U_l} -2 \lambda V_1 e^{W_{1,\lambda}} \right) \phi_1 \, dv_g - \int_{\Sigma}\phi_2 \Delta_g z(\cdot, \xi) \, dv_g \right)\\
			&&	- \lambda \int_{\Sigma} V_1  e^{W_{1,\lambda}} \left( e^{\phi_1} -1 -  \phi_1 \right)\, dv_g - \rho_2 \log \left( \int_{\Sigma} V_2 e^{W_{2,\lambda}+\phi_2} \,dv_g  \right)+ \rho_2\log \left( \int_{\Sigma} V_2 e^{W_{2,\lambda}} \,dv_g  \right) .
		\end{eqnarray*}
		Recall that $\bphi=(\phi_1,\phi_2)\in K_{\xi}^{\perp}$ satisfies that for $p$ sufficiently close to $1$
		\[ \|\bphi\|\leq R \lambda^{ \frac{ 2-p}{2p}}|\log \lambda|. \]
		For  $p,r>1$ sufficiently close to $1$, the H\"{o}lder inequality implies that 
		\begin{eqnarray*}
			\left| \int_{\Sigma
			} \left( \sum_{l=1}^{m} \chi_l e^{-\varphi_l} e^{U_l} -2 \lambda V_1 e^{W_{1,\lambda}} \right) \phi_1 \, dv_g\right| &\leq& \left\|\sum_{l=1}^{m} \chi_l e^{-\varphi_l} e^{U_l} -2 \lambda V_1 e^{W_{1,\lambda}} \right\|_{L^p(\Sigma)}\|\phi_1\|_{L^{\frac{p}{p-1}}(\Sigma)}\\
			&\leq& \mathcal{O}(\lambda^{\frac{2-p}{2p}}\|\phi_1\|)=o(1),\quad\text{\small (by Lemma~\ref{lem:diif_e^W_sum_e^u})}\\
			\left| \int_{\Sigma} \Delta_g z(\cdot, \xi) \phi_2\, dv_g\right|&\leq& \|z(\cdot,\xi)\|\|\phi_2\| =o(1),\\
			\left|  \lambda \int_{\Sigma} V_1  e^{W_{1,\lambda}} \left( e^{\phi_1} -1 -  \phi_1 \right)\, dv_g\right| &\leq & \mathcal{O}\left( \lambda^{\frac{1-pr}{pr}} e^{\frac{ps}{8\pi}\|\phi_1\|^2} \|\phi_1\|^2 \right) =o(1).
			\quad\text{\small (by \eqref{eq:est_1st_component_nonlin_term})}
		\end{eqnarray*}
		By mean value theorem, for some $\theta\in(0,1)$,
		\begin{eqnarray*}
			- \rho_2 \log \left( \int_{\Sigma} V_2 e^{W_{2,\lambda}+\phi_2} \,dv_g  \right)+ \rho_2\log \left( \int_{\Sigma} V_2 e^{W_{2,\lambda}} \,dv_g  \right) &=& \frac{\int_{\Sigma}  V_2 e^{W_{2,\lambda}+\theta \phi_2} \phi_2\,dv_g }{\int_{\Sigma}  V_2 e^{W_{2,\lambda}+\theta \phi_2} \,dv_g}\\
			&=&\mathcal{O}(\|\phi_2\|)=o(1).
		\end{eqnarray*}
		So,  we obtain that
		\[ \tilde{E}_{\lambda}(\xi)=E_{\lambda}(W_{\lambda})+o(1). \]
 Lemma~\ref{lem:key_energy_app} yields \eqref{expansion_E_re} holds in $C(\Sigma)$ and uniformly for any $\xi$ in any compact subset of $\Xi_{k,m}$. 
		Applying Theorem~\ref{thm2}, there exists $\{c^{\lambda}_{ij}\in\R: j=1,\ldots,m, i=1,\ldots,\i(\xi_j) \}$ such that
		\begin{equation}\label{eq:solution_inf_dim}
			\bW_{\lambda}+\bphi -i^*(F(\bW_{\lambda}+\bphi ))=\left(\begin{array}{c}
				\sum_{i, j} c^{\lambda}_{ij} PZ^i_j
				\\
				0
			\end{array}\right).
		\end{equation}

		\begin{eqnarray*}
			\partial_{(\xi_j)_i}\tilde{E}_{\lambda}(\xi) &=& \partial_{(\xi_j)_i}E_{\lambda}(\bW_{\lambda}) \\
			&&+ \int_{\Sigma} \left(- \frac{2}{3} \Delta_g \phi_1 - \frac{1}{3} \Delta_g \phi_2 - \frac 1 2 (f_1(W_{1,\lambda} + \phi_1) - f_1(W_{1,\lambda}))\right) \partial_{(\xi_j)_i} W_{1,\lambda}\, dv_g \\
			&&	+ \int_{\Sigma} \left(- \frac{2}{3} \Delta_g \phi_2 - \frac{1}{3} \Delta_g \phi_1 - \frac 1 2 (f_2(W_{2,\lambda} + \phi_2) - f_2(W_{2,\lambda}))\right) \partial_{(\xi_j)_i}W_{2,\lambda}\, dv_g\\
			&&	+ \int_{\Sigma} \left(- \frac{2}{3} \Delta_g  (W_{1,\lambda}+\phi_1) - \frac{1}{3} \Delta_g(W_{2,\lambda}+\phi_2) - \frac 1 2f_1(W_{1,\lambda}+\phi_1)\right) \partial_{(\xi_j)_i}\phi_1 \, dv_g\\
			&&+ \int_{\Sigma} \left(- \frac{2}{3} \Delta_g(W_{2,\lambda}+\phi_2) - \frac{1}{3} \Delta_g(W_{1,\lambda}+\phi_1)  -\frac 1 2  f_2(W_{2,\lambda}+\phi_2)\right) \partial_{(\xi_j)_i}\phi_2\, dv_g .		
		\end{eqnarray*}
		Lemma~\ref{lem:diff_PZ_partial_PU} implies that
		\[ \partial_{(\xi_j)_i} W_{1,\lambda}=PZ^i_j+\mathcal{O}(1) \text{ and } \partial_{(\xi_j)_i} W_{2,\lambda}=-\frac 1 2 PZ^i_j+\mathcal{O}(1).\]
		Since $\lan \phi_1, PZ^i_j\ran =0, $
		\begin{eqnarray*}
			\int_{\Sigma} (-\Delta_g\phi_1)\partial_{(\xi_j)_i} W_{1,\lambda}\, dv_g &=&\int_{\Sigma}\lan\nabla\phi_l ,(\nabla PZ^i_j+\mathcal{O}(1))\ran_g \, dv_g \\
			&=&\lan \phi_1, PZ^i_j\ran+ \mathcal{O}\left(\|\bphi\| \right)=o(1).
		\end{eqnarray*}
		Similarly, we can deduce that 
		$  \int_{\Sigma} (-\Delta_g\phi_1)\partial_{(\xi_j)_i} W_{2,\lambda}\, dv_g=o(1)$.
		\begin{eqnarray*}
			\int_{\Sigma} (-\Delta_g \phi_2) \left(\partial_{(\xi_j)_i} W_{1,\lambda}+2\partial_{(\xi_j)_i} W_{2,\lambda} \right)\, dv_g &=& \frac  3 2 \int_{\Sigma}(-\Delta_g\phi_2) \partial_{(\xi_j)_i} z(\cdot,\xi)\, dv_g\\
			&=&\mathcal{O}(\|\phi_2\|)=o(1). 
		\end{eqnarray*}
		The estimate~\eqref{eq:non_linear_2rd} yields that  for $q>1$
		\begin{eqnarray*}
			&&\int_{\Sigma}	(f_2(W_{2,\lambda} + \phi_2) - f_2(W_{2,\lambda})) \partial_{(\xi_j)_i}W_{2,\lambda}\, dv_g\\
			&=& \mathcal{O}\left( (\|\phi_2\|^2+ \|f'_2(W_{2,\lambda})\phi_2\|_{L^q(\Sigma)}) \|PZ^i_j\|_{L^{\frac q {q-1}}(\Sigma)}\right)\\
			& =&\mathcal{O}(\|\phi_2\| )=o(1). 
		\end{eqnarray*}
		The estimate~\eqref{eq:est_1st_component_nonlin_term} implies that for $q, r>1$ sufficiently close to $1$
		\begin{eqnarray*}
			&&	\int_{\Sigma}	(f_1(W_{1,\lambda} + \phi_1) - f_1(W_{1,\lambda})) \partial_{(\xi_j)_i}W_{1,\lambda}\, dv_g\\
			&=&2\lambda\int_{\Sigma} V_1 e^{W_{1,\lambda}} \phi_1 \partial_{(\xi_j)_i}W_{1,\lambda}\, dv_g \\
			&&+\mathcal{O}\left( \int_{\Sigma} |f_1(W_{1,\lambda} + \phi_1) - f_1(W_{1,\lambda})-f'_1(W_{1,\lambda})\phi_1 \partial_{(\xi_j)_i}W_{1,\lambda}| \, dv_g\right)\\
			& =&  2\lambda\int_{\Sigma} V_1 e^{W_{1,\lambda}} \phi_1 (\chi_jZ^i_j+\mathcal{O}(1))\, dv_g+ \mathcal{O}\left(  \lambda^{\frac{ 1 -qr}{qr}}\|\phi_1\|^2 \|PZ^i_j\|_{L^{\frac q {q-1}}(\Sigma)} \right)\\
			&=& 2\lambda\int_{\Sigma} V_1 e^{W_{1,\lambda}} \phi_1 \chi_jZ^i_j\, dv_g    +o(1).
		\end{eqnarray*}
		By~Lemma~\ref{lem:diif_e^W_sum_e^u}, for any $q\in (1,2)$
		\begin{eqnarray*}
			&&2\lambda\int_{\Sigma} V_1 e^{W_{1,\lambda}} \phi_1 \chi_jZ^i_j\, dv_g\\
			&=& \int_{\Sigma}\chi_j e^{-\varphi_j} e^{U_j}Z^i_j \phi_1\, dv_g + \mathcal{O}\left(\left\| 2\lambda V_1 e^{W_{1,\lambda}}-\sum_{l=1}\chi_l e^{-\varphi_l} e^{U_l} \right\|_{L^q(\Sigma)}\|\phi_1\|\right)\\
			&=&\lan PZ^i_j, \phi_1\ran + \mathcal{O}(\lambda^{\frac{2-q}{2q}})=o(1). 
		\end{eqnarray*}
		We notice that  $\int_{\Sigma} |PZ^i_j|^q\, dv_g =	\left\{\begin{array}{ll}
			\mathcal{O}(|\log\lambda|), & q=2\\
			\mathcal{O}\left(1\right), & q\in (1,+\infty)\setminus\{2\}
		\end{array}\right.$. 
		Applying \eqref{eq:solution_inf_dim}, we derive that for $j=1,\ldots,m, i=1,\ldots,\i(\xi_{j})$
		\begin{eqnarray*}
			\sum_{i',j'} c^{\lambda}_{i',j'}\lan PZ^i_j,PZ^{i'}_{j'}\ran &=&	\left\lan 	W_{1, \lambda}+\phi_1 -i^*(f_1(	W_{1, \lambda}+\phi_1)- \frac 1 2 f_2(W_{2,\lambda}+\phi_2)), PZ^i_j\right\ran \\
			&=& \left\lan 	W_{1, \lambda} -i^*(f_1(	W_{1, \lambda})-\frac 1 2 f_2(W_{2,\lambda})),PZ^i_j\right\ran+ \mathcal{O}\left( |\lan\phi_1, PZ^i_j\ran|  \right) \\
			&&-\int_{\Sigma} (f_1(W_{1,\lambda} + \phi_l) - f_1(W_{1,\lambda}))PZ^i_j\, dv_g\\
			&& +\frac{1}{2}\int_{\Sigma} (f_2(W_{2,\lambda} + \phi_2) - f_2(W_{2,\lambda}))PZ^i_j\, dv_g . \\
		\end{eqnarray*}
		According to the proof of Lemma~\ref{lem:key_energy_app},  
		\begin{eqnarray*}
			\lan 	W_{1, \lambda} -i^*(f_1(W_{1, \lambda})), PZ^i_j\ran&=& \int_{\Sigma}\left(\sum_{l=1}^m \chi_l e^{-\varphi_l}e^{U_l} - 2\lambda V_1 e^{W_{1, \lambda}}\right)PZ^i_j\, dv_g \\
			&=&2 \partial_{(\xi_j)_i}\Lambda_{k,m}(\xi)+o(1). 
		\end{eqnarray*}
		By \eqref{eq:expansion_eW2}, we have for $q>2$
		\begin{eqnarray*}
			-\frac  1 2\int_{\Sigma} f_2(W_{2,\lambda}) PZ^i_j\, dv_g &=&-\rho_2\int_{\Sigma} \frac{ \tilde{V}_2(\cdot,\xi) e^{z(\cdot,\xi)}} {\int_{\Sigma}\tilde{V}_2(\cdot,\xi) e^{z(\cdot,\xi)\, dv_g}}PZ^i_j \, dv_g +\mathcal{O}\left( \lambda \int_{\Sigma}|PZ^i_j|\, dv_g\right)\\
			&=& \mathcal{O}\left((1+\lambda)\left\|PZ^i_j \right\|_{L^q(\Sigma)}\right)= \mathcal{O}\left(1\right). 
		\end{eqnarray*}
		Applying~\eqref{eq:est_1st_component_nonlin_term} and~\eqref{eq:non_linear_2rd},  we derive that for $q,r>1$ sufficiently close to $1$, 
		\begin{eqnarray*}
			\int_{\Sigma} (f_2(W_{2,\lambda} + \phi_2) - f_2(W_{2,\lambda}))PZ^i_j\, dv_g&=& \mathcal{O}\left( (\|\phi_2\|^2+ \|f'_2(W_{2,\lambda})\phi_2\|_{L^q(\Sigma)}) \|PZ^i_j\|_{L^{\frac q {q-1}}(\Sigma)}\right)\\
			&=& o(1),\\
			\int_{\Sigma} (f_1(W_{1,\lambda} + \phi_l) - f_1(W_{1,\lambda}))PZ^i_j\, dv_g&=&  2\lambda\int_{\Sigma} V_1 e^{W_{1,\lambda}} \phi_1 (\chi_jZ^i_j+\mathcal{O}(1))\, dv_g\\
			&&+ \mathcal{O}\left(  \lambda^{\frac{ 1 -qr}{qr}}\|\phi_1\|^2 \|PZ^i_j\|_{L^{\frac q {q-1}}(\Sigma)} \right)\\
			&=& 2\lambda\int_{\Sigma} V_1 e^{W_{1,\lambda}} \phi_1 \chi_jZ^i_j\, dv_g    +o(1)\\
			&=&\lan PZ^i_j,\phi_1\ran +o(1)=o(1) .
		\end{eqnarray*}
		Applying Lemma~\ref{lem4}, 
		we deduce that 
		\[ c^{\lambda}_{ij} \frac{8\varrho(\xi_j)D_i}{\delta_j^2\pi }+\mathcal{O}\left(\lambda^{-\frac  1 2} \sum_{i,j}|c^{\lambda}_{ij}|\right)=\mathcal{O}(1),\]
		as $\lambda\rightarrow 0$. 
		Considering $\delta_j^2=\lambda d_j^2$, $\sum_{i,j}|c^{\lambda}_{i,j}|=\mathcal{O}(\lambda). $
		By~\eqref{eq:solution_inf_dim}, we have 
		\[  -\frac 2 3 \Delta_g(W_{2,\lambda}+\phi_2) -\frac 1 3 \Delta_g(W_{1,\lambda}+\phi_1) -\frac 1 2 f_2((W_{2,\lambda}+\phi_2))= -\frac 1 3 \sum_{i,j} c^{\lambda}_{ij}\Delta_g PZ^i_j\]
		and 
		\[ - \frac{2}{3} \Delta_g  (W_{1,\lambda}+\phi_1) - \frac{1}{3} \Delta_g(W_{2,\lambda}+\phi_2) - \frac 1 2f_1(W_{1,\lambda}+\phi_1) = -\frac 2 3 \sum_{i,j} c^{\lambda}_{ij}\Delta_g PZ^i_j. \]
		Then, 
		\begin{eqnarray*}
			&&\int_{\Sigma} \left(- \frac{2}{3} \Delta_g  (W_{1,\lambda}+\phi_1) - \frac{1}{3} \Delta_g(W_{2,\lambda}+\phi_2) - \frac 1 2f_1(W_{1,\lambda}+\phi_1)\right) \partial_{(\xi_j)_i}\phi_1\, dv_g\\
			&=& \frac 2 3 \sum_{i,j} c^{\lambda}_{i'j'}\lan  PZ^{i'}_{j'}, \partial_{(\xi_j)_i}\phi_1\ran. \\
		\end{eqnarray*}
		By straightforward calculation,  for $q>1$ sufficiently close to $1$ such that $\frac{2-2q}{2q}+ \frac{2-p}{2p}>0$
		\begin{eqnarray}\label{eq:pz_pa_phi_l}
			\lan PZ^{i'}_{j'}, \partial_{\left(\xi_{j}\right)_{i}} \phi_1 \ran&=&\partial_{(\xi_j)_i}\lan PZ^{i'}_{j'}, \phi_1 \ran- \lan \partial_{\left(\xi_{j}\right)_{i}} PZ^{i'}_{j'}, \phi_1 \ran\\
			&=&\int_{\Sigma} \partial_{\left(\xi_{j}\right)_{i}}( -\chi_{j'}e^{-\varphi_{j'}} e^{U_{j'}} Z^{i'}_{j'} )\phi_1\, dv_g\nonumber \\
			&=&\mathcal{O}\left( \|\partial_{\left(\xi_{j}\right)_{i}}( -\chi_{j'}e^{-\varphi_{j'}} e^{U_{j'}} Z^{i'}_{j'} )\|_{L^q(\Sigma)}\|\phi_1\|\right)\nonumber\\
			& =&\mathcal{O}(\lambda^{\frac{2-4q}{2q}+ \frac{2-p}{2p}} |\log\lambda|)=o(\lambda^{-1}),\nonumber
		\end{eqnarray}
		where we applied the condition $\lan PZ^i_j,\phi\ran= 0.$
		Hence, we obtain that 
		\[ 	\int_{\Sigma} \left(- \frac{2}{3} \Delta_g  (W_{1,\lambda}+\phi_1) - \frac{1}{3} \Delta_g(W_{2,\lambda}+\phi_2) - \frac 1 2f_1(W_{1,\lambda}+\phi_1)\right) \partial_{(\xi_j)_i}\phi_1\, dv_g =o(1).\]
		Similarly, by the same approach, we can deduce that 
		\begin{eqnarray*}
			\int_{\Sigma} \left(- \frac{2}{3} \Delta_g(W_{2,\lambda}+\phi_2) - \frac{1}{3} \Delta_g(W_{1,\lambda}+\phi_1)  -\frac 1 2  f_2(W_{2,\lambda}+\phi_2)\right) \partial_{(\xi_j)_i}\phi_2\, dv_g=o(1). 
		\end{eqnarray*}
		Combining all the estimates above, we prove that 
		\[ \partial_{(\xi_j)} \tilde{E}_{\lambda}(\xi)=\partial_{(\xi_j)_i} E(W_{\lambda})+o(1)= \partial_{(\xi_j)_i}\Lambda_{k,m}(\xi)+o(1). \]
	\end{proof}
	\section{Proof of the main results}
	The	next Lemma shows that the critical point $\xi$ of $\tilde{E}_{\lambda}$  if and only if $\bW_{\lambda}+\bphi_{\xi,\lambda}$ solves~\eqref{eq:toda_equi}. 
	\begin{lem}\label{lem:equi_critical}
		There exists
		$\xi \in\Xi_{k,m}$ be a critical point of $\xi \mapsto \tilde{E}_{\lambda}(\xi)$ if and only if there exists $\varepsilon_0>0$ such that  $\xi\in \Xi_{k,m}^{\varepsilon_0}$ and   $\bu = \bW_{\lambda} +\bphi_{\xi,\lambda}$ constructed by Theorem~\ref{thm2} solves~\eqref{eq:toda_equi}.
	\end{lem}
	\begin{proof}
		Suppose that $\xi$ is a critical point of~$\tilde{E}_{\lambda}$. Then for $j=1,\ldots,m, i=1,\ldots, \i(\xi_j)$, Theorem~\ref{thm2} implies that 
		\begin{eqnarray*}
			0= \partial_{(\xi_j)_i} \tilde{E}_{\lambda}&=& \lan
			\bu -i^*(E(\bu)) , \partial_{(\xi_{j})_i} (\bW_{\lambda}+\bphi_{\xi,\lambda})\ran\\
			&\stackrel{\eqref{eq:solution_inf_dim}}{=}&
			\sum_{i',j'} c^{\lambda}_{i'j'} \lan PZ^{i'}_{j'}, \partial_{(\xi_{j})_i} (W_{1, \lambda}+(\bphi_{\xi,\lambda})_1)\ran  \\
			&=& \sum_{i',j'} c^{\lambda}_{i'j'} \lan PZ^{i'}_{j'},PZ^i_j\ran +\mathcal{O}\left(\sum_{i',j'} c^{\lambda}_{i',j'} (\lambda |\log\lambda|+|\lan PZ^{i'}_{j'}, \partial_{(\xi_j)_i} (\bphi_{\xi,\lambda})_1 \ran|\right)
			\\
			&&    \text{\small (by Lemma~\ref{lem4} and~\eqref{eq:pz_pa_phi_l})} \\
			&=&  c^{\lambda}_{ij}\frac{8D_i\varrho(\xi_j)}{\pi\delta_j^2}+\mathcal{O}\left(\lambda^{-\frac 1 2 }\sum_{i',j'}|c^{\lambda}_{i'j'}|\right) +o\left(\lambda^{-1}\sum_{i',j'}|c^{\lambda}_{i'j'}| \right).
		\end{eqnarray*}
		By the arbitrariness of $i,j$, it follows 
		\[ \sum_{i,j}c^{\lambda}_{ij}\frac{8D_i\varrho(\xi_j)}{\pi\delta_j^2}\left(1+\mathcal{O}(\lambda^{\frac 1 2})+o(1) \right)=0.\] 
		There exists $\lambda_0>0$ sufficiently small such that 
		for $\lambda\in (0,\lambda_0)$, 
		\[ c^{\lambda}_{ij}=0 \quad (j=1,\ldots,m, i=1,\ldots,\i(\xi_j)). \]
		Considering~\eqref{eq:solution_inf_dim}, $\bu$ solves \eqref{eq:toda_equi}. 
		Conversely, we assume that for some $\xi\in\Xi_{k,m}$, $\bu=\bW_{\lambda}+\bphi_{\xi,\lambda}$ constructed by Theorem~\ref{thm2} is a solution of~\eqref{eq:toda_equi}.
		It is easy to see that $\bu-i^*(F(\bu))=0$.
		Consequently, for $j=1,\ldots,m, i=1,\ldots,\i(\xi_j)$
		\[ \partial_{(\xi_i)_j}\tilde{E}_{\lambda}(\xi)= \lan
		\bu -i^*(E(\bu)) , \partial_{(\xi_{j})_i} (\bW_{\lambda}+\bphi_{\xi,\lambda})\ran=0, \]
		which means $\xi$ is a critical point of $\tilde{E}_{\lambda}$\end{proof}
	
	\begin{altproof}{Theorem~\ref{thm:main}}
		It is clear that $\cC:=\{\xi\}$ is a $C^1$-stable critical point set of $\Lambda_{k,m}$. For any $n\in \mathbb{N}_+$, there exist a sequence $\lambda_n\rightarrow 0$ and  $\xi_{\lambda_n}=\left(\xi_{1, \lambda_n}, \ldots, \xi_{m, \lambda_n}\right) \in \Xi^{\varepsilon_0}_{k,m}$ such that $d_g(\xi_{\lambda_n},\cC)<\frac{1}{n}$ and $\xi_{\lambda_n}$ is a critical point of $\tilde{E}_{\lambda_n}: \Xi_{k,m} \rightarrow \mathbb{R}$. Without loss of generality, assume that 
		\[ \xi_{\lambda_n}=\left(\xi_{1, \lambda_n}, \ldots, \xi_{m, \lambda_n}\right)\rightarrow \xi=(\xi_1,...,\xi_m)\in \cC, \]
		as $n\rightarrow +\infty.$
		Then define 
		$\bu^n=\bW_{\lambda_n}+\bphi_{\xi^n,\lambda_n}$ through Theorem~\ref{thm2}.
		Lemma~\ref{lem:equi_critical} yields that 
		$u_{\lambda_n}$ solves \eqref{eq:toda_equi} as $\lambda_n\rightarrow 0$. 
		Let $\rho^n_1= \lambda_n \int_{\Sigma} V_1 e^{u^n_1}\, dv_g.$
		By \eqref{eq:mass_1st_component} and  $\left|\mathrm{e}^{s}-1\right| \leqslant \mathrm{e}^{|s|}|s|$ for any $s \in \mathbb{R}$, for $q>1$ sufficiently close to $1$
		\begin{eqnarray*}
			\rho^n_1&=& \lambda_n \int_{\Sigma} V_1 e^{W_{1,\lambda_n}}\, dv_g+ \mathcal{O}\left(  \lambda_n \int_{\Sigma} V_1 e^{W_{1,\lambda_n}} |(\bphi_{\xi^n,\lambda_n})_1|\, dv_g\right)\\
			&=&\frac  1 2 \sum_{j=1}^m \varrho(\xi_j)+\mathcal{O}\left(\|\lambda_n V_1 e^{W_{1,\lambda_n}}\|_{L^q(\Sigma)}\|\bphi_{\xi^n,\lambda_n}\|\right)\\
			&=& \frac  1 2 \sum_{j=1}^m \varrho(\xi_j)+o(1).
		\end{eqnarray*}
		For any $\Psi \in C(\Sigma),$
		using Lemma~\ref{lem:diif_e^W_sum_e^u}, we have 
		\begin{eqnarray*}
			2\rho^n_1\int_{\Sigma} \frac{V_1 \mathrm{e}^{u_1^{n}} } { \int_{\Sigma}V_1 \mathrm{e}^{u_1^{n}}\, dv_g} \Psi dv_g& =&2 \lambda_n \int_{\Sigma} V_1 \mathrm{e}^{u_1^{n}}\Psi dv_g\\
			& =&
			\sum_{ j=1}^m \int_{\Sigma} \chi\circ y_{\xi^n_j}^{-1} e^{-\varphi_{\xi^n_j}}e^{U_{\delta_j(\xi^n), \xi^n_j}}\Psi dv_g +o(1) \\
			&=& \sum_{j=1}^m \varrho(\xi_j) \Psi\left(\xi_j\right)+o(1) \quad (n\rightarrow 0). 
		\end{eqnarray*} 
		By~\eqref{eq:expansion_W2}, we can obtain the estimate~\eqref{eq:expansion_u2}.  The proof is concluded.
	\end{altproof}
	\begin{altproof}{Corollary~\ref{cor:1}}
		For any $z\in \oH$,  the Moser-Trudinger inequality implies that 
		\begin{eqnarray*}
			I_{\xi}(z)&=& \frac 1 2 \int_{\Sigma}|\nabla z|^2_g\, dv_g -2\rho_2\log\left( \int_{\Sigma} \tilde{V}_2(\cdot,\xi) e^{z}\right)\geq C,
		\end{eqnarray*} 
		for some constant $C$ (uniformly for $\xi\in \Xi_{k,m}$). 
		For any $\rho_2 \in (0,2\pi)$, $z(\cdot,\xi)$ exists but may be degenerate for $\xi\in \Xi_{k,m}$. 
		We fix an arbitrary $z=z(\cdot,\xi)$ solves~\eqref{eq:singular_mf}. 
		Since $z(\cdot,\xi)$ is a minimizer of $I_{\xi}$, 
		\begin{eqnarray*}
			I_{\xi}(z(\cdot,\xi))\leq I_{\xi}(0)&=& -2\rho_2 \log\left( \int_{\Sigma} V_2 e^{\sum_{j=1}^m \frac 1 2 \varrho(\xi_j) G^g(\cdot,\xi_j)}\right)\\
			&\leq& \left| 2\rho_2 \left(\sum_{j=1}^m \frac 1 2\varrho(\xi_j) \int_{\Sigma} G^g(\cdot,\xi_j)\, dv_g  +\sum_{j=1}^m\int_{\Sigma} \log V_2 \, dv_g \right)\right| \\
			&&\text{\small ( by Jensen's inequality and $ \int_{\Sigma} G^g(\cdot,\xi_j)\, dv_g =0$) }\\
			&=& 2m \rho_2 \max_{\Sigma}|\log V_2|. 
		\end{eqnarray*}
		Let $\fJ_{\xi,\rho_2}:=\{ z\in \oH: z \text{ solves \eqref{eq:singular_mf} for $\xi$ and $\rho_2$}\}.$
		So we obtain that for some constant $C_0>0$
		\[ \sup_{\xi\in \Xi_{k,m}, \rho_2\in (0,2\pi
			)}\sup_{z\in\fJ_{\xi,\rho_2}}|I_{\xi}(z)|\leq C_0.\]
		By~\cite[Lemma 7.3]{HBA2024}, we know that $\mathcal{F}_{k,m}(\xi)\rightarrow +\infty$ as $\xi\rightarrow \partial\Xi_{k,m}.$
		We recall that 
		\[ \Lambda_{k,m}(\xi)= \frac 1 2 I_{\xi}(z(\cdot,\xi)) -\frac 1 4 \cF_{k,m}(\xi). \]
		Fix an arbitrary point $\xi^0\in \Xi_{k,m}$. 
		There exists $\varepsilon_0>0$ sufficiently small  such that $\xi^0\in \Xi_{k,m}^{\varepsilon_0}$ and for any $\xi\in \Xi_{k,m}\setminus\Xi_{k,m}^{\varepsilon_0}$ 
		\[ \mathcal{F}_{k,m}(\xi)> 4 C_0 + \mathcal{F}_{k,m}(\xi^0). \]
		We take $\Dc$ to be the interior of $\Xi^{\frac 1 2 \varepsilon_0}_{k,m}$.
		Let $$\cC:=\{\xi\in \Xi_{k,m}: \cF_{k,m}(\xi)=\inf_{\Xi_{k,m}}\Lambda_{k,m}\}\subset\Xi^{\varepsilon_0}_{k,m}\subset\Dc.$$  It is easy to see that $\xi $  is a non-degenerate critical points set of $\Lambda_{k,m}$. 
		By Lemma~\ref{lem:hypo_H_small_rho2}, there exists $\rho_0>0$ such that for any $\rho_2\in (0,\rho_0)$,  the hypothesis~\hyperref[item:H]{(C1)} holds for $\Dc$. Applying~Theorem~\ref{thm:main}, we complete the proof. 
	\end{altproof}
	\begin{altproof}
		{Corollary~\ref{cor:2}}

		As in the proof of  Theorem~\ref{thm:residual},  we have 
		$$\cV_{reg}=\{ (V_1, V_2)\in C^{2,\alpha}(\Sigma, \R_+)\times C^{2,\alpha}(\Sigma, \R_+):\eqref{eq:shadow_t}\text{ is non-degenerate for } t\in [0,1]\cap \Q \}$$
		is a residual set in $C^{2,\alpha}(\Sigma, \R_+)\times C^{2,\alpha}(\Sigma, \R_+)$. We fix an arbitrary $(V_1, V_2)\in \cV_{reg}$. Using Proposition~\ref{prop:ex_sols_shadow}, 
		there exists a non-degenerate  solution $(w,\xi)$ of~\eqref{eq:shadow}. The hypothesis  \hyperref[item:G]{(C1)} holds if one of the conditions $a)$ or $b)$ is satisfied.  
		Using Theorem~\ref{thm:main}, we can conclude Corollary~\ref{cor:2}.  
	\end{altproof}
	\section{Appendix}
	\subsection{Some important estimates}
	The following lemma is the asymptotic expansion of $PU_j$. 
	\begin{lem}\label{lem:extension_PU}
		$PU_j= \chi_j ( U_j-\log( 8 \delta_j^2))+\varrho(\xi_j) H^g(\cdot,\xi_j)+ \mathcal{O}(\delta_j^2|\log \delta_j|)$ as $\delta_j\rightarrow 0.$
		For any $x\in \Sigma\setminus\{ \xi_j\}$, 
		\[ PU_j= \varrho(\xi_j) G^g(\cdot,\xi_j)+ \mathcal{O}(\delta_j^2|\log \delta_j|),\]
		as $\delta_j\rightarrow 0.$
	\end{lem}
	\begin{proof}
		Let $\eta_{j}= PU_j-\chi_j\cdot( U_j-\log( 8 \delta_j^2)) -\varrho(\xi_j) H^g(\cdot,\xi_j) $.
		
		If $\xi_j\in \intsigma$,  $\partial _{\nu_g}\eta_{j} \equiv 0$ on $\partial \Sigma$.  
		We observe that for any $x\in \partial\Sigma\cap U(\xi_j)$
		\[\partial_{\nu_g} |y_{\xi_j}(x)|^2=- e^{-\frac 1 2 \hat{\varphi}_{\xi_j}(y)} \left. \frac{\partial}{\partial y_2} |y|^2\right|_{y=y_{\xi_j}(x)}= 0.\]
		If $\xi_j\in \partial \Sigma$,  
		for any $x\in \partial \Sigma$, we have  as $\delta_j\rightarrow 0$
		\begin{eqnarray*}
			\partial_{\nu_g}\eta_{j}(x)
			&=& 2\partial_{\nu_g}\left( \chi_j \log\left( 1+\frac{\delta_j^2}{|y_{\xi_j}(x)|^2}\right)\right)\\
			&=&2(\partial_{ \nu_g}\chi_{j}) \frac{\delta_j^2}{|y_{\xi_j}(x)|^2} -2\chi_{j} \partial_{ \nu_g}\log\left( 1+\frac{\delta_j^2}{|y_{\xi_j}(x)|^2}\right)+\mathcal{O}(\delta_j^{4})\\
			&=&\mathcal{O}(\delta_j^2). 
		\end{eqnarray*}
		Thus, for any $i=1,2,j=1,\ldots,m$	$
		\partial_{ \nu_g} \eta_{j} =
		\mathcal{O}(\delta_j^2) \text{ as } \delta_j\rightarrow 0.$
		\begin{eqnarray*}
			&&\int_{\Sigma} \eta_{j}\, dv_g =  \int_{\Sigma} 2\chi_j  \log\left( 1+\frac{\delta_j^2}{|y_{\xi_j}(x)|^2}\right)\, dv_g(x) \\
			&=& 2\int_{B^{\xi_j}_{r_0}} e^{\hat{\varphi}_{\xi_j}(y)}  \log\left( 1+ \frac{\delta_j^2 }{ |y|^2}\right)  dy +2 \int_{B^{\xi_j}_{2r_0}\setminus B_{r_0}(0)} \chi(|y|/r_0) e^{\hat{\varphi}_{\xi_j}(y)}
			\left( \frac{\delta_j^2 }{ |y|^2}+\mathcal{O}(\delta_j^{4})\right) dy\\
			&=&  2\delta_j^2  \int_{\frac 1 {\delta_j}(B^{\xi_j}_{r_0}\cap B_{r_0}(0))} \log \left( 1+ \frac 1 {|y|^2}\right) e^{-\hat{\varphi}(\delta_j y)} dy +\mathcal{ O}(\delta_j^2)\\
			&=& 2\delta_j^2 (1+\mathcal{O}(\delta_j)) \int_{B_{r_0/\delta_j}(0) }\log \left( 1+\frac 1 {|y|^2}\right) dy+\mathcal{O}(\delta_j^2) \\
			&=&{\mathcal{O}}(\delta_j^2|\log \delta_j|), 
		\end{eqnarray*}
		where we applied the fact that 
		\begin{eqnarray*}
			0&\leq& \int_{|y|<\frac {r_0}{\delta_j}}     \log \left( 1+\frac 1 {|y|^2}\right) dy 
			=2\pi \int_0^{r_0/\delta_j} \log   \left( 1+\frac 1 {r^{2}}\right)r  dr\\
			&\leq & \pi \int_0^{r^2_0/(\tau\rho)^2} \log   \left( 1+\frac 1 {t}\right) dt \leq 2\pi\int_1^{r_0/\delta_j} r^{-1} dr +\mathcal{O}(1)\leq \mathcal{O}(|\log \delta_j|).
		\end{eqnarray*}
		For any $x\in U_{2r_0}(\xi)$, $-\Delta_g U_j  =e^{-\varphi_{j}} e^{U_j}$.  It follows that 
		\begin{eqnarray*}
			-\Delta_{g} \eta_{j}&=&
			2(\Delta_g\chi_j )\log\frac{|y_{\xi_j}|^2}{\delta_j^2+|y_{\xi_j}|^2}+4\left\lan \nabla\chi_j,\nabla\log\frac{|y_{\xi_j}|^2}{\delta_j^2+|y_{\xi_j}|^2}\right\ran_g\\
			&&+\frac 1 {|\Sigma|_g} \left( \varrho(\xi_j)- \int_{\Sigma} \chi_j e^{-\varphi_j}  e^{U_j} dv_g \right). 
		\end{eqnarray*}
		We observe that $\Delta_{g} \chi_j\equiv 0$ and $\nabla \chi_j\equiv 0$ in $U_{2r_0}(\xi_j)\setminus U_{r_0}(\xi_j)$. 
		For any $x\in U_{2r_0}(\xi_j)\setminus U_{r_0}(\xi_j)$, we have as $\delta_j\rightarrow 0$
		\[ -2\log\left( 1+ \frac{\delta_j^2}{|y_{\xi_j}(x)|^2}\right)
		= -2\delta_j^2  |y_{\xi_j}(x)|^{-2} +\mathcal{O}(\delta_j^{4}) \] 
		and 
		\[-2\nabla \log\left( 1+ \frac{\delta_j^2}{|y_{\xi_j}(x)|^2}\right)
		= -2\delta_j^2 \nabla |y_{\xi_j}(x)|^{-\alpha_i} +\mathcal{O}(\delta_j^{4}).  \] 
		Moreover, a straightforward calculation of the integral implies that 
		\begin{eqnarray*}
			\int_{\Sigma} \chi_j e^{-\varphi_j} e^{U_j} dv_g 
			&=&\int_{B_{2r_0}^{\xi}} 8 \chi(|y|/r_0) \frac{\delta_j^2}{(\delta_j^2 +|y|^2)^2} dy \\
			&=&  \int_{B_{r_0}^{\xi}} 8 \frac{\delta_j^2}{(\delta_j^2 +|y|^2)^2} dy +\mathcal{ O}(\delta_j^2 )\\
			&=& \varrho(\xi_j)+{\mathcal{O}}(\delta_j^2),
		\end{eqnarray*}
		where we applied the fact that $\int _{|y|<r}8 \frac{\delta_j^2}{(\delta_j^2 +|y|^2)^2}  \, dy = 8\pi \left(1- \frac{\delta_j^2}{\delta_j^2+r^{2}}\right)$ for any $r\geq 0.$
		Hence,  as $\delta_j\rightarrow 0$
		\[ -\Delta_g\eta_{j}= \mathcal{O}(\delta_j^2). \]
		By the Schuader theory in~\cite{Nardi2014}, we derive that as $\delta_j\rightarrow 0$
		\begin{equation*}
			\eta_{j}= \mathcal{O}(\delta_j^2|\log \delta_j|). 
		\end{equation*}
		
	\end{proof} 
	\begin{lem}\label{lem:extension_PZ} For $j=1,\ldots,m$ and $i=1,\ldots, \i(\xi_j)$
		{\it
			\[PZ^0_j(x)=\chi_j (x) \left(Z^0_i(x) +1\right) +{\mathcal{O}}(\delta_j^2|\log \delta_j |) = 4\chi_j(x)\frac{\delta_j^2}{ \delta_j^2 +|y_{\xi_j}(x)|^2}+{ \mathcal{O}} (\delta_j^2|\log \delta_j |), \] 
			in $C^1(\Sigma)$  as $\delta_j\rightarrow 0$. 
			And 
			\[ PZ^0_j(x)=  (\delta_j^2|\log \delta_j |), \]
			in $C^1_{loc}(\Sigma \setminus\{\xi_j\})$ as $\rho\rightarrow 0$.
			
			\[PZ_{j}^i(x)
			= \chi_j(x)Z_{j}^i (x)+ {\mathcal{O}}(1),
			\] 
			in $C^1(\Sigma)$ as $\delta_j\rightarrow 0$, and 
			\[PZ_{j}^i (x)=\mathcal{O}(1), \]
			in $C^1_{loc}(\Sigma\setminus\{\xi_j\})$ as $\delta_j\rightarrow 0$. 
			In addition, the convergences above are 
			uniform for $\xi_j$ in any compact subset of $\intsigma$ or  $\xi_j\in \partial \Sigma$. }
	\end{lem}
	\begin{proof}
			Since the reasoning is almost identical to that in Lemma \ref{lem:extension_PU}, we omit the details here for the sake of brevity. 
	\end{proof}
	\begin{lem}\label{lem:diff_PZ_partial_PU}
		As $\lambda\rightarrow 0$, for $j,j'=1,\ldots,m$ and $i=1,\ldots, \i(\xi_j)$, 
		\[ \partial_{(\xi_j)_i}PU_{j'}= \delta_{j'j}PZ^i_j -PZ^0_{j'} \partial_{(\xi_j)_i}\log  d_{j'}+\mathcal{O}(1),\]
		in $C(\Sigma)$ uniformly for  $\xi$ in any compact subset of $\Xi_{k,m}. $
	\end{lem}
	\begin{proof}
		We observe that 
		$\partial_{(\xi_j)_i}PU_j=\delta_{j'j}\partial_{(\xi_j)_i} PU_{\xi_j,\rho}|_{\rho=\delta_j}+ \partial_{\rho}PU_{\xi_{j'},\rho}|_{\rho=\delta_{j'}}\lambda^{\frac 1 2} \partial_{(\xi_j)_i} d_{j'}.$
		Since $$\partial_{(\xi_j)_i}|y_{\xi_j}(x)|^2=-2(y_{\xi_j})_i+\mathcal{O}(|y_{\xi_j}|^3),$$
		for any $x\in U(\xi_j)$, we have 
		\begin{eqnarray*}
			\partial_{(\xi_j)_i} U_{\xi_j,\rho}|_{\rho=\delta_j}= \frac{4(y_{\xi_j})_i}{\delta_j^2+|y_{\xi_j}|^2}
			+\mathcal{O}\left(\frac{|y_{\xi_j}|^2}{\delta_j^2+|y_{\xi_j}|^3}\right)= Z^i_j	+\mathcal{O}\left(\frac{|y_{\xi_j}|^2}{\delta_j^2+|y_{\xi_j}|^3}\right) . 
		\end{eqnarray*}
		By~\eqref{eq:proj}, we derive that 
		\begin{eqnarray*}
			-\Delta_g(\partial_{(\xi_j)_i} PU_{\xi_j,\rho}|_{\rho=\delta_j} -PZ^i_j )
			&=& \chi_j e^{-\varphi_j}  e^{U_{\xi_j,\delta_j}}(\partial_{(\xi_j)_i}U_j-Z^i_j) -\overline{\chi_j e^{-\varphi_j}  e^{U_{\xi_j,\delta_j}}(\partial_{(\xi_j)_i}U_j-Z^i_j)}\\
			&&+\mathcal{O}(1)\\
			&=& \mathcal{O}\left( \frac{\delta_j^2 |y_{\xi_j}(x)|^3}{(\delta_j^2+|y_{\xi_j}|^2)^3}\right)+\mathcal{O}(1). 
		\end{eqnarray*}
		For any $p\in (1,2)$, 
		$\|-\Delta_g (\partial_{(\xi_j)_i} PU_{\xi_j,\rho}|_{\rho=\delta_j} -PZ^i_j )\|=\mathcal{O}(1)$.
		$
		\int_{\Sigma} (\partial_{(\xi_j)_i} PU_{\xi_j,\rho}|_{\rho=\delta_j} -PZ^i_j )\, dv_g = 0.
		$
		For any $x\in \partial\Sigma$, 
		$
		\partial_{\nu_g} (\partial_{(\xi_j)_i} PU_{\xi_j,\rho}|_{\rho=\delta_j} -PZ^i_j )= 0.
		$
		Hence, applying the $L^p$-theory,
		\[ \| \partial_{(\xi_j)_i} PU_{\xi_j,\rho}|_{\rho=\delta_j} -PZ^i_j\|_{W^{2,p}(\Sigma)} =\mathcal{O}(1). \]
		By the Sobolev's embedding, we have 
		$ \partial_{(\xi_j)_i} PU_{\xi_j,\rho}|_{\rho=\delta_j}=PZ^i_j+\mathcal{O}(1),$
		in $C^1(\Sigma)$. 
		$\delta_j\partial_{\rho}U_{\xi_j,\rho}|_{\rho=\delta_j} = \frac{2(|y_{\xi_j}|^2-\delta_j^2)}{|y_{\xi_j}|^2+\delta_j^2} = -Z^0_j.$
		It follows that
		$\delta_j\partial_{\rho}U_{\xi_j,\rho}|_{\rho=\delta_j}=PZ^0_j. $
		Therefore, 
		\[\partial_{(\xi_j)_i}PU_j= PZ^i_j -PZ^0_j \partial_{(\xi_j)_i} \log d_j  \]
		and for $j'\neq j$
		\[\partial_{(\xi_j)_i}PU_{j'}=  -PZ^0_{j'} \partial_{(\xi_{j})_i} \log d_{j'}.  \]
	\end{proof}
	The asymptotic ``orthogonality" properties of $PZ_{ij}$ as $\rho\rightarrow 0^+$. 
	\begin{lem}~\label{lem4}
		As $\rho\rightarrow 0$ for $j,j'=1,\ldots,m$, $i=0,\ldots,\i(\xi_j)$ and $i'=0,\ldots, \i(\xi_{j'}$
		\[ \langle PZ^i_j, PZ^{i'}_{j'}\rangle=	\left\{\begin{array}{ll}
			\delta_{i'i}\delta_{j'j} \frac{ 8\varrho(\xi_j)D_i}{\pi}+ {\mathcal{O}}(\lambda\log |\lambda|) & \text{ when }  i \text{ or } i'=0\\
			\delta_{i'i}\delta_{j'j} \frac{ 8\varrho(\xi_j)D_i}{\pi\delta_j^2} + {\mathcal{O}}(\lambda^{-\frac 1 2})& \text{ otherwise } 
		\end{array}\right.,
		\] 
		where the $\delta_{ij}$ is the Kronecker symbol, and  $D_0=\int_{\mathbb{R}^2}  \frac{1-|y|^2}{ (1+|y|^2)^4 }dy$, $D_1= D_2=\int_{\mathbb{R}^2} \frac{ |y|^2 }{ (1+|y|^2)^4}dy . $
	\end{lem}
	\begin{proof}
		\begin{eqnarray*}
			\langle PZ^i_j, PZ^{i'}_{j'}\rangle&=& \int_{\Sigma} \chi_je^{-\varphi_j} e^{U_j} Z^i_j PZ^{i'}_{j'} dv_g\\
			&=& \int_{\Sigma\cap U_{2r_0}(\xi_j)} +\int_{\Sigma\setminus U_{2r_0}(\xi_j)}  \chi_je^{-\varphi_j} e^{U_j} Z^i_j PZ^{i'}_{j'} dv_g. 
		\end{eqnarray*}
		For $i=i'=0$, by Lemma~\ref{lem:extension_PZ}, 
		\begin{eqnarray*}
			&&\int_{\Sigma\cap U_{2r_0}(\xi_j) }\chi_je^{-\varphi_j} e^{U_j} Z^i_j PZ^{i'}_{j'} dv_g\\
			&= &16 
			\int_{B^{\xi_j}_{2r_0}} \chi\left(\frac{|y|}{r_0}\right)  \frac{ |y|^2-\delta_j^2}{(\delta_j^2+|y|^2)^3} \left( \frac{ 4\delta_{j'j}\delta_{j}^2 \chi\left(\frac{|y|}{r_0}\right)}{ \delta_j^2 +|y|^2} +\mathcal{O}(\lambda|\log\lambda|)\right) dy \\
			&= & 64 \delta_{j'j} \int_{\Omega_j}  \frac{ 1-|y|^2}{(1+|y|^2)^4}\, dy 
			+ \mathcal{O}( \lambda|\log\lambda|).
		\end{eqnarray*} 
		As  $\lambda \rightarrow 0$,
		\[ 	\lan PZ^0_j, PZ^0_{j'}\ran =\delta_{j'j}\frac{8\varrho(\xi_j)D_0}{\pi} +{\mathcal{O}}(\lambda|\log\lambda|), \]
		where $D_0=\int_{\mathbb{R}^2}  \frac{1-|y|^2}{ (1+|y|^2)^4 }\, dy.$ \par 
		Similarly, for $i'=0$ and $i=1,2$ for $\xi_j\in\intsigma$ and $i=1$ for $\xi_j\in\partial\Sigma$, we have 
		\begin{eqnarray*}
			\lan PZ^i_j, PZ^{i'}_{j'}\ran &=& 	\int_{\Sigma\cap U_{2r_0}(\xi_j)} \chi_j   e^{-\varphi_j} e^{U_j} Z^i_j PZ^{i'}_{j'}\,  dv_g\\
			&=& 32\int_{ B_{2r_0}^{\xi_j}} \chi\left(\frac{|y|}{r_0}\right) \frac{y_i}{( \delta_j^2 +|y|^2)^3} \left(-\frac{ 4\delta_{j'j}\delta_j^2\chi\left(\frac{|y|}{r_0}\right)  }{ \delta_j^2 +|y|^2} \right.\\
			&&\left. +\mathcal{O}(\lambda|\log\lambda|) \right) dy=\mathcal{O}(\lambda|\log\lambda|),
		\end{eqnarray*}
		where we applied the symmetric property
		$\int_{ B_{2r_0}^{\xi_j}} \chi\left(\frac{|y|}{r_0}\right) \frac{y_i}{( \delta_j^2 +|y|^2)^4}\, dy=0. $
		Applying Lemma~\ref{lem:extension_PZ}, for $i=1,\ldots,\i(\xi_j)$, $i'=1,\ldots,\i(\xi_{j'})$, 
		\begin{eqnarray*}
			&&	\int_{\Sigma\cap U_{2r_0}(\xi)} \chi_j e^{-\varphi_j} e^{U_j} Z^i_j PZ^{i'}_{j'}dv_g(x)\\
			&
			=& 32\int_{ B_{2r_0}^{\xi_j}} \chi\left(\frac{|y|}{r_0}\right) \frac{\delta_j^2 y_i}{( \delta_j^2+|y|^2)^3} \left(\delta_{j'j }\chi\left(\frac{|y|}{r_0}\right) 
			\frac{ 4y_{i'}}{ \delta_j^2 +|y|^2}  +{\mathcal{O}}(1) \right)\, dy\\
			&=& \frac{128}{\delta_j^2} \int_{ \frac 1 {\delta_j}B_{r_0} ^{\xi}}  \frac{\delta_{j'j}y_iy_{i'} }{(1+|y|^2)^4} dy + \mathcal{O}(\lambda^{-\frac 1 2})\\
			&=& \frac{ 8\delta_{i'i}\delta_{j'j}\varrho(\xi_j)D_i}{\pi\delta_j^2} \delta_{ij}+
			\mathcal{O}(\lambda^{-\frac 1 2 })\quad ( \lambda\rightarrow 0),
		\end{eqnarray*}
		where  $D_i= \int_{\mathbb{R}^2} \frac{ |y|^2 }{ (1+|y|^2)^4}  dy . $
	\end{proof}
	\begin{lem}
		\label{lem:diif_e^W_sum_e^u} 
		There exists $p_0>1$ such that for any $p\in (1,p_0)$
		$$\left\|2 \lambda V_1e^{W_{1,\lambda}}-\sum_{j=1}^m\chi_j e^{-\varphi_j} e^{U_j} \right\|_{L^p(\Sigma)}=\mathcal{O}\left(\lambda^{\frac{2-p}{2p}}\right),$$
		as $\lambda \rightarrow 0.$
	\end{lem}
	\begin{proof}
		Applying Lemma~\ref{lem:extension_PU}, we have   $x\in\Sigma\setminus \bigcup_{j=1}^m U_{r_0}(\xi_j)$,
		\[ \begin{aligned}
			W_{1,\lambda}(x) &= \sum_{j=1}^m PU_j -\frac 1 2 z(\cdot,\xi)=\sum_{j=1}^m  \varrho(\xi_j)G^g(x,\xi_j) -\frac 1 2 z(x,\xi) +\mathcal{O}\left( \delta_{j}^2|\log\delta_j| \right)= \mathcal{O}(1). 
		\end{aligned} \]
		Let $\Theta_{j}$ be defined by~\eqref{eq:def_theta}. 
		The estimate~\eqref{eq:est_theta} leads to 
		\begin{eqnarray*}
			&&\int_{\Sigma} \left|  2 \lambda V_1e^{W_{1,\lambda}}-\sum_{j=1}^m\chi_j e^{-\varphi_j} e^{U_j}\right|^p \, dv_g=\sum_{j=1}^m \int_{U_{r_0}(\xi_j)} \left| 2 \lambda V_ie^{W_{i,\lambda}}- e^{-\varphi_j} e^{U_j}  \right|^p \, d v_g \\
			&&+ \mathcal{O}\left(\lambda^p+\delta_{j}^{2p}\right)\\
			&=& \sum_{j=1}^m \int_{\frac 1 {\delta_j}B^{\xi_j}_{r_0}(\xi_j)} \left| 2 \lambda V_1\circ y_{\xi_j}^{-1}(\delta_j y) e^{\varphi_{\xi_j}\circ y_{\xi_j}^{-1}(\delta_j y)}  e^{W_{i,\lambda}\circ y_{\xi_j}^{-1}(\delta_j y)}-  e^{U_j\circ y_{\xi_j}^{-1}(\delta_j y)}  \right|^p \, d v_g\\
			&& + \mathcal{O}\left(\lambda^p+\delta_{j}^{2p}\right) \\
			&=&8\sum_{j=1}^m \int_{\Omega_j} \frac{|\delta_j|^{2-2p}}{(1+|y|^2)^{2p}}\left|e^{\Theta_{j}(y)}-1 \right|^p \, dy + \mathcal{O}\left(\lambda^p+\delta_j^{2p}\right) \\
			&=& 8\sum_{j=1}^m \int_{\Omega_j} \frac{|\delta_j|^{2-p}|y|^p}{(1+|y|^2)^{2p}}\ \, dy + \mathcal{O}\left(\lambda^p|\log\lambda|^p\right)= \mathcal{O}(\lambda^{\frac{2-p}{2}}).
		\end{eqnarray*}
	\end{proof}

	\subsection{Compactness of a shadow system}
	\begin{prop}\label{prop:cpt_shadow}
		Given any $\alpha\in (0,1)$,	there exist constants $ C, \delta > 0 $ such that for any  $ t \in [0, 1] $ and any solution $ (w^t, \xi^t) $ of \eqref{eq:shadow_t}, we have the following prior estimates, 
		\[
		\|w^t\|_{C^{2,\alpha}(\Sigma)} < C.
		\]
	\end{prop}
	\begin{proof}
		Define that 
		\[ \tilde{w}^t= w^t -\log\left(\int_{\Sigma} V_2e^{w^t-\sum_{i=1} \frac{\varrho(\xi_i) }{2} G^g(\cdot,\xi_i)} dv_g\right).\]
		It follows that $\int_{\Sigma} V_2 e^{\tilde{w}^t -\sum_{i=1} \frac{\varrho(\xi_i) }{2} G^g(\cdot,\xi_i)} dv_g =1.$
		We can rewrite the system~\eqref{eq:shadow_t} as follows: 
		\begin{equation}
			\label{eq:shadow_t_re}
			\left\{
			\begin{array}{ll}
				-\Delta_g \tilde{w}^t =2\rho_2 \left( {V_2 e^{ \tilde{w}^t- \sum_{j=1}^m  \frac {\varrho(\xi_j)}2 G^g(x,\xi_j)}} - 1\right)&  \text{ in } \intsigma\\
				\partial_{\nu_g} \tilde{w}^t=0 &\text{ on } \partial\Sigma\\
				\left. \nabla\left( \cF_{k,m}(x)- \sum_{i=1}^m(1-t)\varrho(x_i) \tilde{w}^t(x_i)\right)\right|_{x=\xi}=0	& \text{ in }\Xi_{k,m}
			\end{array}
			\right.,
		\end{equation}
		
		For a sequence of $(\tilde{w}^t_n, \xi^n)$ that solves~\eqref{eq:singular_mf_t} with parameter $\rho^n_2\rightarrow \rho_2$ as $n\rightarrow +\infty.$ We define that $ w^t_n=\tilde{w}^t_n
		+\log\left(\int_{\Sigma} V_2e^{w^n_t-\sum_{i=1} \frac{\varrho(\xi^n_i) }{2} G^g(\cdot,\xi^n_i)} dv_g\right).$
		We define the blow-up points set as
		\[ \cB=\left\{x_0\in \Sigma:\exists x_n\rightarrow x_0  \text{ s.t. }\lim_{n\rightarrow +\infty} \tilde{w}_n^t(x_n)= +\infty\right\},\]
		and the local limit mass 
		$$\tilde{\sigma}(x)= \lim_{r\rightarrow 0} \lim_{n\rightarrow +\infty} \int_{U_r(x)} 2\rho_2^n V_2 e^{\tilde{w}^t_n -\sum_{i=1}^m\frac 1 2 \varrho(\xi_i) G^g(\cdot,\xi_i)} dv_g. $$
		{\it Step 1. If $\rho_2\notin 2\pi \N$, then $\cB=\emptyset$. Furthermore, there exist constants $C>0$ such that 
			\[ \sup_{\Sigma} \tilde{w}^t_n \leq C. \]
			
		}
		\begin{claim}\label{claim:cpt}
			If $ \tilde{\sigma}(x_0)<\frac  1 2 \varrho(x_0)$, then there exists a small open neighborhood $U_{x_0}$ around $x_0$ such that $ e^{\tilde{w}^t_n}\in L^p( U_{x_0})$ for some $p>1$ and $\tilde{w}^t_n \in L^{\infty}(U_{x_0})$. Consequently, $\tilde{\sigma}(x_0)=0$ and $x_0\notin \cB.$
		\end{claim}
		
		{\it Proof of Claim~\ref{claim:cpt}.}
		Using the isothermal coordinates,
		we define  $\tilde{v}^n(y):= \tilde{w}^t_n\circ y^{-1}_{x_0}(y)$ and $\tilde{V}_i(y):= V_i \circ y^{-1}_{x_0}(y)$, we can obtain the local version of \eqref{eq:singular_mf_t}. Specifically,
		\begin{equation}\label{eq:singular_sacal_y}
			-\Delta \tilde{v}^n(y)= 2\rho^n_2e^{\varphi_{x_0}(y)}  \left( \tilde{V}(y)e^{\tilde{v}^n(y)- \sum_{i=1}^m \frac{\varrho(\xi_i)}2 G^g( y_{x_0}^{-1}(y), \xi_i)} - 1\right)  \text{ in } B^{x_0}, 
		\end{equation} and 
		\[	\partial_{y_2} \tilde{v}^n(y)=\partial_{y_2} \tilde{v}^n(y)=0 \text{ on }
		B^{x_0}\cap \{y\in \R^2: y_2=0\}
		\] 
		When $x_0\in \partial \Si$, we have to apply the even extension for $\tilde{v}^n$ (refer to~\cite{ni_takagi1991}). 
		
		In distribution sense, for $r_0>0$ sufficiently small, $\tilde{v}^n$ solves 
		\begin{equation}\label{eq:singular_mf_local}
			-\Delta \tilde{v}^n(y)=\tilde{f}^n \text{ in } B_{r_0},
		\end{equation}
		$\tilde{f}^n=2\rho^n_2 e^{\varphi_{x_0}(y)}\left( \tilde{V}(y)e^{\tilde{v}^n(y)-\sum_{i=1}^m \frac{\varrho(\xi_i)}{2} G^g(y_{x_0}^{-1}(y),\xi_i)} - 1 \right)$ in $B_{r_0}. $ 
		For any real-valued function $f$, let $f_+:=\max\{0,f\}$.
		We decompose $(\tilde{v}^n)_+$ into two parts
		$$
		(\tilde{v}^n)_+=\tilde{v}^n_{1 }+\tilde{v}^n_{2}, \text{ in } B_r,
		$$ for any $r\in (0,r_0)$, 
		where $\tilde{v}_{1}^n$ is the solution of
		$$
		\left\{\begin{array}{ll}
			-\Delta\tilde{v}_{1 }^n=(\tilde{f}^n)_+&\text { in } B_{r} \\
			\tilde{v}_{1}^n =0& \text{ on }\partial B_{r}
		\end{array}\right.
		$$
		and $\tilde{v}^n_{2}$ is the solution of 
		$$
		\left\{\begin{array}{ll}
			-\Delta\tilde{v}_{2}^n=0&\text { in } B_{r}\\
			\tilde{v}_{2}^n =(\tilde{v}^n)_+& \text{ on }\partial B_{r}
		\end{array}\right..
		$$
		Applying the result given by Brezis and Merle in \cite[Theorem 1]{brezis_uniform_1991}  for $D=B_{r}$, we have
		$$
		\int_{B_{r}} \exp \left(\frac{(4 \pi-\epsilon) (\tilde{v}^n_{1})_+}{\left\|(\tilde{f}^n)_+\right\|_{L^1(B_{r})}}\right) \leqslant C \frac{r^2}{\epsilon}
		$$
		for some constant $C>0$, for any $r\in (0,r_0)$ and $\epsilon \in(0,4\pi)$. 
		
		Due to the assumption, for any $c_0\in (0, 4\pi)$, there exist constants $r_1>0$ and $N\in \N$ such that 
		\[ \int_{B_{r}} 2 \rho^n_2 e^{\varphi_{x_0}} \tilde{V} e^{\tilde{v}^n- \sum_{i=1}^m \frac{\varrho(\xi_i)} 2 G^g(y_{x_0}^{-1}, \xi_i)}\leq 4\pi- c_0, \forall r\in (0,r_1), n\geq N. \]
		Observe that there exists a constant $c_1>0$ such that 
		\[ \int_{B_r}  2 \rho^n_2 e^{\varphi_{x_0}} \leq c_1 \rho^n_2r^2. \]
		Since $\rho_2^n\rightarrow \rho_2$, $\rho_2^n$ is uniformly bounded.
		Combining the estimates above, we have 
		\[  \int_{B_r} |(\tilde{f}^n)_+|\leq 4\pi -c_0 + c_1 \sup_{n} |\rho^n_2|  r^2.   \]
		We take $r^*=\min\{r_0, r_1, \sqrt{\frac {c_0} {2 c_1 \sup_{n\in\N}|\rho^n_2|}} \}$ and fix a $\epsilon\in (0,4\pi)$ sufficiently small such that $4\pi-\epsilon>4\pi -\frac 1 4 c_0$.  
		For any $n \in N$ and $r\in (0, r^*),$
		$
		\int_{B_{r}} e^{p\tilde{v}^n } \leqslant C,
		$
		for some constant $C$, where $p= \frac{4\pi-\frac 1 4 c_0}{4\pi -\frac 1 2 c_0}>1$. Thus $(\tilde{f}^n)_+\in L^{p}(B_{r})$. By the $L^p$-theory and  Sobolev inequality, 
		\begin{equation*}
			\label{eq:infty_u^n_i1} \|\tilde{v}^n_{1}\|_{L^{\infty}\left( B_{\frac 12 r}\right)}\leq C,
		\end{equation*}
		for some constant $C>0.$ 
		On the other hand, the maximum principle implies that $\tilde{v}^n_{2}>0$. Given the estimate~\eqref{eq:infty_u^n_i1}, we have $\int_{B_{\frac 12 r}} \tilde{v}_{2}^{n}\leq \int_{B_{\frac 12 r}}\left((\tilde{v}^n)_++\left|\tilde{v}^n_{1}\right|\right) \leq C+\|\tilde{v}^n\|_{L^1(B_{r^*})}$. Additionally, we observe that  	$\|-\Delta_g \tilde{w}^t_n\|_{L^1(\Sigma)}$ is uniformly bounded.  Using the $L^p$-theory , we then have  $\|\tilde{w}^t_n\|_{L^q(\Sigma)}\leq C$, for any $q\in (1,2)$ and  for some constant $C>0.$  Applying  the mean-value theorem of harmonic functions, it follows that 
		$$
		\begin{array}{l}
			\left\|\tilde{v}_{2}^n \right\|_{L^{\infty}
				\left(B_{\frac 1 4 r} \right)} \leq C\left\|\tilde{v}_{2}^n\right\|_{L^1\left(B_{\frac 12 r}\right)} 
			\leq C(1+\left\|\tilde{v}^n\right\|_{L^1(B_{r^*})}) \leq C .
		\end{array}
		$$
		To conclude, we proved $\|({v}^n)_+\|_{L^{\infty}(y_{x_0}^{-1}(B_{\frac 1 4r}\cap B^{x_0})}\leq C$ for some constant $C>0$, i.e. $\tilde{w}^n_t$ is bounded from above around a small neighborhood of $x_0$. 
		So, Claim~\ref{claim:cpt} is complete. \par 
		
		Given that 
		$\int_{\Sigma} 2\rho^n_2 V_2 e^{ \tilde{w}^t_n(y)- \sum_{i=1}^m \frac{\varrho(\xi_i)}2 G^g( y_{x_0}^{-1}(y), \xi_i)}\equiv 1$, 
		Claim~\ref{claim:cpt} indicates that $\cB$ must be a finite set. Otherwise, the presence of infinitely many blow-up points in $ \cB\setminus\{ \xi_i\}_{i=1}^m$, each with a local limit mass $\tilde{\sigma}(x)\geq \frac 12 \varrho(x)$ for any $x\in \cB\setminus\{ \xi_i\}_{i=1}^m$. 
		$2\rho_2\geq \sum_{x\in \cB} \frac  1 2 \varrho(x)$, which leads to a contradiction.

		For
		\[
		-\Delta u = a(x) F(u) + a_0(x), \quad \text{in } U \subset \mathbb{R}^2,
		\]
		we have the following  Pohozaev's identity:
		\begin{eqnarray*}
			&	\int_U 2 (a F(u) + a_0 u) + \int_U (y \cdot \nabla a) F(u) + (y \cdot \nabla a_0) u  \nonumber \\
			&\quad = \int_{\partial U} (y \cdot \nabla u) \frac{\partial u}{\partial \nu} - (y \cdot \nu) \frac{|\nabla u|^2}{2} + (y \cdot \nu) (a F(u) + a_0 u),
		\end{eqnarray*}
		where \( F(u) = \int_0^u f(s) \, ds \) and $\nu$ is the unit outward norm of $\partial U$. 
		
		Applying the Pohozaev's identity,  we derive that 
		\begin{eqnarray}\label{eq:phoev}
			&&	\int_{\partial B_r} \frac{r}{2} |\nabla  \tilde{v}^n|^2 ds_r - r\int_{\partial B_r} |\partial_{\nu}\tilde{v}^n|^2 ds_r\\
			&=& r\int_{\partial B_r}2\rho^n_2 e^{\varphi_{x_0}}   e^{\tilde{v}^n-\sum_{i=1}^m \frac{\varrho(\xi_i)}{2} G^g(y_{x_0}^{-1}(y),\xi_i)}ds_r\nonumber\\
			&&- 4 \int_{B_r}\rho^n_2 e^{\varphi_{x_0}}  \left( e^{\tilde{v}^n-\sum_{i=1}^m \frac{\varrho(\xi_i)}{2} G^g(y_{x_0}^{-1}(y),\xi_i)}-{\tilde{v}^n}\right)  \, dy\nonumber\\
			&& - \int_{B_r}   2\rho^n_2 e^{\tilde{v}^n}\lan \nabla (e^{\varphi_{x_0}}\tilde{V}e^{-\sum_{i=1}^m \frac{\varrho(\xi_i)}{2} G^g(y_{x_0}^{-1}(y),\xi_i)}), y\ran  \, dy- 2\rho^n_2 \int_{B_r} e^{\varphi_{x_0}}  y\cdot \nabla \tilde{v}^n \, dy,\nonumber
		\end{eqnarray}
		where $ds_r$ is the line element on $\partial B_r$ and $ \nu $ denotes the unit outward vector on $ \partial B_r $.
		
		
		Suppose that $\cB\neq\emptyset$, we will prove that 
		\[ \int_{\Sigma} V_2 e^{w^t_n -\sum_{i=1}^m \frac{\varrho(\xi_i)}{2} G^g(\cdot,\xi_i)} dv_g \rightarrow +\infty. \]
		
		Using the Jensen's Inequality with $\int_{\Sigma} w^t_n dv_g =\int_{\Sigma} G^g(\cdot,\xi_i) dv_g=0$ for any $i=1,\cdots,m,$ we can deduce that 
		$\int_{\Sigma} V_2 e^{w^t_n -\sum_{i=1}^m \frac{\varrho(\xi_i)}{2} G^g(\cdot,\xi_i)} dv_g$ has a positive lower bound. 
		Assume that  
		\[0< c^*\leq  \int_{\Sigma} V_2 e^{w^t_n -\sum_{i=1}^m \frac{\varrho(\xi_i)}{2} G^g(\cdot,\xi_i)} dv_g\leq C^*\]
		for some constant $c^*,C^*> 0.$
		
		Up to a subsequence, $\int_{\Sigma} V_2 e^{w^t_n -\sum_{i=1}^m \frac{\varrho(\xi_i)}{2} G^g(\cdot,\xi_i)} dv_g$ is weakly convergent to some Borel measure denoted by $\mu$ in sense of measure. 
		
		For any $x_0\in \cB$, we 
		choose $\epsilon_0>0$ such that $U_{\epsilon_0}(x_0)\cap ( \{\xi_{i}\}_{i=1}^m \cup \cB \setminus \{x_0\})=\emptyset,$ where $U_{\epsilon_0}(x_0)$ is defined by the isothermal coordinates in Section~\ref{sec:pre}. Since $\tilde{w}^t_n\in L^{\infty}_{loc}(\Sigma\setminus \cB)$, there exists a constant $c'>0$ such that $\sup _{\partial U_{\epsilon_0}(x_0)}|\tilde{w}^n_t|\leq c'$. 
		For $x_0\in \partial\Sigma$, we apply the even extension to extend the functions on an open disk. We consider the following equations:
		\begin{equation}
			\left\{\begin{array}{ll}
				-\Delta \tilde{h}^n= 2\rho^n_2 e^{\varphi_{x_0}} \left( {\tilde{V} e^{\tilde{v}^n- \sum_{j=1}^m  \frac {\varrho(\xi_j)}2 G^g(y_{x_0}^{-1}(y),\xi_j)}} - 1\right) & \text{ in } B_{\epsilon_0}\\
				\tilde{h}^n= -c' & \text{ on } \partial B_{\epsilon_0}
			\end{array}\right.
		\end{equation}
		By the maximum principle, we deduce that 
		$\tilde{v}^n\geq \tilde{h}^n$ in $B_{\epsilon_0}$.
		By the regularity theory, 
		$ \tilde{h}^n$ converges to some $\tilde{h}$ weakly in $W^{1,q}(B_{\epsilon_0})$ for $q\in (1,2)$. It follows that 
		$\tilde{h}$ solves 
		\[ 	\left\{\begin{array}{ll}
			-\Delta \tilde{h}= \mu & \text{ in } B_{\epsilon_0}\\
			\tilde{h}= -c' & \text{ on } \partial B_{\epsilon_0}
		\end{array}\right.\]
		in the sense of distribution. 
		We consider the Green's function defined by the following equations
		\[ 	\left\{\begin{array}{ll}
			-\Delta G^*(x,x_0) = 4\pi(\delta_0-1), & \text{ in } B_{\epsilon_0}\\
			G^*(x,x_0)=-c,  & \text{ on } \partial B_{\epsilon_0} 
		\end{array}\right..
		\]
		By the maximum principal and Claim~\ref{claim:cpt}, we derive that $h\geq G^*(y,x_0)$. Observe that 
		$$G^*(y,x_0)= 2 \log |y|+ A_{x_0}(y),$$ 
		where $A_{x_0}$ is a smooth function depending on $x_0.$ 
		Fatou's lemma yields that 
		\[  
		+\infty = \int_{B_{\varepsilon_0}} G^*(\cdot,x_0) \leq \int_{B_{\epsilon_0}} e^h    \leq \liminf_{n\rightarrow +\infty} \int_{B_{\epsilon_0}} e^{h_n}     \leq \liminf_{n\rightarrow +\infty} \int_{B_{\epsilon_0}} e^{\tilde{w}^t_n}\leq C \cdot C^*. 
		\]
		There is a contradiction.
		So $\int_{\Sigma} V_2 e^{w^t_n -\sum_{i=1}^m \frac{\varrho(\xi_i)}{2} G^g(\cdot,\xi_i)} dv_g \rightarrow +\infty.$ Then $\tilde{w}^t_n \rightarrow -\infty$ uniformly in any compact subset of $\Sigma\setminus\cB.$
		
		We call the definition of $U_{\epsilon_0}(x_0)= y^{-1}_{x_0}(B^{x_0}_{\epsilon_0})$.
		By~\eqref{eq:singular_sacal_y}, we have 
		\begin{eqnarray*}
			-\Delta \tilde{v}^n(y)= 2\rho^n_2 \tilde{V}^*(y)  |y|^{2\gamma(x_0)}e^{\tilde{v}^n(y)} - 2\rho^n_2, \text{ in } B_{\epsilon_0}
		\end{eqnarray*}
		where $V^*\in C^{\infty}(B_{\epsilon_0},\R_+)$ with a lower bound away from $0$, and  $\gamma(x_0)=\left\{ \begin{array}{ll}
			0 &  x_0\notin\{\xi_i\}_{i=1}^m\\ 1 & x_0\in \{ \xi_i\}_{i=1}^m
		\end{array}\right.$.
		
		Since $\int_{\Sigma} \left|{V_2 e^{ \tilde{w}^t_n- \sum_{j=1}^m  \frac {\varrho(\xi_j)}2 G^g(x,\xi_j)}} - 1 \right|\leq 2 $, by the regularity theory, for any $q\in (1,2)$, we derive that 
		$$\|w^t_n\|_{W^{1,q}(\Sigma)}\leq C$$ for some constant $C>0.$  For $2\rho_2= \sum_{x\in\cB} \tilde{\sigma}(x)$, 
		\[ w^t_n \rightarrow \sum_{x\in \cB} \tilde{\sigma}(x) G^g(\cdot, x), \]
		weakly in $ W^{1,q}(\Sigma)$ and strongly in $C^2_{loc}(\Sigma\setminus\cB)$, where $G^g$ is the Green's function defined by~\eqref{eq:green}.
		We observe that  for any $x\in y_{x_0}^{-1}(B_{\epsilon_0}^{x_0})$, 
		\[ -\Delta_g G^g(x,x_0) = \delta_{x_0}- 1\]
		and $G^g(x,x_0) = \frac{4}{\varrho(x_0)} \log \frac 1 { |y_{x_0}(x)|}+ H^g(x,x_0)$, where $H^g(\cdot,x_0)\in C^2(B_{\epsilon_0}^{x_0})$.
		Passing the limit $n\rightarrow+\infty$ firstly and the $r\rightarrow 0^+$ for~\eqref{eq:phoev}, we have 
		\[ \tilde{\sigma}(x_0)^2=\varrho(x_0)(1+\gamma(x_0)) \tilde{\sigma}(x_0). \]
		Since $x_0\in \cB$, $\tilde{\sigma}(x_0)>0$. It follows $$\tilde{\sigma}(x_0)= \varrho(x_0)=\left\{\begin{array}{ll}
			8\pi(1+\gamma(x_0)) & \text{ if } x_0\in\intsigma\\
			4\pi(1+\gamma(x_0))& \text{ if } x_0\in \partial\Sigma
		\end{array}\right.. $$  Since $\gamma(x_0)=0$ or $1$, 
		$2\rho_2=\sum_{x\in \cB} \varrho(x)\in 2\pi \N$, which contradicts  the assumption that $\rho_2\notin 2\pi \N$.

		\text{\it Step 2. For any $t\in [0,1],$ $(w^t,\xi^t)$ a solution to~\eqref{eq:shadow_t},  $\| w^t\|_{C^1(\Sigma
				)}\leq C,$  for some constant  $C>0$. }\par
		Using the Green's representation formula, 
		\begin{eqnarray*}
			w^t(x)=  2\rho_2\int_{\Sigma} G^g(\cdot,x) \frac{V_2 e^{w^t-\sum_{i=1}^m \frac{\varrho(\xi^t_i)} 2 G^g(\cdot,\xi_i^t)}}{\int_{\Sigma} V_2 e^{w^t-\sum_{i=1}^m \frac{\varrho(\xi^t_i)} 2 G^g(\cdot,\xi_i^t) }dv_g} dv_g .
		\end{eqnarray*}
		{\it Step 1} implies that $\sup_{\Sigma} (w^t-\log( {\int_{\Sigma} V_2 e^{w^t-\sum_{i=1}^m \frac{\varrho(\xi^t_i)} 2 G^g(\cdot,\xi_i^t) }dv_g}) )$ is uniformly bounded from above. By the regularity theory, it follows that $\| w^t\|_{C^{2,\alpha}(\Si
			)}\leq C$, for any $\alpha\in (0,1)$ and some constant  $C>0$.
	\end{proof}

	\bibliographystyle{plain} 
	\bibliography{Toda_Blow_2024} 

\begin{thebibliography}{10}

\bibitem{Ahmedou-Bartsch-Fiernkranz:2023}
Mohameden Ahmedou, Thomas Bartsch, and Tim Fiernkranz.
\newblock Equilibria of vortex type {H}amiltonians on closed surfaces.
\newblock {\em Topol. Methods Nonlinear Anal.}, 61(1):239--256, 2023.

\bibitem{ao2014new}
Weiwei Ao and Liping Wang.
\newblock New concentration phenomena for {S}{U}(3) {T}oda system.
\newblock {\em J. Differential Equations}, 256:1548--1580, 2014.

\bibitem{baraket_construction_1997}
Sami Baraket and Frank Pacard.
\newblock Construction of singular limits for a semilinear elliptic equation in dimension 2.
\newblock {\em Calc. Var. Partial Differ. Equ.}, 6(1):1--38, 1997.

\bibitem{Bartolucci2020}
Daniele Bartolucci, Changfeng Gui, Yeyao Hu, Aleks Jevnikar, and Wen Yang.
\newblock Mean field equations on tori: existence and uniqueness of evenly symmetric blow-up solutions.
\newblock {\em Discrete Contin. Dyn. Syst.}, 40(6):3093--3116, 2020.

\bibitem{Bartsch2017TheMP}
Thomas Bartsch, Anna~Maria Micheletti, and Angela Pistoia.
\newblock The morse property for functions of {K}irchhoff-{R}outh path type.
\newblock {\em Discrete Contin. Dyn. Syst. - S}, 12(7):1867--1877, 2019.

\bibitem{battaglia_jevnikar_ruiz2015}
Luca Battaglia, Aleks Jevnikar, Andrea Malchiodi, and David Ruiz.
\newblock A general existence result for the {T}oda system on compact surfaces.
\newblock {\em Adv. Math.}, 285:937--979, 2015.

\bibitem{bolton1988conformal}
John Bolton, Gary~R. Jensen, Marco Rigoli, and Lyndon~M. Woodward.
\newblock On conformal minimal immersions of {$S^2$} into \( \mathbb{C}{P}^n \).
\newblock {\em Math. Ann.}, 279(4):599--620, 1988.

\bibitem{bolton1997geometrical}
John. Bolton and Lyndon~M. Woodward.
\newblock Some geometrical aspects of the 2-dimensional {T}oda equations.
\newblock In {\em Geometry, Topology and Physics}, pages 69--81, Berlin, 1997. Campinas, de Gruyter.

\bibitem{brezis_uniform_1991}
Ha\"{i}m Brezis and Frank Merle.
\newblock Uniform estimates and blow-up behavior for solutions of { $ -{\Delta} u={V(x)} e^u$} in two dimensions.
\newblock {\em Comm. Partial Differential Equations}, 16(8-9):1223--1253, 1991.

\bibitem{chern1955}
Shiing-shen Chern.
\newblock An elementary proof of the existence of isothermal parameters on a surface.
\newblock {\em Proc. Amer. Math. Soc.}, 6:771--782, 1955.

\bibitem{chern1987harmonic}
Shiing-Shen Chern and Jon~G. Wolfson.
\newblock Harmonic maps of the two-sphere into a complex {G}rassmann manifold. {I}{I}.
\newblock {\em Ann. of Math.}, 125(2):301--335, 1987.

\bibitem{DAprile_Pistoia_Ruiz2015}
Teresa D'Aprile, Angela Pistoia, and David Ruiz.
\newblock A continuum of solutions for the {S}{U}(3) {T}oda system exhibiting partial blow-up.
\newblock {\em Proc. Lond. Math. Soc.}, 111(4):797--830, 2015.

\bibitem{D'aprile_asymmetric_2016}
Teresa D'Aprile, Angela Pistoia, and David Ruiz.
\newblock Asymmetric blow-up for the {S}{U}(3) {T}oda system.
\newblock {\em J. Funct. Anal.}, 271(3):495--531, 2016.

\bibitem{DelPino2012Nondegeneracy}
M.~Del~Pino, P.~Esposito, and M.~Musso.
\newblock Nondegeneracy of entire solutions of a singular {L}iouville equation.
\newblock {\em Proc. Amer. Math. Soc.}, 140:581--588, 2012.

\bibitem{del_pino_singular_2005}
Manuel del Pino, Michal Kowalczyk, and Monica Musso.
\newblock Singular limits in {Liouville}-type equations.
\newblock {\em Calc. Var. Partial Differential Equations}, 24(1):47--81, 2005.

\bibitem{doliwa1997holomorphic}
Adam Doliwa.
\newblock Holomorphic curves and toda systems.
\newblock {\em Lett. Math. Phys.}, 39(1):21--32, 1997.

\bibitem{dunne1995self}
Gerald~V. Dunne.
\newblock {\em Self-dual {C}hern-{S}imons theories}.
\newblock Lecture Notes in Physics. Springer, Berlin, 1995.

\bibitem{dunne1991self}
Gerald~V. Dunne, R.~Jackiw, So-Young Pi, and Carlo~A. Trugenberger.
\newblock Self-dual {Chern}-{Simons} solitons and two-dimensional nonlinear equations.
\newblock {\em Phys. Rev. D}, 43:1332--1345, 1991.

\bibitem{Esposito2014singular}
Pierpaolo Esposito and Pablo Figueroa.
\newblock Singular mean field equations on compact {R}iemann surfaces.
\newblock {\em Nonlinear Anal.}, 111:33--65, 2014.

\bibitem{Esposito2005}
Pierpaolo Esposito, Massimo Grossi, and Angela Pistoia.
\newblock On the existence of blowing-up solutions for a mean field equation.
\newblock {\em Ann. Inst. H. Poincaré Anal. Non Linéaire.}, 22(2):227--257, 2005.

\bibitem{guest1997harmonic}
Martin~A. Guest.
\newblock {\em Harmonic maps, loop groups, and integrable systems}, volume~38 of {\em London Mathematical Society Student Texts}.
\newblock Cambridge University Press, Cambridge, 1997.

\bibitem{henry2005perturbation}
Daniel Henry.
\newblock {\em Perturbation of the Boundary in Boundary-Value Problems of Partial Differential Equations}.
\newblock London Mathematical Society Lecture Note Series. Cambridge University Press, Cambridge, 2005.

\bibitem{Hu2024}
Zhengni Hu.
\newblock On solutions for singular toda system on riemann surfaces with boundary, 2024.
\newblock Preprint on arXiv.

\bibitem{BartschHuSubmitted}
Zhengni Hu and Thomas Bartsch.
\newblock The {Morse} property of limit functions appearing in mean field equations on surfaces with boundary.
\newblock {\em J Geom Anal}, 34(7):220, 2024.

\bibitem{HBA2024}
Zhengni Hu, Thomas Bartsch, and Mohameden Ahmedou.
\newblock Blow-up solutions for mean field equations with neumann boundary conditions on {R}iemann surfaces, 2024.
\newblock Preprint on arXiv.

\bibitem{jevnikar_kallel_Malchiodi2015}
Aleks Jevnikar, Sadok Kallel, and Andrea Malchiodi.
\newblock A topological join construction and the {T}oda system on compact surfaces of arbitrary genus.
\newblock {\em Anal. PDE}, 8(8):1963--2027, 2015.

\bibitem{jost_lin_wang2006}
Jürgen Jost, Changshou Lin, and Guofang Wang.
\newblock Analytic aspects of the {{Toda}} system: {I}{I.} bubbling behavior and existence of solutions.
\newblock {\em Comm. Pure Appl. Math.}, 58(4):526--558, 2005.

\bibitem{Lee2018degree}
Youngae Lee, Chang-Shou Lin, Juncheng Wei, and Wen Yang.
\newblock Degree counting and shadow system for {T}oda system of rank two: one bubbling.
\newblock {\em J. Differential Equations}, 264(7):4343--4401, 2018.

\bibitem{lee_degree_2018}
Youngae Lee, Chang-Shou Lin, Juncheng Wei, and Wen Yang.
\newblock Degree counting and {Shadow} system for {Toda} system of rank two: {One} bubbling.
\newblock {\em J. Differential Equations}, 264(7):4343--4401, 2018.

\bibitem{li2005solutions}
Jiayu Li and Yuxiang Li.
\newblock Solutions for {T}oda systems on {R}iemann surfaces.
\newblock {\em Ann. Sc. Norm. Super. Pisa Cl. Sci.}, 4:703--728, 2005.

\bibitem{Li1997OnAS}
Yanyan Li.
\newblock On a singularly perturbed elliptic equation.
\newblock {\em Adv. Differ. Equ.}, 2:955--980, 1997.

\bibitem{lin2012classification}
Chang-Shou Lin, Juncheng Wei, and Dong Ye.
\newblock Classification and nondegeneracy of {S}{U}({N}+1) {T}oda system with singular sources.
\newblock {\em Invent. math.}, 190:169--207, 2012.

\bibitem{lin2012sharp}
Chang-Shou Lin, Juncheng Wei, and Chunyi Zao.
\newblock Sharp estimates for fully bubbling solutions of a {S}{U}(3) {T}oda system.
\newblock {\em Geom. Funct. Anal.}, 22:1591--1635, 2012.

\bibitem{lin1941motion}
Chia-Ch'iao Lin.
\newblock On the motion of vortices in two dimensions. {I}. existence of the {K}irchhoff-{R}outh function.
\newblock {\em Proc. Natl. Acad. Sci.}, 27(11):570--575, 1941.

\bibitem{Malchiodi2007SomeER}
Andrea Malchiodi and Cheikh~Birahim Ndiaye.
\newblock Some existence results for the {T}oda system on closed surfaces.
\newblock {\em Rend. Lincei Mat. Appl.}, 18:391--412, 2007.

\bibitem{Malchiodi2010NewIM}
Andrea Malchiodi and David Ruiz.
\newblock New improved {M}oser--{T}rudinger inequalities and singular {L}iouville equations on compact surfaces.
\newblock {\em Geom. Funct. Anal.}, 21:1196--1217, 2010.

\bibitem{musso_new_2016}
Monica Musso, Angela Pistoia, and Juncheng Wei.
\newblock New blow-up phenomena for {S}{U}({N}+ 1) {Toda} system.
\newblock {\em J. Differential Equations}, 260(7):6232--6266, 2016.

\bibitem{Nardi2014}
Giacomo Nardi.
\newblock Schauder estimate for solutions of {P}oisson's equation with {N}eumann boundary condition.
\newblock {\em L'Enseign. Math.}, 60(3):421--435, 2015.

\bibitem{ni_takagi1991}
Wei-Ming Ni and Izumi Takagi.
\newblock On the shape of least-energy solutions to a semilinear {Neumann} problem.
\newblock {\em Comm. Pure Appl. Math.}, 44(7):819--851, 1991.

\bibitem{nolasco1999double}
Margherita Nolasco and Gabriella Tarantello.
\newblock Double vortex condensates in the {C}hern-{S}imons theory.
\newblock {\em Calc. Var. Partial Differ. Equ.}, 9:31--94, 1999.

\bibitem{nolasco2000vortex}
Margherita Nolasco and Gabriella Tarantello.
\newblock Vortex condensates for the {S}{U}(3) {C}hern-{S}imons theory.
\newblock {\em Comm. Math. Phys.}, 213(3):599--639, 2000.

\bibitem{yang1999relativistic}
Yisong Yang.
\newblock The relativistic non-abelian {C}hern-{S}imons equation.
\newblock {\em Comm. Math. Phys.}, 186(1):199--218, 1999.

\bibitem{yang2001solitons}
Yisong Yang.
\newblock {\em Solitons in Field Theory and Nonlinear Analysis}.
\newblock Springer Monographs in Mathematics. Springer, New York, 2001.

\bibitem{yang2006extremal}
Yunyan Yang.
\newblock Extremal functions for {M}oser-{T}rudinger inequalities on 2-dimensional compact {R}iemannian manifolds with boundary.
\newblock {\em Int. J. Math.}, 17(3):313--330, 2006.

\bibitem{yang2021125440}
Yunyan Yang and Jie Zhou.
\newblock Blow-up analysis involving isothermal coordinates on the boundary of compact {R}iemann surface.
\newblock {\em J. Math. Anal. Appl.}, 504(2):125440, 2021.

\bibitem{zhu2011solutions}
Xiao~Bao Zhu.
\newblock Solutions for {T}oda system on {R}iemann surface with boundary.
\newblock {\em Acta Math. Sin.-English Ser.}, 27(8):1501--1520, 2011.

\end{thebibliography}
		\vspace{2mm}\noindent
	{\sc Thomas Bartsch, Mohameden Ahmedou, Zhengni Hu}\\
	Mathematisches Institut\\
	Universit\"at Giessen\\
	Arndtstr.\ 2\\
	35392 Giessen, Germany\\
	\href{mailto:Thomas.Bartsch@math.uni-giessen.de}{Thomas.Bartsch@math.uni-giessen.de}
	\\
	\href{mailto:Mohameden.Ahmedou@math.uni-giessen.de}{Mohameden.Ahmedou@math.uni-giessen.de}
	\\
	\href{mailto:Zhengni.Hu@math.uni-giessen.de}{Zhengni.Hu@math.uni-giessen.de}
	\\
\end{document}